\documentclass[11pt]{amsart}

\usepackage{marginnote}
\usepackage{anysize} 
\usepackage{comment}
\marginsize{1in}{1in}{1in}{1in}
\usepackage{comment}
\usepackage{listings}
\usepackage{xcolor}
\usepackage{amsmath}
\usepackage{mathtools}
\usepackage{booktabs}
\usepackage[all]{xy}
\usepackage[utf8]{inputenc}
\usepackage{varioref}
\usepackage{upgreek}
\usepackage{amsfonts}
\usepackage{amssymb}
\usepackage{bbm}
\usepackage{esint}
\usepackage{graphicx}
\usepackage{subcaption}
\usepackage{tikz}
\usepackage{empheq}
\usepackage{enumitem}
\usepackage{tikz-cd}
\usepackage{upgreek}
\usepackage{todonotes}
\usetikzlibrary{matrix,arrows,decorations.pathmorphing}
\usepackage{mathrsfs}
\usepackage[hypertexnames=false,backref=page,pdftex,
 	pdfpagemode=UseNone,
 	breaklinks=true,
 	extension=pdf,
 	colorlinks=true,
 	linkcolor=blue,
 	citecolor=red,
 	urlcolor=blue,
 ]{hyperref}

\def\sp{\operatorname{sp}}

\newcommand{\HH}{{\mathbb H}}

\newcommand{\tr}{\operatorname{tr}}

\newcommand{\str}{\operatorname{str}}

\newcommand{\id}{\operatorname{id}}

\newcommand{\End}{\operatorname{End}}
\newcommand{\ad}{\operatorname{ad}}

\newcommand{\Dirac}{\operatorname{D}}

\def\EV{{\text{EV}}}
\def\sp{{\mathfrak{sp}}}
\def\wt{{\mathrm{wt}}}
\def\WW{{\mathcal{W}}}

\def\so{{\mathfrak{so}}}

\def\CC{{\mathbb C}}
\def\RR{{\mathbb R}}

\def\su{{\mathfrak{su}}} 

\def\psl{{\mathfrak{psl}}}

\def\sl{{\mathfrak{sl}}}
\def\st{{\text{st}}}
\def\nst{{\text{nst}}}
\def\self{{\text{self}}}
\def\skew{{\text{skew}}}

\def\HH{{\mathcal{H}}}

\def\osp{{\mathfrak{osp}}}

\def\hh{{\mathfrak h}}

\def\kk{{\mathfrak k}}

\def\bb{{\mathfrak b}}
\def\even{{\mathfrak{g}_{\bar{0}}}}
\def\odd{{\mathfrak{g}_{\bar{1}}}}

\def\nn {{\mathfrak{n}}}

\def\gg{{\mathfrak{g}}}

\def\ZZ{{\mathbb Z}}

\def\slmn{{\mathfrak{sl}(m\vert n)}}

\def\glmn{{\mathfrak{gl}(m\vert n)}} 
\def\UE{{\mathfrak U}}

\newcommand{\gl}{{\mathfrak{gl}}}

\makeatletter
\newcommand*{\da@rightarrow}{\mathchar"0\hexnumber@\symAMSa 4B }
\newcommand*{\da@leftarrow}{\mathchar"0\hexnumber@\symAMSa 4C }
\newcommand*{\xdashrightarrow}[2][]{%
  \mathrel{%
    \mathpalette{\da@xarrow{#1}{#2}{}\da@rightarrow{\,}{}}{}%
  }%
}
\newcommand{\xdashleftarrow}[2][]{%
  \mathrel{%
    \mathpalette{\da@xarrow{#1}{#2}\da@leftarrow{}{}{\,}}{}%
  }%
}
\newcommand*{\da@xarrow}[7]{%
  \sbox0{$\ifx#7\scriptstyle\scriptscriptstyle\else\scriptstyle\fi#5#1#6\m@th$}%
  \sbox2{$\ifx#7\scriptstyle\scriptscriptstyle\else\scriptstyle\fi#5#2#6\m@th$}%
  \sbox4{$#7\dabar@\m@th$}%
  \dimen@=\wd0 %
  \ifdim\wd2 >\dimen@
    \dimen@=\wd2 %
  \fi
  \count@=2 %
  \def\da@bars{\dabar@\dabar@}%
  \@whiledim\count@\wd4<\dimen@\do{%
    \advance\count@\@ne
    \expandafter\def\expandafter\da@bars\expandafter{%
      \da@bars
      \dabar@ 
    }%
  }%
  \mathrel{#3}%
  \mathrel{%
    \mathop{\da@bars}\limits
    \ifx\\#1\\%
    \else
      _{\copy0}%
    \fi
    \ifx\\#2\\%
    \else
      ^{\copy2}%
    \fi
  }%
  \mathrel{#4}%
}
\makeatother

\makeatletter
\newsavebox\myboxA
\newsavebox\myboxB
\newlength\mylenA

\newcommand*\xtilde[2][0.8]{%
    \sbox{\myboxA}{$\m@th#2$}%
    \setbox\myboxB\null
    \ht\myboxB=\ht\myboxA%
    \dp\myboxB=\dp\myboxA%
    \wd\myboxB=#1\wd\myboxA
    \sbox\myboxB{$\m@th\widetilde{\copy\myboxB}$}
    \setlength\mylenA{\the\wd\myboxA}
    \addtolength\mylenA{-\the\wd\myboxB}%
    \ifdim\wd\myboxB<\wd\myboxA%
       \rlap{\hskip 0.5\mylenA\usebox\myboxB}{\usebox\myboxA}%
    \else
        \hskip -0.5\mylenA\rlap{\usebox\myboxA}{\hskip 0.5\mylenA\usebox\myboxB}%
    \fi}

\newbox\usefulbox

\def\getslant #1{\strip@pt\fontdimen1 #1}

\def\xxtilde #1{\mathchoice
 {{\setbox\usefulbox=\hbox{$\m@th\displaystyle #1$}%
    \dimen@ \getslant\the\textfont\symletters \ht\usefulbox
    \divide\dimen@ \tw@ 
    \kern\dimen@ 
    \xtilde{\kern-\dimen@ \box\usefulbox\kern\dimen@ }\kern-\dimen@ }}
 {{\setbox\usefulbox=\hbox{$\m@th\textstyle #1$}%
    \dimen@ \getslant\the\textfont\symletters \ht\usefulbox
    \divide\dimen@ \tw@ 
    \kern\dimen@ 
    \xtilde{\kern-\dimen@ \box\usefulbox\kern\dimen@ }\kern-\dimen@ }}
 {{\setbox\usefulbox=\hbox{$\m@th\scriptstyle #1$}%
    \dimen@ \getslant\the\scriptfont\symletters \ht\usefulbox
    \divide\dimen@ \tw@ 
    \kern\dimen@ 
    \xtilde{\kern-\dimen@ \box\usefulbox\kern\dimen@ }\kern-\dimen@ }}
 {{\setbox\usefulbox=\hbox{$\m@th\scriptscriptstyle #1$}%
    \dimen@ \getslant\the\scriptscriptfont\symletters \ht\usefulbox
    \divide\dimen@ \tw@ 
    \kern\dimen@ 
    \xtilde{\kern-\dimen@ \box\usefulbox\kern\dimen@ }\kern-\dimen@ }}%
 {}}

\newcommand*\xoverline[2][0.75]{%
    \sbox{\myboxA}{$\m@th#2$}%
    \setbox\myboxB\null
    \ht\myboxB=\ht\myboxA%
    \dp\myboxB=\dp\myboxA%
    \wd\myboxB=#1\wd\myboxA
    \sbox\myboxB{$\m@th\overline{\copy\myboxB}$}
    \setlength\mylenA{\the\wd\myboxA}
    \addtolength\mylenA{-\the\wd\myboxB}%
    \ifdim\wd\myboxB<\wd\myboxA%
       \rlap{\hskip 0.5\mylenA\usebox\myboxB}{\usebox\myboxA}%
    \else
        \hskip -0.5\mylenA\rlap{\usebox\myboxA}{\hskip 0.5\mylenA\usebox\myboxB}%
    \fi}

\def\xxoverline #1{\mathchoice
 {{\setbox\usefulbox=\hbox{$\m@th\displaystyle #1$}%
    \dimen@ \getslant\the\textfont\symletters \ht\usefulbox
    \divide\dimen@ \tw@ 
    \kern\dimen@ 
    \overline{\kern-\dimen@ \box\usefulbox\kern\dimen@ }\kern-\dimen@ }}
 {{\setbox\usefulbox=\hbox{$\m@th\textstyle #1$}%
    \dimen@ \getslant\the\textfont\symletters \ht\usefulbox
    \divide\dimen@ \tw@ 
    \kern\dimen@ 
    \xoverline{\kern-\dimen@ \box\usefulbox\kern\dimen@ }\kern-\dimen@ }}
 {{\setbox\usefulbox=\hbox{$\m@th\scriptstyle #1$}%
    \dimen@ \getslant\the\scriptfont\symletters \ht\usefulbox
    \divide\dimen@ \tw@ 
    \kern\dimen@ 
    \xoverline{\kern-\dimen@ \box\usefulbox\kern\dimen@ }\kern-\dimen@ }}
 {{\setbox\usefulbox=\hbox{$\m@th\scriptscriptstyle #1$}%
    \dimen@ \getslant\the\scriptscriptfont\symletters \ht\usefulbox
    \divide\dimen@ \tw@ 
    \kern\dimen@ 
    \xoverline{\kern-\dimen@ \box\usefulbox\kern\dimen@ }\kern-\dimen@ }}%
 {}}
\makeatother

\makeatletter
\newcommand{\mylabel}[2]{#2\def\@currentlabel{#2}\label{#1}}
\makeatother

\makeatletter
\newcommand{\Mac}{}
\DeclareRobustCommand{\Mac}{%
  M%
  \raisebox{\dimexpr\fontcharht\font`M-\height}{%
    \check@mathfonts\fontsize{\sf@size}{0}\selectfont
    c%
  }%
}
\makeatother

\newtheoremstyle{citing}
  {}
  {}
  {\itshape}
  {}
  {\bfseries}
  {\textbf{.}}
  {.5em}
  {\thmnote{#3}}

\theoremstyle{plain}
\newtheorem{theorem}{Theorem}

\newtheorem{lemma}[theorem]{Lemma}
\newtheorem{corollary}[theorem]{Corollary}

\newtheorem{proposition}[theorem]{Proposition}

\theoremstyle{remark}
\newtheorem{example}[theorem]{Example}

\theoremstyle{definition}

\newtheorem{definition}[theorem]{Definition}

\numberwithin{equation}{section}

\theoremstyle{remark}
\newtheorem{remark}[theorem]{Remark}

{\theoremstyle{citing}
}

{\theoremstyle{definition}
}

\title{Unitary Branching for $\slmn, \osp(m\vert 2n)$ and $F(4)$}

\author{Steffen Schmidt}
\address{Center for Quantum Mathematics, 
University of Southern Denmark, DK-5230 Odense, Denmark}
\email{stschmidt@imada.sdu.dk}

\let\origmaketitle\maketitle
\def\maketitle{
  \begingroup
  \def\uppercasenonmath##1{} 
  \let\MakeUppercase\relax 
  \origmaketitle
  \endgroup
}

\begin{document}
\thispagestyle{empty}

\begin{abstract}
We determine the branching of irreducible unitary representations of the basic classical Lie superalgebras $\mathfrak{sl}(m|n)$, $\mathfrak{osp}(m|2n)$, and $F(4)$ under restriction to their even subalgebras. The proof combines Dirac inequalities for the Huang--Pand\v{z}i\'{c}--Dirac operator $\Dirac_{\gg,\even}$ with a super analogue of Freudenthal's recursion derived from the Dirac formalism. We extend the Dirac method to positive systems that are not adapted to the relevant real form by transporting the Dirac inequality through odd reflections and determining the resulting correction terms. This yields finite recursive procedures that determine the occurring even constituents and their branching multiplicities, together with closed formulas for the latter, uniformly in the degree of atypicality. For $\mathfrak{sl}(m|n)$, the corrected Dirac inequalities also give a characterization of unitarity for all positive systems relevant to unitary representations.
\end{abstract}

\maketitle

\setlength{\parindent}{1em}
\setcounter{tocdepth}{2}

\tableofcontents

\section{Introduction}
\noindent Unitary representation theory of Lie superalgebras is fundamental to supersymmetric physics, where states are organized into irreducible unitary representations, called supermultiplets, of the relevant supersymmetry algebra. Of particular interest are the irreducible unitary representations of the basic classical Lie superalgebras that emerge as complex superconformal algebras. In spacetime dimension greater than two, these occur only in dimensions three through six \cite{Shnider}: in dimension three they are the orthosymplectic Lie superalgebras $\osp(m\vert4)$, in dimension four the special linear Lie superalgebras $\sl(4\vert n)$, with $\psl(4\vert4)$ in the exceptional case $n=4$, in dimension five the exceptional Lie superalgebra $F(4)$, and in dimension six the orthosymplectic Lie superalgebras $\osp(8\vert2n)$.

These superconformal examples form part of the larger families $\sl(m\vert n)$, $\osp(m\vert2n)$, and $F(4)$ considered in this article. The broader families are likewise relevant in physics. For example,  $\sl(2\vert1)$ appears in the supersymmetric $t$--$J$ model of strongly correlated electrons \cite{EsslerKorepin}, while $\osp(m\vert2n)$ occurs in integrable supersphere sigma models and Gross--Neveu models \cite{SaleurWehefritzKaufmann}, as well as in integrable superspin chains \cite{ArnaudonEtAl}. Throughout, $\gg$ denotes one of the Lie superalgebras $\sl(m\vert n)$, $\osp(m\vert2n)$, or $F(4)$.

General unitary representations of $\gg$ are defined with respect to the choice of a real form. These real forms are in one-to-one correspondence with conjugate-linear anti-involutions \cite{ParkerRealForms, serganova1983classification}, that is, conjugate-linear parity-preserving maps $\omega\colon\gg\to\gg$ satisfying $\omega^{2}=\operatorname{id}_{\gg}$ and 
$
\omega([X,Y])=[\omega(Y),\omega(X)]$
for all $X,Y\in\gg$, the corresponding real form being
\begin{equation}
\gg^{\omega}=\{X\in\gg:\omega(X)=-(-i)^{p(X)}X\}.
\end{equation}
A representation $\mathcal{H} = \HH_{\bar{0}}\oplus \HH_{\bar{1}}$ of $\gg$ is then called $\omega$-unitary if it carries a positive definite Hermitian form $\langle\cdot,\cdot\rangle_{\mathcal{H}}$ such that $\langle \HH_{\bar{0}},\HH_{\bar{1}}\rangle_{\HH}=0$ and 
\begin{equation}
\langle Xv,w\rangle_{\mathcal{H}}
=
\langle v,\omega(X)w\rangle_{\mathcal{H}}
\end{equation}
for all $X\in\gg$ and $v,w\in\mathcal{H}$ \cite[Definition~2.3]{jakobsen1994full}. This is equivalent to an alternative formulation in terms of super Hermitian forms \cite{SchmidtDirac, SchmidtWalcher}. 
The unitary representation theory is concentrated in a small class of real forms. Nontrivial $\omega$-unitary representations exist only for $\su(p,q\vert n,0)\cong\su(p,q\vert0,n),$ $\osp(m,0\vert2n;\RR)\cong\osp(0,m\vert2n;\RR),$ and for the real forms $F(4,0)$ and $F(4,2)$. Moreover, every irreducible unitary representation for these real forms is either a highest weight or a lowest weight representation. We therefore restrict to unitary highest weight representations of $\sl(m\vert n)$ corresponding to $\su(p,q\vert0,n)$, with $p+q=m$, of $\osp(m\vert2n;\RR)$ corresponding to $\osp(m,0\vert2n;\RR)$, and of $F(4)$ corresponding to $F(4,0)$ and $F(4,2)$. We suppress $\omega$ from the notation and refer to them simply as unitary representations. The lowest weight case is completely analogous.

Every unitary representation of $\gg$ is completely reducible, so it is enough to study irreducible unitary highest weight modules
$
\HH\cong L(\Lambda).
$
Their restriction to the even subalgebra $\even$ is again completely reducible and unitary with respect to the corresponding real form of $\even$. Thus, up to parity,
\begin{equation}
\HH\big\vert_{\even}
\cong
\bigoplus_{\mu\in\Gamma_L(\Lambda)}
L_0(\mu)^{\oplus m_\mu},
\end{equation}
where $m_\mu\in\ZZ_{+}$ is the multiplicity of the irreducible $\even$-module $L_0(\mu)$. The resulting problem is to determine which highest weights $\mu$ occur and with what multiplicities. Equivalently, one seeks to determine the set $\Gamma_L(\Lambda)$ together with the integers $m_\mu$. This is the unitary branching problem for the pair $(\gg,\even)$.

An explicit branching law reduces the study of $\HH$ to the representation theory of the reductive Lie algebra $\even$ by determining the bosonic symmetry types occurring in $\HH$ and their multiplicities. This decomposition also provides the basis for assigning a generalized superdimension to any unitary representation of $\gg$, which will be developed in upcoming work. In superconformal applications, the same branching law decomposes a superconformal multiplet into conformal multiplets and determines its conformal-primary content together with the corresponding $R$-symmetry quantum numbers. At the same time, it makes the effect of atypicality explicit: shortening changes both the $\even$-types that occur and their multiplicities, and is therefore directly visible in the branching law. Determining this decomposition is nevertheless subtle, since the branching behavior depends sensitively on the atypicality of $\Lambda$ and on the choice of positive system.

In this article, we solve the unitary branching problem for $\sl(m\vert n)$, $\osp(m\vert2n)$, and $F(4)$ by combining Dirac inequalities with a super Freudenthal recursion. We first consider a positive system adapted to the real form. In this case, the relative Dirac operator $\Dirac_{\gg,\even}$ is compatible with the unitary structure, and $\Dirac_{\gg,\even}^{2}$ gives a Dirac inequality for each $\even$-highest weight occurring in $\HH\vert_{\even}$. The same quadratic expression occurs as the coefficient of the corresponding weight multiplicity in a super Freudenthal identity, obtained from matrix coefficients of $\Dirac_{\gg,\even}^{2}$. Unitarity implies the strict positivity of this coefficient and therefore permits the Freudenthal recursion to be inverted. This gives a procedure valid for arbitrary atypicality, without a separate analysis of the different atypical cases. The weight multiplicities are determined recursively, and the branching multiplicities then follow from a triangular relation. Iterating the recursion gives a closed formula. For positive systems not adapted to the real form, we introduce corrected Dirac inequalities which restore the positivity required for the same recursion. Along the way, we obtain a new characterization of unitarity for $\slmn$ with respect to the standard positive system, completing the characterization of unitarity in terms of eigenvalues of $\Dirac_{\gg,\even}^{2}$ in this case.

\subsection{Organization and results} Section~\ref{sec::preliminaries_slmn_and_osp} collects the structural results on unitary highest weight representations of $\sl(m\vert n)$, $\osp(m\vert2n)$, and $F(4)$ used below. We first determine the positive systems relevant to unitarity and relate them by explicit sequences of odd reflections. For $\sl(m\vert n)$, $\osp(m\vert2n)$, and $F(4)$, these relations are given in Lemma~\ref{lemm::relating_standard_and_non_standard_simple_system}, Lemma~\ref{lemm::relating_A_and_B_simple_system}, and Lemma~\ref{lemm::odd_reflections_F(4)}, respectively. They provide, in particular, an explicit comparison of the corresponding highest weights. Proposition~\ref{prop::transport_of_unitarity} shows that unitarity is preserved under these changes of positive system. We then introduce adapted positive systems and determine the real forms admitting nontrivial unitary representations, together with the positive systems compatible with their conjugate-linear anti-involutions. The corresponding statements are Lemma~\ref{lemm::adapted_systems_slmn} for $\sl(m\vert n)$, Lemmas~\ref{lemm::adapted_systems_osp} and~\ref{lemm::adapted_systems_osp_star} for $\osp(m\vert2n)$, and Lemma~\ref{lemm::F4_adapted_positive_systems} for $F(4)$. Finally, we recall the even filtration and the resulting restrictions on the irreducible $\even$-constituents. These results form the basis for the branching problem considered in the subsequent sections.
\\
\\
Section~\ref{sec::Kostants_cubic_Dirac_operator} relates the relative Dirac operator $\Dirac_{\gg,\even}$ to unitarity and explains the role of the positive system. We first recall the construction of $\Dirac_{\gg,\even}$, its square, and the oscillator module $M(\odd)$ for the Weyl algebra of $\odd$. Proposition~\ref{prop::unitarity_oscillator_representation} characterizes the positive systems for which $M(\odd)$ is unitary, while Theorem~\ref{thm::dirac-square-action} determines the action of $\Dirac_{\gg,\even}^{2}$ on the $\even$-composition factors of $M\otimes M(\odd)$.

For an $\omega$-adapted positive system of sign $\varepsilon$, Theorem~\ref{thm::Dirac_adjointness} gives
$\Dirac_{\gg,\even}^{\ast}=-\varepsilon\Dirac_{\gg,\even}$.
The square of the Dirac operator therefore has a definite sign, and Proposition~\ref{prop::Dirac_inequality_adapted} yields the corresponding Dirac inequality. Restricting to constituents of $M$, Corollary~\ref{cor::Dirac_inequality_on_M} gives
\begin{equation}
-\varepsilon\bigl((\mu+2\rho,\mu)-(\Lambda+2\rho,\Lambda)\bigr)\geq0,
\end{equation}
with strict inequality for $\mu\neq\Lambda$.

Thus, for adapted positive systems, unitarity is reflected directly in the sign of $\Dirac_{\gg,\even}^{2}$. For non-adapted systems this argument no longer applies, since the required adjointness property is lost. The main result of the section is to replace the resulting Dirac inequality by a corrected one, thereby extending the Dirac method to arbitrary relevant positive systems.

The case of $\sl(m\vert n)$ provides the basic example. Assume $pq\neq0$. The standard positive system is then not $\omega$-adapted to the relevant noncompact real form, so that $\Dirac_{\gg,\even}$ has no definite adjointness property. We remedy this by decomposing
\begin{equation}
\Dirac_{\gg,\even}=\Dirac_{\self}+\Dirac_{\skew},
\qquad
\Dirac_{\self}^{\ast}=\Dirac_{\self},
\qquad
\Dirac_{\skew}^{\ast}=-\Dirac_{\skew}.
\end{equation}
Lemma~\ref{lemm::properties_D_self_and_D_skew} identifies these two summands with the relative Dirac operators of the naturally embedded subsuperalgebras $\sl(p\vert n)$ and $\sl(q\vert n)$. Thus the failure of adjointness of the full operator is separated into two Dirac operators with definite adjointness. We next compare the standard realization with the adapted nonstandard realization
$L(\Lambda;\Delta_{\st}^{+})\cong
\Uppi^{\upepsilon}L(\Lambda';\Delta_{\nst}^{+})$.
For this purpose, we introduce a correction term $\mathcal C_{\nst}^{\st}$. Lemma~\ref{lemm::corrected_Dirac_action} shows that the action of
$\Dirac_{\gg,\even}^{2}+\mathcal C_{\nst}^{\st}$ in the standard realization agrees with that of the Dirac square in the adapted realization. The Dirac inequality for the latter therefore induces a corrected Dirac inequality for the standard positive system. This gives Theorem~\ref{thm::unitarity_standard_case}, which characterizes unitarity directly in the standard realization. Namely, $L(\Lambda;\Delta_{\st}^{+})$ is unitary if and only if $L_0(\Lambda')$ is unitary and
\begin{equation}
-(\Lambda+2\rho^{\st},\Lambda)
+(\mu+2\rho^{\st},\mu)
+
2\sum_{i=p+1}^{m}\sum_{a=1}^{n}
(\mu,\epsilon_i-\delta_a)> 0
\end{equation}
for every $\even$-highest weight $\mu\neq\Lambda'$ occurring in the even filtration. This extends the Dirac criterion of \cite{SchmidtDirac} from the adapted nonstandard system to the standard positive system and completes the Dirac-theoretic characterization of unitarity for the two positive systems relevant to $\sl(m\vert n)$.

In general, the weight $\Lambda'$ is obtained from $\Lambda$ by a sequence of odd reflections, and no closed formula for this transformation is available. Corollary~\ref{cor::finding_Lambda'} gives a direct alternative: $\Lambda'$ is the unique $\even$-constituent of the standard realization for which the corrected Dirac inequality is an equality. Thus the highest weight of the adapted realization can be recovered intrinsically from the branching data, without carrying out the odd-reflection procedure.

Finally, Section~\ref{subsec::corrected_Dirac} extends this construction to $\osp(m\vert2n)$ and $F(4)$. Transporting the Dirac inequality along the odd reflections relating the relevant positive systems, we obtain explicit corrected inequalities for the non-$\omega$-adapted systems. Thus a Dirac inequality is available for every positive system occurring in the unitary highest weight theory considered here. 
\\ 
\\
In Section~\ref{sec::unitary_branching}, we solve the branching problem for unitary highest weight representations. The even filtration of Corollary~\ref{corollary::even-filtration-simple} gives a first restriction on the $\even$-constituents which may occur. Unitarity and the Dirac inequality sharpen this restriction and lead to an explicit set of admissible highest weights, denoted by $\Lambda-\Gamma(\Lambda)$. For $\osp(m\vert2n)$ and $F(4)$, Lemma~\ref{lemm::finite_stopping_branching} further reduces the problem to a finite computation. These conditions are only necessary. To decide which candidates actually occur, we derive in Proposition~\ref{prop::Freudenthal_identity} a super analogue of Freudenthal's identity directly from matrix coefficients of the square of the relative Dirac operator:
\begin{equation}
\bigl((\Lambda+2\rho,\Lambda)-(\mu+2\rho,\mu)\bigr)d(\mu)
=
2\sum_{\alpha\in\Delta^{+}}\sum_{k\geq1}
(-1)^{(k+1)p(\alpha)}
(\mu+k\alpha,\alpha)d(\mu+k\alpha).
\end{equation}
The significance of this formula is that its leading coefficient is exactly the quadratic expression governed by the Dirac inequality. Thus the same operator which supplies the positivity condition also produces the recursion. For a unitary highest weight representation, strictness of the Dirac inequality at every non-top $\even$-highest weight makes this coefficient nonzero. The identity can therefore be inverted and determines the weight multiplicities $d(\mu)$ recursively inside the simple representation.

We then recover the branching multiplicities from these weight multiplicities. Writing
$a_{\nu,\mu}=\dim L_0(\nu)^\mu$, Theorem~\ref{thm::branching_recursion} gives
\begin{equation}
\begin{aligned}
m_\mu
&=
d(\mu)
-
\sum_{\substack{\nu\in\Lambda-\Gamma(\Lambda)\\ \nu>\mu}}
m_\nu a_{\nu,\mu}
\\
&=
\sum_{\substack{\nu\in\Lambda-\Gamma(\Lambda)\\ \nu>\mu}}
\left(
\frac{
2\displaystyle\sum_{\alpha\in\Delta^+}\sum_{k\geq1}
(-1)^{(k+1)p(\alpha)}
(\mu+k\alpha,\alpha)
a_{\nu,\mu+k\alpha}
}{
(\Lambda+2\rho,\Lambda)-(\mu+2\rho,\mu)
}
-a_{\nu,\mu}
\right)m_\nu .
\end{aligned}
\end{equation}
Since the even weight multiplicities $a_{\nu,\mu}$ are explicitly computable, this formula determines the branching multiplicities recursively from higher to lower weights. Corollary~\ref{cor::branching_support} then determines the actual branching support by the nonvanishing of these multiplicities. Thus the combination of the $\even$-filtration, unitarity, the strict Dirac inequality, and the super Freudenthal recursion determines both which $\even$-constituents occur and their multiplicities. Theorem~\ref{thm::closed_formula} solves the triangular recursion explicitly and gives a closed formula for $m_{\mu}$ depending only on $\Lambda$ and $\mu$.

Finally, the correction theory developed in Section~\ref{sec::Kostants_cubic_Dirac_operator} extends this construction to non-$\omega$-adapted positive systems. In this setting, the ordinary Freudenthal identity is replaced by a corrected identity whose coefficient is precisely the corrected Dirac expression. The same recursive argument therefore applies, and the resulting recursive and closed branching formulas extend to all positive systems relevant to the unitary highest weight representations considered in this paper.

\subsection{Conventions}
We write $\ZZ_{\geq0}$ for the set of nonnegative integers, $\ZZ_{+}$ for the set of positive integers, and $\ZZ_{2}\coloneqq\ZZ/2\ZZ={\bar0,\bar1}$. The ground field is $\CC $, unless otherwise stated. A super vector space is a $\ZZ_{2}$-graded vector space $V=V_{\bar0}\oplus V_{\bar1}$. For homogeneous $v\in V$, its parity is denoted by $p(v)$. For $a\in\ZZ_{2}$, we set $i^{\bar0}\coloneqq1$ and $i^{\bar1}\coloneqq i$, so that $i^{\overline{a}+\overline{b}}=(-1)^{\overline{ab}}i^{\overline{a}}i^{\overline{b}}$ for all $\overline{a},\overline{b}\in \ZZ_{2}$. 

For super vector spaces $V$ and $W$, a homogeneous linear map $f\colon V\to W$ has parity $\bar{a}\in\ZZ_{2}$ if $f(V_{\bar{b}})\subseteq W_{\bar{a}+\bar{b}}$ for every $\bar{b}\in\ZZ_{2}$. In particular, an even linear map preserves the grading. We grade $\operatorname{End}(V)$ by
\begin{equation}
\operatorname{End}(V)_{\bar{a}}\coloneqq{f\in\operatorname{End}(V):f(V_{\bar{b}})\subseteq V_{\bar{a}+\bar{b}}\text{ for all }\bar{b}\in\ZZ_{2}}.
\end{equation}
A bilinear form $B$ on $V$ is even if $B(V_{\bar{a}},V_{\bar{b}})=0$ whenever $\bar{a}+\bar{b}\neq\bar0$. Hermitian forms are taken to be conjugate-linear in the first argument and linear in the second. Unless stated otherwise, all linear maps and bilinear forms are even. 

A representation of a Lie superalgebra $\gg$ on $V$ is a Lie superalgebra homomorphism $\rho\colon\gg\to\operatorname{End}(V)$ satisfying $\rho(\gg_{\bar{a}})\subseteq\operatorname{End}(V)_{\bar{a}}$ for every $\bar{a}\in\ZZ_{2}$. We suppress $\rho$ whenever convenient and write $Xv\coloneqq\rho(X)v$. All representations are left representations. 

For a representation $V=V_{\bar0}\oplus V_{\bar1}$ of $\gg$, we denote by $\Pi V$ its parity reversal, defined by
$(\Pi V)_{\bar0}=V_{\bar1}$ and $(\Pi V)_{\bar1}=V_{\bar0}$, with the same action of $\gg$. In general, $V$ and $\Pi V$ are not isomorphic as $\ZZ_2$-graded representations, although they are isomorphic after forgetting the grading.

If $V=V_{\bar0}\oplus V_{\bar1}$ is a representation of $\gg$, then $\even$ preserves both homogeneous components. Hence $V_{\bar0}$ and $V_{\bar1}$ are ordinary $\even$-representations. We denote the restriction of $V$ to $\even$, with the grading forgotten, by
\begin{equation}
V\big\vert_{\even}=V_{\bar0}\oplus V_{\bar1}.
\end{equation}

When $\gg=\sl(m\vert n)$, we keep the ordered pair $(m,n)$ fixed and do not identify $\sl(m\vert n)$ with $\sl(n\vert m)$ by exchanging the even and odd blocks.

Finally, the universal enveloping algebra of $\gg$ and its center are denoted by $\UE(\gg)$ and $Z(\gg)$, respectively.
\section{\texorpdfstring{$\osp(m\vert 2n), \slmn$}{} and \texorpdfstring{F(4)}{}: real forms, positive systems, and unitarity}\label{sec::preliminaries_slmn_and_osp}
\noindent
In this section, we collect the structural facts concerning the Lie superalgebras $\sl(m\vert n),$ $\osp(m\vert 2n)$, and $F(4)$ that will be used in the sequel. We fix conventions for Cartan subalgebras, root systems, positive systems, weights, and invariant bilinear forms, and recall the corresponding decompositions into even and odd parts. We then describe the choices of Borel subalgebras relevant to the study of highest weight representations. Since, in contrast to the situation for reductive Lie algebras, Borel subalgebras of a basic classical Lie superalgebra need not be conjugate, the highest weight of a fixed simple representation depends essentially on the chosen positive system. Particular attention is given to the relation between the different positive systems occurring in the classification of unitary highest weight representations. These positive systems are related by sequences of odd reflections, and we record the corresponding transformations. We further discuss their compatibility with the conjugate-linear anti-involutions defining the relevant real forms. Finally, we recall the notion of a highest weight relative to a fixed Borel subalgebra and introduce the corresponding notion of atypicality.

\subsection{Structure theory}\label{subsec::structure_theory_slmn}
The starting point is the family of general linear Lie superalgebras. Let $V=\CC ^{m\vert n}$, where $m,n\geq1$ and $m+n>2$. The general linear Lie superalgebra $\gl(m\vert n)\coloneqq\gl(V)$ consists of block matrices $X=\left(\begin{smallmatrix}A&B\\ C&D\end{smallmatrix}\right)$, where $A$ and $D$ are square matrices of sizes $m$ and $n$, respectively, and $B$ and $C$ are matrices of sizes $m\times n$ and $n\times m$, respectively. The diagonal blocks form the even part $\gl(m\vert n)_{\bar{0}}$, while the off-diagonal blocks form the odd part $\gl(m\vert n)_{\bar{1}}$. The bracket is $[X,Y]=XY-(-1)^{p(X)p(Y)}YX$ for all homogeneous $X,Y \in \gl(m\vert n)$ and then extended linearly. 

The supertrace of $X=\left(\begin{smallmatrix}A&B\\ C&D\end{smallmatrix}\right)\in \glmn$ is defined by $\str(X)\coloneqq\operatorname{tr}(A)-\operatorname{tr}(D)$. The supertrace satisfies $\str([X,Y])=0$ for all $X,Y\in\gl(m\vert n)$. It defines a bilinear form $(X,Y)\coloneqq\str(XY)$, which is nondegenerate, even, supersymmetric and invariant. Thus, for homogeneous $X,Y,Z\in\gl(m\vert n)$, one has $(\gl(m\vert n)_{\bar0},\gl(m\vert n)_{\bar1})=0$, $(X,Y)=(-1)^{p(X)p(Y)}(Y,X)$, and $([X,Y],Z)=(X,[Y,Z])$.

The special linear and orthosymplectic Lie superalgebras $\sl(m\vert n)$ and $\osp(m\vert 2n)$ will be realized below as Lie subsuperalgebras of suitable general linear Lie superalgebras. For $m\neq n$, the Lie superalgebra $\sl(m\vert n)$ is basic classical, and the same holds for $\osp(m\vert 2n)$ and $F(4)$. Recall that a finite-dimensional complex Lie superalgebra $\gg=\even\oplus\odd$ is called basic classical if it is simple, its even part $\even$ is reductive, and it admits a nondegenerate invariant even supersymmetric bilinear form. Such a form is unique up to multiplication by a nonzero scalar. Throughout, we use the normalization
\begin{equation}\label{eq::supertrace_form}
(X,Y)\coloneqq
\begin{cases}
\str(XY)& \text{for} \ \sl(m\vert n) \ \text{with}\ m\neq n,\\
\frac12\str(XY)& \text{for} \ \osp(m\vert 2n),\\ 
\frac16\str_{\gg}\bigl(\ad(X)\ad(Y)\bigr)& \text{for} \ F(4).
\end{cases}
\end{equation}
Although $\gl(m\vert n)$ is not simple, we follow the convention of including it among the basic classical Lie superalgebras. In what follows, $\gg$ denotes one of the Lie superalgebras $\sl(m\vert n)$ with $m\neq n$, $\gl(n\vert n)$, $\osp(m\vert 2n)$ or $F(4)$. The case of $\sl(n\vert n)$ will be treated by restriction from $\gl(n\vert n)$.

In each of the standard matrix realizations below, the Lie superalgebras have a Cartan subalgebra $\hh\subseteq\even$. The adjoint action of $\hh$ yields the root space decomposition
\begin{equation}
\gg=\hh\oplus\bigoplus_{\alpha\in\Delta}\gg^\alpha,
\qquad
\gg^\alpha\coloneqq
\{X\in\gg:[H,X]=\alpha(H)X\text{ for all }H\in\hh\}.
\end{equation}
The root system decomposes as $\Delta=\Delta_{\bar0}\sqcup\Delta_{\bar1}$ according to the parity of the corresponding root spaces, all of which are one-dimensional. For a positive system $\Delta^+=\Delta_{\bar0}^+\sqcup\Delta_{\bar1}^+$, a root $\alpha\in\Delta^+$ is called simple if it cannot be written as a sum of two positive roots. The set of simple roots is denoted by $\Uppi$. The corresponding Weyl vector is
\begin{equation}
\rho=\rho_{\bar0}-\rho_{\bar1},
\qquad
\rho_{\bar z}\coloneqq\frac12\sum_{\alpha\in\Delta_{\bar z}^+}\alpha, \qquad \bar{z}\in \ZZ_{2}.
\end{equation} 
Moreover, any choice of a positive system $\Delta^{+}$ determines a triangular decomposition $\gg = \nn^{-}\oplus \hh \oplus \nn^{+}$, where $\nn^{\pm} \coloneqq \bigoplus_{\alpha\in \Delta^{+}}\gg^{\pm \alpha}$. We call $\bb\coloneqq\hh\oplus\nn^+$ the associated Borel subalgebra. For a fixed Cartan subalgebra $\hh$, positive systems, simple systems, and Borel subalgebras are in bijection, and we pass between them without further comment.

Finally, the Weyl group $W$ of $\gg$ is the Weyl group of the reductive Lie algebra $\even$. For an ordinary reductive Lie algebra, any two simple systems are conjugate under the Weyl group. In contrast to the ordinary case, $W$ need not act transitively on the simple systems of $\gg$. For instance, the standard and nonstandard simple systems for $\sl(m\vert n)$ (see~Example~\ref{ex::systems}) are not conjugate under $W$, but they are related by a sequence of odd reflections.
Let $\Uppi$ be a simple system and let $\theta\in\Uppi\cap\Delta_{\bar{1}}$ be isotropic. The odd reflection at $\theta$ produces the simple system $\Uppi_\theta\coloneqq r_\theta(\Uppi)$, where, for $\alpha\in\Uppi$,
\begin{equation}
r_\theta(\alpha)=
\begin{cases}
-\theta,&\alpha=\theta,\\
\alpha+\theta,&\alpha\neq\theta\text{ and }(\alpha,\theta)\neq0,\\
\alpha,&\alpha\neq\theta\text{ and }(\alpha,\theta)=0.
\end{cases}
\label{eq::odd-reflection}
\end{equation}
The corresponding positive system is $\Delta_\theta^+=(\Delta^+\setminus\{\theta\})\cup\{-\theta\}$. Any two simple systems inducing the same positive system $\Delta_{\bar{0}}^+$ on the even roots are related by a sequence of odd reflections.

\subsubsection{Special linear Lie superalgebras}

The special linear Lie superalgebra is the codimension-one ideal $\sl(m\vert n)\coloneqq\ker(\str)$ of $\gl(m\vert n)$. Its even part is $\sl(m\vert n)_{\bar0}\cong\sl(m)\oplus\sl(n)\oplus\CC z_{m\vert n}$, where $z_{m\vert n}\coloneqq\operatorname{diag}(nI_m,mI_n)$ spans the center of $\sl(m\vert n)_{\bar0}$. If $m\neq n$, then $\sl(m\vert n)$ is simple and $\gl(m\vert n)=\sl(m\vert n)\oplus\CC I_{m+n}$. If $m=n$, then $Z(\sl(n\vert n))=\CC I_{2n}$ and $\psl(n\vert n)\coloneqq\sl(n\vert n)/\CC I_{2n}$ is simple. The restriction of the supertrace form of $\gl(n\vert n)$ to $\sl(n\vert n)$ has radical $\CC I_{2n}$, and it is therefore not nondegenerate. In the $m=n$ case, $\gl(n\vert n)$ will be used instead of $\sl(n\vert n)$. Thus, throughout this subsubsection, $\gg$ denotes $\sl(m\vert n)$ if $m\neq n$ and $\gl(n\vert n)$ if $m=n$. Moreover, throughout we fix the Lie superalgebra $\sl(m\vert n)$ and do not identify it with $\sl(n\vert m)$ by exchanging the even and odd blocks.

As a $\even$-representation under the adjoint action, $\odd$ decomposes into two simple $\even$- representations $\gg_{\pm 1}$, where $\gg_{\pm 1}$ are exactly the case where either $C=0$ or $B=0$ in the standard block form. These spaces are moreover abelian. In particular, this gives a $\ZZ$-grading $\gg = \gg_{-1}\oplus \gg_{0}\oplus \gg_{+1}$ with $\gg_{0}\coloneqq \gg_{\bar{0}}$. 

Let $\mathfrak d$ be the diagonal subalgebra of $\gl(m\vert n)$. The dual space $\mathfrak d^{\ast}$ has basis $\epsilon_1,\ldots,\epsilon_m,\delta_1,\ldots,\delta_n$, defined by $\epsilon_i(\operatorname{diag}(h_1,\ldots,h_{m+n}))=h_i$ and $\delta_a(\operatorname{diag}(h_1,\ldots,h_{m+n}))=h_{m+a}$ for $1\leq i\leq m$ and $1\leq a\leq n$. The restriction of the supertrace form $(\cdot,\cdot)$ to $\mathfrak{d}$ is nondegenerate and therefore induces a bilinear form on $\mathfrak d^{\ast}$ satisfying $(\epsilon_i,\epsilon_j)=\delta_{ij}$, $(\delta_a,\delta_b)=-\delta_{ab}$, and $(\epsilon_i,\delta_a)=0$.

Set $\kappa\coloneqq\sum_{i=1}^m\epsilon_i-\sum_{a=1}^n\delta_a\in \mathfrak{d}^{\ast}$. The Cartan subalgebra and its dual are
\begin{equation}
\hh=
\begin{cases}
\mathfrak d\cap\sl(m\vert n)=\ker(\kappa),&m\neq n,\\
\mathfrak d,&m=n,
\end{cases}
\qquad
\hh^{\ast}\cong
\begin{cases}
\mathfrak d^{\ast}/\CC \kappa,&m\neq n,\\
\mathfrak d^{\ast},&m=n.
\end{cases}
\end{equation}

A weight $\lambda\in\hh^{\ast}$ is represented by a tuple $\widetilde\lambda=(\lambda_1,\ldots,\lambda_m\vert\mu_1,\ldots,\mu_n)\in\mathfrak d^{\ast}$, determined up to addition of a multiple of $\kappa$ if $m\neq n$. The tuple $\widetilde\lambda$ is not required to be orthogonal to $\kappa$. Pairings $(\lambda,\alpha)$ with roots $\alpha\in\Delta$ may be computed using any representative of $\lambda$, since $(\kappa,\alpha)=0$. By contrast, pairings between arbitrary weights are understood with respect to the bilinear form on $\hh^{\ast}$ induced by the nondegenerate restriction of $(\cdot,\cdot)$ to $\hh$.

The nonzero root spaces are the same as those of $\gl(m\vert n)$, and the root system is $\Delta=\Delta_{\bar0}\sqcup\Delta_{\bar1}$, where
\begin{equation}
\Delta_{\bar0}
=
\{\epsilon_i-\epsilon_j:i\neq j\}
\sqcup
\{\delta_a-\delta_b:a\neq b\},
\qquad
\Delta_{\bar1}
=
\{\pm(\epsilon_i-\delta_a):1\leq i\leq m,\ 1\leq a\leq n\}.
\label{eq:sl-roots}
\end{equation}

The positive even roots are fixed as
\begin{equation}
\Delta_{\bar0}^+
=
\{\epsilon_i-\epsilon_j:1\leq i<j\leq m\}
\sqcup
\{\delta_a-\delta_b:1\leq a<b\leq n\}.
\label{eq:sl-positive-even-roots}
\end{equation}
This is the standard choice for which the positive root spaces in $\even$ are spanned by the upper triangular matrix units $E_{ij}$, with $1\leq i<j\leq m$, and $E_{m+a,m+b}$, with $1\leq a<b\leq n$. For fixed $\Delta_{\bar{0}}^{+}$, two positive systems $\Delta^{+}$ extending $\Delta_{\bar{0}}^{+}$ are relevant for unitarity (see~Section~\ref{subsec:real_forms_slmb}).

\begin{example}\label{ex::systems} Let $m=p+q$ with $p,q\geq 0$. \begin{enumerate} \item[(a)] The standard, or distinguished, positive system is $$ \Delta_{\mathrm{st}}^+ =\Delta_{\bar{0}}^+\sqcup \{\epsilon_i-\delta_a:1\leq i\leq m,\ 1\leq a\leq n\}. $$
In this case, the odd positive root vectors span $\gg_{+1}$. We denote the associated Borel subalgebra by $\bb_{\text{st}}$. Its simple system is $$ \Uppi_{\mathrm{st}} =\{\epsilon_i-\epsilon_{i+1}:1\leq i<m\} \cup\{\epsilon_m-\delta_1\} \cup\{\delta_a-\delta_{a+1}:1\leq a<n\}. $$ 

\item[(b)] The $pq$-nonstandard positive system is defined by
\[
\Delta_{\mathrm{nst}}^+
=
\Delta_{\bar0}^+
\sqcup
\{\epsilon_i-\delta_a:1\leq i\leq p,\ 1\leq a\leq n\}
\sqcup
\{\delta_a-\epsilon_j:p<j\leq m,\ 1\leq a\leq n\}.
\]
For $q=0$, this coincides with the standard positive system. For $p=0$, the positive even roots remain unchanged, whereas the signs of all odd roots are reversed relative to the standard positive system. When $p$ and $q$ are fixed and $pq\neq 0$, we simply refer to $\Delta_{\mathrm{nst}}^+$ as the nonstandard positive system. We explain the origin of this positive system in Section~\ref{subsec::maximal_compact_subalgebra}. The corresponding Borel subalgebra will be denoted by $\bb_{\mathrm{nst}}$. If $pq\neq 0$, its simple system is $$ \begin{aligned} \Uppi_{\mathrm{nst}} ={}&\{\epsilon_i-\epsilon_{i+1}:1\leq i<p\} \cup\{\epsilon_p-\delta_1\} \cup\{\delta_a-\delta_{a+1}:1\leq a<n\}\\ &\cup\{\delta_n-\epsilon_{p+1}\} \cup\{\epsilon_j-\epsilon_{j+1}:p<j<m\}. \end{aligned} $$ \end{enumerate} \end{example}

We relate both systems using odd reflections. A simple system inducing $\Delta_{\bar0}^+$ can be encoded by an ordered weight word $s=(s_1,\ldots,s_{m+n})$ in which each element of $\{\epsilon_1,\ldots,\epsilon_m,\delta_1,\ldots,\delta_n\}$ occurs exactly once and the relative orders of the $\epsilon_i$ and of the $\delta_a$ are preserved. The corresponding simple system is $\Uppi_s\coloneqq\{s_i-s_{i+1}:1\leq i<m+n\}$. The standard and nonstandard systems correspond respectively to
\begin{equation}
s_{\mathrm{st}}
=
(\epsilon_1,\ldots,\epsilon_m,\delta_1,\ldots,\delta_n),
\qquad
s_{\mathrm{nst}}
=
(\epsilon_1,\ldots,\epsilon_p,\delta_1,\ldots,\delta_n,
\epsilon_{p+1},\ldots,\epsilon_m).
\end{equation}
An odd reflection interchanges two adjacent entries of different types. Moving $\epsilon_{p+1},\ldots,\epsilon_m$ successively across the $\delta_a$ therefore gives the following relation.

\begin{lemma}\label{lemm::relating_standard_and_non_standard_simple_system}
One has for $m=p+q$ with $p,q\geq 0$
\begin{equation*}
\Uppi_{\mathrm{st}}
=
\bigl(r_{\delta_1-\epsilon_m}\cdots r_{\delta_n-\epsilon_m}\bigr)
\cdots
\bigl(r_{\delta_1-\epsilon_{p+1}}\cdots r_{\delta_n-\epsilon_{p+1}}\bigr)
\Uppi_{\mathrm{nst}}.
\end{equation*}
\end{lemma}

For later computations, we next record the Weyl vectors associated with the two positive systems introduced above. For the standard positive system, the even and odd Weyl vectors are given by
\begin{equation}
\rho_{\bar0}
=
\frac12
\left(
\sum_{i=1}^m(m-2i+1)\epsilon_i
+
\sum_{a=1}^n(n-2a+1)\delta_a
\right),
\qquad
\rho_{\bar1}
=
\frac12
\left(
n\sum_{i=1}^m\epsilon_i
-
m\sum_{a=1}^n\delta_a
\right).
\label{eq:rho}
\end{equation}
For the nonstandard positive system associated with $m=p+q$, Lemma~\ref{lemm::relating_standard_and_non_standard_simple_system} gives
\begin{equation}
\rho_{\mathrm{nst}}
=
\rho_{\mathrm{st}}
-
\sum_{i=p+1}^m\sum_{a=1}^n(\delta_a-\epsilon_i).
\label{eq:rho-standard-nonstandard}
\end{equation}

\subsubsection{Orthosymplectic Lie superalgebras} Let $V=\CC ^{m\vert 2n}$, where $m,n\geq1$, and set $r\coloneqq\lfloor m/2\rfloor$. We equip $V$ with a nondegenerate even supersymmetric bilinear form $\langle\cdot,\cdot\rangle$. The orthosymplectic Lie superalgebra is the Lie subsuperalgebra of $\gl(m\vert 2n)$ preserving this form. Concretely, for $\bar z\in\ZZ_2$, define $\osp(m\vert 2n)_{\bar z}$ to consist of all
$X\in\gl(m\vert 2n)_{\bar z}$ such that
\begin{equation}
\langle Xv,w\rangle
+
(-1)^{p(X)p(v)}
\langle v,Xw\rangle
=0
\end{equation}
for all homogeneous $v,w\in V$. Any two nondegenerate even supersymmetric bilinear forms on $V$ are equivalent under an even change of basis and therefore define isomorphic Lie superalgebras. It is convenient to realize $\osp(m\vert 2n)$ as a Lie subsuperalgebra of
$\sl(m\vert 2n)$ in the following matrix realization. For every $n$, let
$I_n$ denote the $n\times n$ identity matrix and set
$
J_n\coloneqq
\left(
\begin{smallmatrix}
0&I_n\\
-I_n&0
\end{smallmatrix}
\right).$ Then
\begin{equation}
\osp(m\vert 2n)
=
\left\{
\begin{pmatrix}
A&B\\
C&D
\end{pmatrix}
\in\sl(m\vert 2n)
:
\begin{pmatrix}
-A^T&-C^TJ_n\\
-J_nB^T&J_nD^TJ_n
\end{pmatrix}
=
\begin{pmatrix}
A&B\\
C&D
\end{pmatrix}
\right\}.
\end{equation}
In the following, we write $\gg\coloneqq\osp(m\vert 2n)$ and equip it with the
supertrace form~\eqref{eq::supertrace_form}. Its even subalgebra is
$\even\cong\so(m)\oplus\sp(2n)$, while
$\odd\cong\CC^m\otimes\CC^{2n}$ as a $\even$-representation. For
$m\neq2$, the odd part $\odd$ is irreducible as a $\even$-representation.
If $m=2$, the standard $\so(2)$-representation decomposes as
$\CC^2\cong\CC_{+}\oplus\CC_{-}$ into its two one-dimensional weight
spaces, and consequently $\odd$ decomposes into irreducible $\even$-representations.

Let $r=\lfloor m/2\rfloor$ and let $\mathfrak h\subseteq\mathfrak{osp}(m\vert 2n;\mathbb C)$ be the subalgebra consisting of elements
\begin{equation}
H=\operatorname{diag}\left(
t_1K,\ldots,t_rK,0_{m-2r},
\begin{pmatrix}
0&iD\\
-iD&0
\end{pmatrix}
\right),
\qquad
K=\begin{pmatrix}
0&i\\
-i&0
\end{pmatrix},
\end{equation}
where $D=\operatorname{diag}(d_1,\ldots,d_n)$ and $t_i,d_a\in\mathbb C$. The block $0_{m-2r}$ is absent when $m=2r$ and consists of a single zero entry when $m=2r+1$. Define $\epsilon_i,\delta_a\in\mathfrak h^*$ by $\epsilon_i(H)=t_i$ and $\delta_a(H)=d_a$. Note that the weights of the standard $\so(m)$-representation are $\{\pm\epsilon_i:1\leq i\leq r\}$ if $m=2r$ and $\{0,\pm\epsilon_i:1\leq i\leq r\}$ if $m=2r+1$, while those of the standard $\sp(2n)$-representation are $\{\pm\delta_a:1\leq a\leq n\}$. A weight $\lambda\in\hh^{\ast}$ is written as $\lambda=(\lambda_1,\ldots,\lambda_r\vert\mu_1,\ldots,\mu_n)$.

Writing $m=2r$ or $m=2r+1$, the root system $\Delta=\Delta_{\bar0}\sqcup\Delta_{\bar1}$ is given by
\begin{equation}
\begin{aligned}
\Delta_{\bar0}
&=
\{\pm\delta_a\pm\delta_b:1\leq a<b\leq n\}
\cup\{\pm2\delta_a:1\leq a\leq n\}
\cup\{\pm\epsilon_i\pm\epsilon_j:1\leq i<j\leq r\},\\
\Delta_{\bar1}
&=
\{\pm\epsilon_i\pm\delta_a:1\leq i\leq r,\ 1\leq a\leq n\}.
\end{aligned}
\label{eq:osp-roots}
\end{equation}
If $m=2r+1$, one additionally has $\{\pm\epsilon_i:1\leq i\leq r\}\subseteq\Delta_{\bar0}$ and $\{\pm\delta_a:1\leq a\leq n\}\subseteq\Delta_{\bar1}$. The roots $\epsilon_i\pm\delta_a$ are isotropic, whereas the additional odd roots $\delta_a$ occurring for odd $m$ are non-isotropic.

In this article, the standard positive system for the even roots is fixed as
\begin{equation}
\Delta_{\bar0}^+
=
\{\delta_a\pm\delta_b:1\leq a<b\leq n\}
\cup\{2\delta_a:1\leq a\leq n\}
\sqcup
\{\epsilon_i\pm\epsilon_j:1\leq i<j\leq r\}.
\label{eq:osp-positive-even-roots}
\end{equation}
If $m=2r+1$, we additionally include $\{\epsilon_i:1\leq i\leq r\}$ in $\Delta_{\bar0}^+$. For fixed $\Delta_{\bar{0}}^{+}$, two positive systems $\Delta^{+}$ extending $\Delta_{\bar{0}}^{+}$ are relevant for unitarity. We explain the choice in Section~\ref{subsec::maximal_compact_subalgebra}.

\begin{example}
\label{ex::osp-systems} We consider the two positive systems denoted by Case~A and Case~B in Jakobsen's classification \cite{jakobsen1994full}.
\begin{enumerate}
\item[(a)] In Case~A, the positive odd roots are
\[
\Delta_{A,\bar1}^{+}
=
\begin{cases}
\{\delta_a\pm\epsilon_i:1\leq a\leq n,\ 1\leq i\leq r\}
 \cup\{\delta_a:1\leq a\leq n\},
&m=2r+1,\\
\{\delta_a\pm\epsilon_i:1\leq a\leq n,\ 1\leq i\leq r\},
&m=2r.
\end{cases}
\]
Set $\Delta_A^{+}\coloneqq\Delta_{\bar0}^{+}\sqcup\Delta_{A,\bar1}^{+}$ and denote the corresponding Borel subalgebra by $\bb_A$. For $r\geq1$, the simple system of $\osp(2r+1\vert2n)$ is
\[
\Uppi_A
=
\{\delta_a-\delta_{a+1}:1\leq a<n\}
\cup\{\delta_n-\epsilon_1\}
\cup\{\epsilon_i-\epsilon_{i+1}:1\leq i<r\}
\cup\{\epsilon_r\},
\]
whereas, for $r\geq2$, the simple system of $\osp(2r\vert2n)$ is
\[
\Uppi_A
=
\{\delta_a-\delta_{a+1}:1\leq a<n\}
\cup\{\delta_n-\epsilon_1\}
\cup\{\epsilon_i-\epsilon_{i+1}:1\leq i<r\}
\cup\{\epsilon_{r-1}+\epsilon_r\}.
\]

\item[(b)] In Case~B, the positive odd roots are
\[
\Delta_{B,\bar1}^{+}
=
\begin{cases}
\{\epsilon_i\pm\delta_a:1\leq i\leq r,\ 1\leq a\leq n\}
 \cup\{\delta_a:1\leq a\leq n\},
&m=2r+1,\\
\{\epsilon_i\pm\delta_a:1\leq i\leq r,\ 1\leq a\leq n\},
&m=2r.
\end{cases}
\]
Set $\Delta_B^{+}\coloneqq\Delta_{\bar0}^{+}\sqcup\Delta_{B,\bar1}^{+}$ and denote the corresponding Borel subalgebra by $\bb_B$. For $r\geq1$, the simple system of $\osp(2r+1\vert2n)$ is
\[
\Uppi_B
=
\{\epsilon_i-\epsilon_{i+1}:1\leq i<r\}
\cup\{\epsilon_r-\delta_1\}
\cup\{\delta_a-\delta_{a+1}:1\leq a<n\}
\cup\{\delta_n\},
\]
whereas, for $r\geq2$, the simple system of $\osp(2r\vert2n)$ is
\[
\Uppi_B
=
\{\epsilon_i-\epsilon_{i+1}:1\leq i<r\}
\cup\{\epsilon_r-\delta_1\}
\cup\{\delta_a-\delta_{a+1}:1\leq a<n\}
\cup\{2\delta_n\}.
\]
\end{enumerate}
The roots $\delta_a+\epsilon_i$ are positive in both cases, while each isotropic root $\delta_a-\epsilon_i$ occurring in Case~A is replaced in Case~B by its negative $\epsilon_i-\delta_a$. When $m$ is odd, the non-isotropic odd roots $\delta_a$ are positive in both systems.

For $m=2$, the formulas become
\[
\Uppi_A
=
\{\delta_a-\delta_{a+1}:1\leq a<n\}
\cup\{\delta_n-\epsilon_1,\delta_n+\epsilon_1\},
\qquad
\Uppi_B
=
\{\epsilon_1-\delta_1\}
\cup\{\delta_a-\delta_{a+1}:1\leq a<n\}
\cup\{2\delta_n\}.
\]
For $m=1$, the two positive systems coincide and
$
\Uppi_A=\Uppi_B
=
\{\delta_a-\delta_{a+1}:1\leq a<n\}
\cup\{\delta_n\}.
$ \end{example}

We relate the Case~A and Case~B positive systems using a sequence of odd reflections. As for special linear Lie superalgebras, the simple systems may be encoded by ordered weight words, but the orthosymplectic root systems require an additional root at the right endpoint. Write $m=2r$ or $m=2r+1$ and set $\ell\coloneqq r+n$. An ordered weight word is a shuffle $s=(s_1,\ldots,s_\ell)$ of $\epsilon_1,\ldots,\epsilon_r,\delta_1,\ldots,\delta_n$ preserving the relative orders of the $\epsilon_i$ and of the $\delta_a$. The corresponding simple system is
\begin{equation}
\Uppi_s=
\begin{cases}
\{s_i-s_{i+1}:1\leq i<\ell\}\cup\{s_\ell\},
&m=2r+1,\\
\{s_i-s_{i+1}:1\leq i<\ell\}\cup\{s_{\ell-1}+s_\ell\},
&m=2r\text{ and }s_\ell\text{ is of }\epsilon\text{-type},\\
\{s_i-s_{i+1}:1\leq i<\ell\}\cup\{2s_\ell\},
&m=2r\text{ and }s_\ell\text{ is of }\delta\text{-type}.
\end{cases}
\label{eq:osp-weight-word}
\end{equation}
For $m=2r$, the complete parametrization additionally allows $\epsilon_r$ to be replaced by $-\epsilon_r$. The resulting data are called extended parity sequences. This additional sign is not needed for the two systems considered here. The ordered weight words corresponding to Cases~A and~B are
\begin{equation}
s_A=(\delta_1,\ldots,\delta_n,\epsilon_1,\ldots,\epsilon_r),
\qquad
s_B=(\epsilon_1,\ldots,\epsilon_r,\delta_1,\ldots,\delta_n).
\label{eq:osp-case-words}
\end{equation}
An odd reflection interchanges two adjacent entries of different types. Thus, the passage from Case~A to Case~B is obtained by moving the $\epsilon_i$ successively to the left across the $\delta_a$. Consequently, we obtain the following lemma.

\begin{lemma}\label{lemm::relating_A_and_B_simple_system} One has
\begin{equation*}
\Uppi_A
=
\bigl(r_{\epsilon_1-\delta_n}\cdots r_{\epsilon_1-\delta_1}\bigr)
\cdots
\bigl(r_{\epsilon_r-\delta_n}\cdots r_{\epsilon_r-\delta_1}\bigr)
\Uppi_B.
\end{equation*}
\end{lemma}

\subsubsection{The exceptional Lie superalgebra \texorpdfstring{$F(4)$}{F(4)}}\label{subsec::F4}
Unlike the classical Lie superalgebras
$\mathfrak{sl}(m\vert n)$ and $\mathfrak{osp}(m\vert 2n)$, the exceptional simple Lie superalgebra $F(4)$ does not come with a canonical defining matrix
realization. Its natural description is instead obtained as follows. Let $(V,\langle\cdot,\cdot\rangle_{V})$ be a
seven-dimensional complex orthogonal vector space, let $S$ be the
eight-dimensional spin representation of $\mathfrak{so}(V)$, and let
$(W,\langle\cdot,\cdot\rangle_{W})$ be a two-dimensional complex symplectic vector space. Then
\begin{equation}
F(4)_{\bar0}
=
\mathfrak{so}(V)\oplus\mathfrak{sp}(W),
\qquad
F(4)_{\bar1}
=
S\otimes W,
\end{equation}
where $\mathfrak{sp}(W)\cong\mathfrak{sl}(2)$ and $\so(V)\cong \so(7)$. The bracket on the even
part is the ordinary Lie bracket, and the action of the even part on the
odd part is the tensor product action. The bracket of two odd elements is,
up to normalization,
\begin{equation}
[s\otimes u,t\otimes v]
=
\langle u,v\rangle_{W}\,\Gamma(s,t)
+
\langle s,t\rangle_{S}(\langle u,\cdot\rangle_{W}v+\langle v,\cdot\rangle_{W}u),
\end{equation}
where
$\Gamma:\bigwedge^{2}S\to\mathfrak{so}(V)$ is the canonical
$\mathfrak{so}(V)$-equivariant map, and $\langle \cdot,\cdot\rangle_{S}$ is the invariant symmetric bilinear form on $S$. Of course, $F(4)$ admits faithful matrix realizations. For example, the
adjoint representation gives an embedding
$
F(4)\hookrightarrow\mathfrak{gl}(24\vert16).
$

For each of the real forms considered below (see Section~\ref{subsec:real_forms_slmb}), fix a Cartan subalgebra
$\mathfrak t_{\mathbb R}$ of a maximal compact subalgebra of
$F(4)_{\bar0,\mathbb R}$, and let
$
\mathfrak h\coloneqq\mathfrak t_{\mathbb R}\otimes_{\mathbb R}\mathbb C
\subseteq F(4)_{\bar{0}}.
$ The roots corresponding to the $\mathfrak{so}(7)$-factor are expressed
in terms of $\epsilon_1,\epsilon_2,\epsilon_3$, whereas $\delta$ is the
root of the $\mathfrak{sl}(2)$-factor. The even and odd root systems are
\begin{equation}
\Delta_{\bar0}
=
\{\pm\delta\}
\cup
\{\pm\epsilon_i:1\leq i\leq3\}
\cup
\{\pm\epsilon_i\pm\epsilon_j:1\leq i<j\leq3\},
\quad 
\Delta_{\bar1}
=
\{
\tfrac12
\bigl(
\pm\delta\pm\epsilon_1\pm\epsilon_2\pm\epsilon_3
\bigr)
\}.
\end{equation}

The normalized super Killing form $(X,Y)\coloneqq\frac16\str_{\gg}\bigl(\ad(X)\ad(Y)\bigr)$ on $F(4)$ induces a form $(\cdot,\cdot)$ on $\hh^{\ast}$ determined by 
\begin{equation}\label{eq::F4_bilinear_form}
(\epsilon_i,\epsilon_j)=\delta_{ij},
\qquad
(\delta,\delta)=-3,
\qquad
(\epsilon_i,\delta)=0.
\end{equation}
Note that every odd root is isotropic with respect to $(\cdot,\cdot)$.

In what follows, we fix the standard positive system $\Delta_{\bar0}^{+}$, that is,
\begin{equation}\Delta_{\bar0}^{+} = \{\delta\} \cup\{\epsilon_i:1\leq i\leq3\} \cup\{\epsilon_i\pm\epsilon_j:1\leq i<j\leq3\}.
\end{equation}

Up to the action of the Weyl group
$
W \cong \bigl((\mathbb Z_2)^3\rtimes S_3\bigr)\times\mathbb Z_2,$
there are five simple systems, whose associated positive systems contain $\Delta_{\bar{0}}^{+}$. Of particular importance are 
\begin{equation}
\begin{aligned}
\Uppi_1
&=
\left\{
\epsilon_1-\epsilon_2,\,
\epsilon_2-\epsilon_3,\,
\epsilon_3,\,
\frac12(-\epsilon_1-\epsilon_2-\epsilon_3+\delta)
\right\},
\\
\Uppi_2
&=
\left\{
\delta,\,
\epsilon_3,\,
\epsilon_2-\epsilon_3,\,
\frac12(\epsilon_1-\epsilon_2-\epsilon_3-\delta)
\right\},
\\
\Uppi_3
&=
\left\{
\frac12(\epsilon_1-\epsilon_2-\epsilon_3+\delta),\,
\frac12(\epsilon_1-\epsilon_2+\epsilon_3-\delta),\,
\frac12(-\epsilon_1+\epsilon_2+\epsilon_3+\delta),\,
\epsilon_2-\epsilon_3
\right\},
\\
\Uppi_4
&=
\left\{
\frac12(-\epsilon_1-\epsilon_2+\epsilon_3+\delta),\,
\frac12(\epsilon_1+\epsilon_2+\epsilon_3-\delta),\,
\epsilon_2-\epsilon_3,\,
\epsilon_1-\epsilon_2
\right\},
\\
\Uppi_5
&=
\left\{
\frac12(-\epsilon_1+\epsilon_2-\epsilon_3+\delta),\,
\frac12(\epsilon_1+\epsilon_2-\epsilon_3-\delta),\,
\epsilon_3,\,
\epsilon_1-\epsilon_2
\right\}.
\end{aligned}
\end{equation}
We denote the associated positive systems by $\Delta^{+}(\Uppi_{i})$ for $i=1,\ldots,5$. Among these positive systems, we are particularly interested in
$\Delta^{+}(\Pi_{1})$ and $\Delta^{+}(\Pi_{2})$, for reasons explained in
Section~\ref{subsec:real_forms_slmb}.

We introduce further notation which proves to be helpful in the discussion. We follow \cite[Section 10]{jakobsen1994full}. Set $\gg\coloneqq F(4)$. Choose nonzero odd root vectors $y_1,\ldots,y_8,z_1,\ldots,z_8$ in the one-dimensional odd root spaces 
\begin{equation}
\begin{aligned}
y_1&\in\gg^{\frac12(-\epsilon_1-\epsilon_2-\epsilon_3-\delta)},&
z_1&\in\gg^{\frac12(-\epsilon_1-\epsilon_2-\epsilon_3+\delta)},&
y_2&\in\gg^{\frac12(-\epsilon_1-\epsilon_2+\epsilon_3-\delta)},&
z_2&\in\gg^{\frac12(-\epsilon_1-\epsilon_2+\epsilon_3+\delta)},\\
y_3&\in\gg^{\frac12(-\epsilon_1+\epsilon_2-\epsilon_3-\delta)},&
z_3&\in\gg^{\frac12(-\epsilon_1+\epsilon_2-\epsilon_3+\delta)},&
y_4&\in\gg^{\frac12(-\epsilon_1+\epsilon_2+\epsilon_3-\delta)},&
z_4&\in\gg^{\frac12(-\epsilon_1+\epsilon_2+\epsilon_3+\delta)},\\
y_5&\in\gg^{\frac12(\epsilon_1-\epsilon_2-\epsilon_3-\delta)},&
z_5&\in\gg^{\frac12(\epsilon_1-\epsilon_2-\epsilon_3+\delta)},&
y_6&\in\gg^{\frac12(\epsilon_1-\epsilon_2+\epsilon_3-\delta)},&
z_6&\in\gg^{\frac12(\epsilon_1-\epsilon_2+\epsilon_3+\delta)},\\
y_7&\in\gg^{\frac12(\epsilon_1+\epsilon_2-\epsilon_3-\delta)},&
z_7&\in\gg^{\frac12(\epsilon_1+\epsilon_2-\epsilon_3+\delta)},&
y_8&\in\gg^{\frac12(\epsilon_1+\epsilon_2+\epsilon_3-\delta)},&
z_8&\in\gg^{\frac12(\epsilon_1+\epsilon_2+\epsilon_3+\delta)}
\end{aligned}
\end{equation}
as in \cite{jakobsen1994full}. Thus $y_i$ and $z_i$ have the same $\mathfrak{so}(7)$-weight and
opposite $\delta$-components. They span $\odd$. 

For the five systems $\Uppi_1,\ldots,\Uppi_5$ introduced above,
the corresponding positive odd subspaces are

\begin{equation}\label{eq::odd_root_space_different_systems_F(4)}
\begin{aligned}
\nn_{\bar1}^{+}(\Uppi_1)
&=
\operatorname{span}_{\CC}
\{z_1,z_2,z_3,z_4,z_5,z_6,z_7,z_8\},
\\
\nn_{\bar1}^{+}(\Uppi_2)
&=
\operatorname{span}_{\CC}
\{y_5,y_6,y_7,y_8,z_5,z_6,z_7,z_8\},
\\
\nn_{\bar1}^{+}(\Uppi_3)
&=
\operatorname{span}_{\CC}
\{y_6,y_7,y_8,z_4,z_5,z_6,z_7,z_8\},
\\
\nn_{\bar1}^{+}(\Uppi_4)
&=
\operatorname{span}_{\CC}
\{y_8,z_2,z_3,z_4,z_5,z_6,z_7,z_8\},
\\
\nn_{\bar1}^{+}(\Uppi_5)
&=
\operatorname{span}_{\CC}
\{y_7,y_8,z_3,z_4,z_5,z_6,z_7,z_8\}.
\end{aligned}
\end{equation}

We finally describe the odd reflections relating the five simple systems
$\Uppi_1,\ldots,\Uppi_5$. They form the chain
\begin{equation}
\Uppi_1
\longleftrightarrow
\Uppi_4
\longleftrightarrow
\Uppi_5
\longleftrightarrow
\Uppi_3
\longleftrightarrow
\Uppi_2.
\end{equation}
The corresponding positive odd systems differ successively by the
replacements
\[
z_1\longleftrightarrow y_8,\qquad
z_2\longleftrightarrow y_7,\qquad
z_3\longleftrightarrow y_6,\qquad
z_4\longleftrightarrow y_5.
\]
Indeed, the weight of $z_i$ is the negative of the weight of $y_{9-i}$.
Thus the odd reflection in the root of $z_i$ replaces this root by the
root of $y_{9-i}$.

\begin{lemma}\label{lemm::odd_reflections_F(4)}
The five simple systems are related by
\[
\begin{aligned}
r_{\frac12(-\epsilon_1-\epsilon_2-\epsilon_3+\delta)}(\Uppi_1)
&=\Uppi_4,
&
r_{\frac12(-\epsilon_1-\epsilon_2+\epsilon_3+\delta)}(\Uppi_4)
&=\Uppi_5,
\\
r_{\frac12(-\epsilon_1+\epsilon_2-\epsilon_3+\delta)}(\Uppi_5)
&=\Uppi_3,
&
r_{\frac12(-\epsilon_1+\epsilon_2+\epsilon_3+\delta)}(\Uppi_3)
&=\Uppi_2.
\end{aligned}
\]
\end{lemma}

\begin{proof} We denote the weights of $y_i$ and $z_j$ by $\wt(y_i)$ and
$\wt(z_j)$, respectively. Consider
\[
\wt(z_1)
=
\frac12(-\epsilon_1-\epsilon_2-\epsilon_3+\delta)
\]
and
$ \Uppi_1 =  \{\epsilon_1-\epsilon_2,\epsilon_2-\epsilon_3,
\epsilon_3,\wt(z_1)\}$.
We use the odd reflection formula \eqref{eq::odd-reflection}. First,
\[
r_{\wt(z_1)}\bigl(\wt(z_1)\bigr)
=
-\wt(z_1)
=
\wt(y_8).
\]
Moreover, $\wt(z_1)$ is orthogonal to $\epsilon_1-\epsilon_2$ and
$\epsilon_2-\epsilon_3$, so these two roots remain unchanged. On the
other hand, $(\wt(z_1),\epsilon_3)\neq0$, and hence
\[
r_{\wt(z_1)}(\epsilon_3)
=
\epsilon_3+\wt(z_1)
=
\frac12(-\epsilon_1-\epsilon_2+\epsilon_3+\delta)
=
\wt(z_2).
\]
Consequently,
\[
r_{\wt(z_1)}(\Uppi_1)
=
\{\epsilon_1-\epsilon_2,\epsilon_2-\epsilon_3,
\wt(z_2),\wt(y_8)\}
=
\Uppi_4.
\]
The remaining reflections are computed in the same way.
\end{proof}

\subsection{Real forms}\label{subsec:real_forms_slmb}
A real form of $\gg$ is a real Lie superalgebra $\gg_{\mathbb R}$ with $\gg_{\mathbb R}\otimes_{\mathbb R}\CC \cong\gg$, equivalently the fixed-point algebra of a parity-preserving conjugate-linear involution $\sigma$. For unitarity, we use instead a conjugate-linear anti-involution $\omega$, characterized by $\omega^2=\operatorname{id}$ and $\omega([X,Y])=[\omega(Y),\omega(X)]$ for homogeneous $X,Y\in\gg$ \cite{SchmidtDirac,Schmidt,SchmidtWalcher}. The two conventions are related by
\begin{equation}
\sigma_\omega(X)=-(-i)^{p(X)}\omega(X),\qquad
\gg_{\mathbb R}=\bigoplus_{\bar{a}\in \ZZ_{2}}\{X\in\gg_{\bar a}:\omega(X)=-i^{\bar{a}}X\}.
\label{eq:real-form-anti-involution}
\end{equation}

The real forms of $\sl(m\vert n)$, $\osp(m\vert 2n)$ and $F(4)$ were classified by Parker \cite{ParkerRealForms} and Serganova \cite{serganova1983classification}. We start with $\sl(m\vert n)$. For $m\neq n$, the real forms are $\mathfrak{sl}(m\vert n;\mathbb R)$, $\mathfrak{su}(p,q\vert r,s)$ with $p+q=m$ and $r+s=n$, and, when $m$ and $n$ are even, $\mathfrak{su}^{\ast}(m\vert n)$. For $m=n$, the corresponding real forms descend to $\mathfrak{psl}(n\vert n)$, which admits additional real forms associated with outer automorphisms. 

We explicitly realize the unitary family $\mathfrak{su}(p,q\vert r,s)$. Put $I_{a,b}\coloneqq\operatorname{diag}(I_a,-I_b)$ and $I_{p,q\vert r,s}\coloneqq\operatorname{diag}(I_{p,q},I_{r,s})$, and define
\begin{equation}
\omega_{p,q\vert r,s}(X)
\coloneqq I_{p,q\vert r,s}^{-1}X^\dagger I_{p,q\vert r,s},
\qquad X^\dagger\coloneqq\overline{X}^{\,t}.
\label{eq:unitary-anti-involution}
\end{equation}
The associated real form is denoted by $\mathfrak{su}(p,q\vert r,s)$ and has even part $\mathfrak{s}(\mathfrak{u}(p,q)\oplus\mathfrak{u}(r,s))$. Of particular importance are the isomorphic real forms $\mathfrak{su}(p,q\vert0,n)$ and $\mathfrak{su}(p,q\vert n,0)$, where $p+q=m$. Their anti-involutions agree on $\even$ and differ by a sign on $\odd$, while both have even part $\mathfrak{s}(\mathfrak{u}(p,q)\oplus\mathfrak{u}(n))$. Moreover, these real forms are compact if and only if $pq=0$. 

The relation between the conjugate-linear anti-involution and the choice of positive system requires some care. Let $\Delta^{+}$ be a positive system extending the fixed positive even system $\Delta_{\bar0}^{+}$, and set
$\nn_{\bar1}^{\pm}
\coloneqq
\bigoplus_{\alpha\in\Delta_{\bar1}^{+}}
\gg_{\pm\alpha}$. 
We call $\Delta^{+}$ $\omega$-adapted if there exist dual bases
$\{\partial_i\}_{i=1}^{N}$ of $\nn_{\bar1}^{+}$ and
$\{x_i\}_{i=1}^{N}$ of $\nn_{\bar1}^{-}$, where
$N=\dim\nn_{\bar1}^{+}$, such that
\begin{equation}
(\partial_i,x_j)=\delta_{ij},
\qquad
\omega(x_i)=\varepsilon\partial_i
\end{equation}
for all $i,j$ and a sign $\varepsilon\in\{\pm1\}$ independent of $i$. Equivalently, the Hermitian form
$
h_\omega(x,y)\coloneqq(\omega(x),y)
$
on $\nn_{\bar1}^{-}$ is definite. Indeed, $h_\omega(x_i,x_j)=\varepsilon\delta_{ij}$.

\begin{lemma}\label{lemm::adapted_systems_slmn}
The nonstandard positive system of Example~\ref{ex::systems} is
$\omega_{p,q\vert 0,n}$-adapted with sign $-1$ and
$\omega_{p,q\vert n,0}$-adapted with sign $+1$. The standard positive
system is adapted to either anti-involution if and only if $pq=0$.
Consequently, if $p,q>0$, the nonstandard positive system is the unique
adapted system among the standard and nonstandard positive systems.
\end{lemma}

\begin{proof}
We first consider $\omega_{p,q\vert 0,n}$. For $1\leq i\leq m$ and
$1\leq a\leq n$,
\[
\omega_{p,q\vert 0,n}(E_{i,m+a})
=
\begin{cases}
-E_{m+a,i},&i\leq p,\\
E_{m+a,i},&i>p,
\end{cases}
\qquad
\omega_{p,q\vert 0,n}(E_{m+a,i})
=
\begin{cases}
-E_{i,m+a},&i\leq p,\\
E_{i,m+a},&i>p.
\end{cases}
\]
For the standard positive system, $\nn_{\bar1}^{+}$ is spanned by the
$E_{i,m+a}$, and the corresponding form $h_{\omega_{p,q\vert 0,n}}$ has
sign $-1$ for $i\leq p$ and sign $+1$ for $i>p$. Hence it is definite if
and only if $pq=0$.

For the nonstandard positive system, set
\[
\partial_{ia}
=
\begin{cases}
E_{i,m+a},&i\leq p,\\
E_{m+a,i},&i>p,
\end{cases}
\qquad
x_{ia}
=
\begin{cases}
E_{m+a,i},&i\leq p,\\
-E_{i,m+a},&i>p.
\end{cases}
\]
Then $\{\partial_{ia}\}$ and $\{x_{ia}\}$ are dual bases of
$\nn_{\bar1}^{+}$ and $\nn_{\bar1}^{-}$, respectively, and
\[
(\partial_{ia},x_{jb})=\delta_{ij}\delta_{ab},
\qquad
\omega_{p,q\vert 0,n}(x_{ia})=-\partial_{ia}.
\]
Thus the nonstandard positive system is $\omega_{p,q\vert 0,n}$-adapted
with sign $-1$. Since $\omega_{p,q\vert n,0}$ differs on the odd part by
an overall sign, the same argument gives sign $+1$ for the nonstandard
system and shows again that the standard system is adapted precisely when
$pq=0$.
\end{proof}

The real forms $\su(p,q\vert 0,n)$ and $\su(p,q\vert n,0)$ are distinguished among all real forms by the following theorem.

\begin{theorem}[{\cite{neeb2011lie}}]
A real form of a fixed $\mathfrak{sl}(m\vert n)$ admits a nontrivial unitary representation if and only if it is isomorphic to $\mathfrak{su}(p,q\vert0,n)$ or $\mathfrak{su}(p,q\vert n,0)$ for some $p,q\in\ZZ_{\geq0}$ with $p+q=m$.
\end{theorem}

The real forms of $\osp(m\vert 2n)$ consist of the families
$\osp(p,q\vert 2n;\RR)$, where $p+q=m$, and, when $m=2r$ is even,
$\osp^{\ast}(2r\vert a,b)$, where $a+b=2n$. To realize them, we use explicit
matrix realizations of the corresponding conjugate-linear anti-involutions.

We first consider $\osp(p,q\vert 2n;\RR)$. Set
$\Omega_n\coloneqq\left(\begin{smallmatrix}0&I_n\\-I_n&0\end{smallmatrix}\right)$
and
$J_{p,q\vert 2n}\coloneqq\operatorname{diag}(I_{p,q},i\Omega_n)$.
The associated conjugate-linear anti-involution is
\begin{equation}
\upomega_{p,q\vert 2n}(X)
\coloneqq
J_{p,q\vert 2n}^{-1}X^\dagger J_{p,q\vert 2n},
\qquad
X^\dagger\coloneqq\overline{X}^{\,t}.
\end{equation}
Of particular interest are the isomorphic real forms
$\osp(m,0\vert 2n;\RR)$ and $\osp(0,m\vert 2n;\RR)$, corresponding to
$\upomega_{m,0\vert 2n}$ and $\upomega_{0,m\vert 2n}$, respectively. These
anti-involutions agree on the even part and differ by a sign on the odd
part. Equivalently,
$
\upomega_{m,0\vert 2n}(x)
=
(-1)^{p(x)}\upomega_{0,m\vert 2n}(x)
$
for every homogeneous $x\in\osp(m\vert 2n)$.

For $\osp(m\vert 2n)$, we consider the Case~A and Case~B positive systems
$\Delta_A^{+}$ and $\Delta_B^{+}$ introduced in
Example~\ref{ex::osp-systems}. Their compatibility with the
conjugate-linear anti-involutions $\upomega_{m,0\vert 2n}$ and $\upomega_{0,m\vert 2n}$ is as follows. 

\begin{lemma}\label{lemm::adapted_systems_osp}
The Case~A positive system of Example~\ref{ex::osp-systems} is
$\upomega_{m,0\vert 2n}$-adapted with sign $-1$ and
$\upomega_{0,m\vert 2n}$-adapted with sign $+1$. The Case~B positive
system is adapted to either anti-involution if and only if $m=1$, in which
case $\Delta_A^{+}=\Delta_B^{+}$.
\end{lemma}

\begin{proof}
It is enough to consider $\upomega_{m,0\vert 2n}$, since
$\upomega_{m,0\vert 2n}$ and $\upomega_{0,m\vert 2n}$ differ by a sign on
the odd part. Let $\alpha=\sigma\delta_a+\tau\epsilon_i$, where
$\sigma,\tau\in\{\pm1\}$, or $\alpha=\sigma\delta_a$ when $m$ is odd.
Any nonzero element of $\gg^{-\alpha}$ can be uniquely written as
\[
X=\begin{pmatrix}0&Z\\-\Omega_n Z^{T}&0\end{pmatrix},
\qquad Z\neq0.
\]
The weight condition gives $Z(i\Omega_n)=\sigma Z$, and hence
$\overline Z\Omega_n=i\sigma\overline Z$. Therefore
\[
(\upomega_{m,0\vert 2n}(X),X)
=
i\operatorname{tr}(\overline Z\Omega_n Z^{T})
=
-\sigma\operatorname{tr}(\overline Z Z^{T}).
\]
Since $\operatorname{tr}(\overline Z Z^T)>0$, we may choose
$x_\alpha\in\gg^{-\alpha}$ such that
$(\upomega_{m,0\vert 2n}(x_\alpha),x_\alpha)=-\sigma$ and set
$\partial_\alpha=-\sigma\upomega_{m,0\vert 2n}(x_\alpha)$. Then $(\partial_\alpha,x_\beta)=\delta_{\alpha\beta}$ and $\upomega_{m,0\vert 2n}(x_\alpha)=-\sigma\partial_\alpha$. In Case~A, every positive odd root has $\delta$-coefficient $+1$. Hence
\[
\upomega_{m,0\vert 2n}(x_\alpha)=-\partial_\alpha
\qquad
(\alpha\in\Delta_{A,\bar1}^{+}),
\]
so Case~A is $\upomega_{m,0\vert 2n}$-adapted with sign $-1$. In Case~B, one has
\[
\upomega_{m,0\vert 2n}(x_{\epsilon_i+\delta_a})
=
-\partial_{\epsilon_i+\delta_a},
\qquad
\upomega_{m,0\vert 2n}(x_{\epsilon_i-\delta_a})
=
+\partial_{\epsilon_i-\delta_a}.
\]
Thus, for $m\geq2$, both signs occur and the corresponding Hermitian form
is indefinite. If $m=1$, Case~A and Case~B coincide and the sign is
uniformly $-1$. Finally, passing from $\upomega_{m,0\vert 2n}$ to
$\upomega_{0,m\vert 2n}$ reverses all signs on the odd part, proving the
claim.
\end{proof}

For $m=2r$, let $a,b\geq0$ be even integers with $a+b=2n$, and put
\begin{equation}
R_r\coloneqq\bigoplus_{i=1}^{r}\Omega_1,\qquad
S_{a,b}\coloneqq\operatorname{diag}(I_{a/2,b/2},I_{a/2,b/2}),\qquad
J^{*}_{2r\vert a,b}\coloneqq\operatorname{diag}(iR_r,S_{a,b}).
\end{equation}
The conjugate-linear anti-involution
\begin{equation}
\upomega^{*}_{2r\vert a,b}(X)
\coloneqq(J^{*}_{2r\vert a,b})^{-1}X^\dagger J^{*}_{2r\vert a,b},
\qquad X^\dagger\coloneqq\overline X^{\,t},
\end{equation}
defines the real form $\mathfrak{osp}^{*}(2r\vert a,b)$, whose even part is
$\mathfrak{so}^{*}(2r)\oplus\mathfrak{sp}(a/2,b/2)$.
We retain the compact Cartan subalgebra introduced above and restrict attention to $(a,b)=(2n,0)$ and $(a,b)=(0,2n)$.

\begin{lemma}\label{lemm::adapted_systems_osp_star}
The Case~B positive system is
$\upomega^{*}_{2r\vert2n,0}$-adapted with sign $+1$ and
$\upomega^{*}_{2r\vert0,2n}$-adapted with sign $-1$.
The Case~A positive system is adapted to neither anti-involution.
\end{lemma}
\begin{proof}
The proof is parallel to that of the preceding lemma. The only difference is
that, for $\upomega^{*}_{2r\vert2n,0}$, the sign of
$(\upomega^{*}_{2r\vert2n,0}(x_\alpha),x_\alpha)$ is determined by the
$\epsilon_i$-coefficient of the odd root
$\alpha=\sigma\epsilon_i+\tau\delta_a$. Thus, after normalization,
\[
\upomega^{*}_{2r\vert2n,0}(x_\alpha)
=
\sigma\,\partial_\alpha.
\]
In Case~B, every positive odd root has $\epsilon_i$-coefficient $+1$, whereas
in Case~A both signs occur. Hence Case~B is
$\upomega^{*}_{2r\vert2n,0}$-adapted with sign $+1$, while Case~A is not
adapted. Since $\upomega^{*}_{2r\vert0,2n}$ differs from
$\upomega^{*}_{2r\vert2n,0}$ by a sign on the odd part, the adapted sign in
Case~B is reversed.
\end{proof}

The real forms $\osp(m,0\vert 2n;\RR)$, $\osp(0,m\vert 2n;\RR)$ and, for
$m=2r$, $\osp^{\ast}(2r\vert 2n,0)$ and $\osp^{\ast}(2r\vert 0,2n)$ are distinguished among the real forms of
$\osp(m\vert 2n)$ by the following result.

\begin{theorem}[\cite{neeb2011lie}]
Up to isomorphism, a real form of a fixed $\osp(m\vert 2n)$ admits nontrivial unitary
representations if and only if it is $\osp(m,0\vert 2n;\RR)$ or, when
$m=2r$, $\osp^{\ast}(2r\vert 2n,0)$.
\end{theorem}

Together with the corresponding result for $\sl(m\vert n)$ and $\osp(m\vert 2n)$, this reduces the study of unitarity to the real forms $\su(p,q\vert0,n)$ and $\su(p,q\vert n,0)$ in the special linear case, and to $\osp(m,0\vert 2n;\RR)$, $\osp(0,m\vert 2n;\RR)$ and, when $m=2r$, $\osp^{\ast}(2r\vert 2n,0)$ and $\osp^{\ast}(2r\vert 0,2n)$ in the orthosymplectic case. Moreover, every irreducible unitary representation of these real forms is of highest or lowest weight type (see \cite[Section~7]{neeb2011lie}). This reduces the analysis to highest and lowest weight representations. By Remark~\ref{rmk::HW_LW}, it is sufficient to consider highest weight representations.

Finally, we discuss the real forms of $F(4)$. By the classification of
Serganova, the complex Lie superalgebra $F(4)$ has, up to isomorphism,
four real forms, denoted by $F(4,p)$ for $0\leq p\leq3$. They are
determined by their even parts as follows:
\begin{equation}
\begin{array}{c|c}
\text{real form} & \text{even part}\\
\hline
F(4,0)
&
\mathfrak{so}(7)\oplus\mathfrak{sl}(2,\mathbb R)
\\
F(4,1)
&
\mathfrak{so}(1,6)\oplus\mathfrak{su}(2)
\\
F(4,2)
&
\mathfrak{so}(2,5)\oplus\mathfrak{su}(2)
\\
F(4,3)
&
\mathfrak{so}(3,4)\oplus\mathfrak{sl}(2,\mathbb R).
\end{array}
\end{equation}
Of particular interest to us are the real forms $F(4,0)$ and $F(4,2)$. We denote the corresponding conjugate-linear anti-involutions by $\omega_{0}$ and $\omega_{2}$, respectively. They may be described as follows.

For $F(4,0)$, the restriction of $\omega_0$ to the even part defines
$
F(4,0)_{\bar0}
=
\mathfrak{so}(7;\mathbb R)\oplus\mathfrak{sl}(2;\mathbb R).
$
With the normalization of the odd root vectors fixed above,
$\omega_0$ acts on the odd part by
\begin{equation}
\omega_0(y_i)=-z_{9-i},
\qquad
\omega_0(z_i)=-y_{9-i},
\qquad
1\leq i\leq8.
\end{equation}
There is an equivalent choice of conjugate-linear anti-involution
defining an isomorphic real form. Namely, for homogeneous
$X\in F(4)$, set
$
\omega_0'(X)
\coloneqq (-1)^{p(X)}\omega_0(X).
$
The anti-involutions $\omega_0$ and $\omega_0'$ agree on the even part and differ by a sign on the odd part. To distinguish the two real forms, we denote by $F(4,0)$ the real form associated with $\omega_0$ and by
$F(4,0)'$ the real form associated with $\omega_0'$.

For $F(4,2)$,  the restriction of $\omega_{2}$ to the even part corresponds to
$
F(4,2)_{\bar0}
=
\mathfrak{so}(2,5)\oplus\mathfrak{su}(2)$. With the normalization above, $\omega_{2}$ acts on the odd part by
\begin{equation}
\begin{aligned}
\omega_2(y_1)&=-z_8,
&
\omega_2(y_2)&=-z_7,
&
\omega_2(y_3)&=-z_6,
&
\omega_2(y_4)&=-z_5,
\\
\omega_2(y_5)&=z_4,
&
\omega_2(y_6)&=z_3,
&
\omega_2(y_7)&=z_2,
&
\omega_2(y_8)&=z_1,
\end{aligned}
\end{equation}
and
\begin{equation}
\begin{aligned}
\omega_2(z_1)&=y_8,
&
\omega_2(z_2)&=y_7,
&
\omega_2(z_3)&=y_6,
&
\omega_2(z_4)&=y_5,
\\
\omega_2(z_5)&=-y_4,
&
\omega_2(z_6)&=-y_3,
&
\omega_2(z_7)&=-y_2,
&
\omega_2(z_8)&=-y_1.
\end{aligned}
\end{equation}
Again, $\omega_2'$ agrees with
$\omega_2$ on $F(4)_{\bar0}$ and differs from it by a sign on
$F(4)_{\bar1}$. To distinguish both real forms, we
denote by $F(4,2)$ and $F(4,2)'$ the real forms associated with
$\omega_2$ and $\omega_2'$, respectively.

Combining these formulas with the description in~\eqref{eq::odd_root_space_different_systems_F(4)}, we obtain the following result.

\begin{lemma}\label{lemm::F4_adapted_positive_systems}
Among the five positive systems $\Delta^{+}(\Uppi_1),\ldots,\Delta^{+}(\Uppi_5)$ extending $\Delta_{\bar0}^{+}$, precisely $\Delta^{+}(\Uppi_1)$ is $\omega_0$-adapted with sign $-1$, whereas precisely $\Delta^{+}(\Uppi_2)$ is $\omega_2$-adapted with sign $+1$.\end{lemma}

The real forms $F(4,0)\cong F(4,0)'$ and $F(4,2)\cong F(4,2)'$ are distinguished among the real forms of $F(4)$.

\begin{lemma}
Up to isomorphism, neither $F(4,1)$ nor $F(4,3)$ admits a nontrivial
irreducible unitary representation.
\end{lemma}

\begin{proof} The proof follows from the work of Neeb--Salmasian \cite{neeb2011lie}. Although the real forms of $F(4)$ are excluded from the statement of \cite[Theorem~6.2.1]{neeb2011lie} for simplicity, the invariant-cone argument applies without change to the present case. Concretely, suppose that either $F(4,1)$ or $F(4,3)$ admits a nontrivial irreducible unitary representation. Since $F(4)$ is simple, such a representation is faithful, and hence the corresponding Lie supergroup is $\star$-reduced. By \cite[Proposition~6.1.2(i)]{neeb2011lie}, the invariant cone $\operatorname{Cone}(\mathcal{G})\subseteq\even$ is therefore pointed. Moreover, since $[\odd,\odd]=\even$ for $F(4)$, this cone is generating. By \cite[Section~5.3]{neeb2011lie}, it follows that $\even$ is quasihermitian, so every noncompact simple ideal of $\even$ must be of Hermitian type. However, $F(4,1)_{\bar0}=\su(2)\oplus\so(1,6)$ and $F(4,3)_{\bar0}=\sl(2,\RR)\oplus\so(3,4)$. The maximal compact subalgebras of $\so(1,6)$ and $\so(3,4)$ are $\so(6)$ and $\so(3)\oplus\so(4)$, respectively, both of which have trivial center. Hence neither $\so(1,6)$ nor $\so(3,4)$ is of Hermitian type, a contradiction. 
\end{proof}

In particular, every irreducible unitary representation of $F(4)$ is
of either highest or lowest weight type. It is enough to consider highest weight
representations.

Before recalling highest weight representations and their basic properties, we describe the maximal compact subalgebras of these real forms and their relation to the positive systems introduced above.

\subsection{Maximal compact subalgebras}\label{subsec::maximal_compact_subalgebra}

The real forms $\gg_{\bar0,\RR}$ of $\even$ associated with the conjugate-linear anti-involutions considered above are reductive real Lie algebras. We denote by $\kk\subseteq\gg_{\bar0,\RR}$ a maximal compact subalgebra and by $\kk^{\CC}$ its complexification. For the real forms of $\osp(m\vert 2n)$ and $\sl(m\vert n)$ relevant to us, we fix the following standard block-diagonally embedded maximal compact subalgebras.

\begin{table}[ht]
\centering
\begin{tabular}{|c|c|c|}
\hline
$\gg$ & real form $\gg_{\RR}$ & maximal compact subalgebra $\kk$ \\
\hline
$\sl(m\vert n)$
&
$\su(p,q\vert 0,n)$, $\su(p,q\vert n,0)$
&
$\mathfrak{s}\bigl(
\mathfrak{u}(p)\oplus
\mathfrak{u}(q)\oplus
\mathfrak{u}(n)
\bigr)$
\\
\hline
$\osp(m\vert 2n)$
&
$\osp(m,0\vert 2n)$, $\osp(0,m\vert 2n)$
&
$\so(m)\oplus\mathfrak{u}(n)$
\\
\hline
$\osp^{\ast}(2r\vert 2n)$
&
$\osp^{\ast}(2r\vert 2n,0)$,
$\osp^{\ast}(2r\vert 0,2n)$
&
$\mathfrak{u}(r)\oplus\mathfrak{usp}(2n)$
\\
\hline
$F(4)$
&
$\begin{array}{c}
F(4,0)\\
F(4,2)
\end{array}$
&
$\begin{array}{c}
\so(7)\oplus\mathfrak{so}(2)\\
\mathfrak{so}(2)\oplus\mathfrak{so}(5)\oplus\su(2)
\end{array}$
\\
\hline
\end{tabular}
\end{table}
\noindent
Note that the two real forms in each pair have the same even real form and therefore the same maximal compact subalgebra. They differ in the action of the conjugate-linear anti-involution on the odd part.

For our choices of Cartan subalgebras, the equal-rank condition is satisfied:
\begin{equation}
\hh\subseteq\kk^{\CC}\subseteq\even\subseteq\gg.
\end{equation}
Thus $\hh$ is simultaneously a Cartan subalgebra of $\gg$, $\even$, and $\kk^{\CC}$. An even root $\alpha\in\Delta_{\bar0}$ is called compact if $\gg^\alpha\subseteq\kk^{\CC}$ and noncompact otherwise. We denote the compact roots by $\Delta_{c} \subset \Delta_{\bar{0}}$ and the noncompact roots by $\Delta_{nc}$. Fixing a positive system $\Delta^{+}$ induces a positive system on $\Delta_{c}$ and $\Delta_{nc}$ by $\Delta_{c}^{+} \coloneqq \Delta_{c}\cap \Delta^{+}$ and $\Delta_{nc}^{+} \coloneqq \Delta_{nc}\cap \Delta^{+}$. We denote the associated Weyl vectors by $\rho_{c}$ and $\rho_{nc}$, respectively. 

The decomposition of the odd part as a $\kk^{\CC}$-representation is important and explains the choice of the positive systems above. In the special linear case, when $p,q>0$, the odd part has four simple $\kk^{\CC}$-constituents. The standard and nonstandard positive systems introduced above select different halves of this decomposition. If $p=0$ or $q=0$, the full odd part has only two simple $\kk^{\CC}$-constituents. A detailed analysis can be found in \cite{SchmidtDirac}.

For $\gg=\osp(m\vert 2n)$ with $m\neq2$, the odd part $\odd$ decomposes as the direct sum of two simple $\kk^{\CC}$-representations, where $\kk$ is a maximal compact subalgebra of either $\osp(m,0\vert 2n)\cong\osp(0,m\vert 2n)$ or, when $m=2r$, $\osp^{\ast}(2r\vert 2n,0)\cong\osp^{\ast}(2r\vert 0,2n)$. The relation between the Case~A/B positive systems and the simple $\kk^{\CC}$-summands of $\odd$ depends on the real form. For $\osp(m,0\vert2n;\RR)$ with $m\neq2$, one has $\kk^{\CC}=\so(m)\oplus\gl(n)$, and the two simple summands of $\odd$ are $\nn_{A,\bar1}^{+}$ and $\nn_{A,\bar1}^{-}$. The Case~B positive odd subspace mixes these two summands. For $\osp^{\ast}(2r\vert2n,0)$, one has $\kk^{\CC}=\gl(r)\oplus\sp(2n)$, and the two simple summands are $\nn_{B,\bar1}^{+}$ and $\nn_{B,\bar1}^{-}$. In this case, the Case~A positive odd subspace mixes the two summands. For $m=2$ and $\osp(2,0\vert 2n)\cong \osp(0,2\vert 2n)$, the $\so(2)$-representation $\CC^{2}$ is reducible, so $\odd$
decomposes into four simple $\kk^{\CC}$-representations. Correspondingly,
$\nn_{A,\bar1}^{+}$ and $\nn_{A,\bar1}^{-}$ are each sums of two simple
constituents. Thus, in contrast to the case $m\neq2$, the decomposition
of $\odd$ into simple $\kk^{\CC}$-representations is finer than the decomposition
determined by the Case~A polarization. Nevertheless, for the purposes of
the unitary highest weight theory considered here, we restrict to the
Case~A and Case~B positive systems appearing in the classification
of \cite{jakobsen1994full}.

For $F(4)$, we only consider the positive systems $\Delta^{+}(\Uppi_1)$ and
$\Delta^{+}(\Uppi_2)$, since they are the only ones among the five
positive systems for which the corresponding $\nn_{\bar{1}}^{+}(\Uppi_{i})$ is stable
under the complexification of a maximal compact subalgebra. More
precisely, $\nn_{\bar1}^{+}(\Uppi_1)$ is $\kk^{\CC}$-stable only for
the maximal compact subalgebra of $F(4,0)$, whereas
$\nn_{\bar1}^{+}(\Uppi_2)$ is $\kk^{\CC}$-stable only for the maximal
compact subalgebra of $F(4,2)$. None of
$\nn_{\bar1}^{+}(\Uppi_3)$, $\nn_{\bar1}^{+}(\Uppi_4)$, and
$\nn_{\bar1}^{+}(\Uppi_5)$ is stable under either of these maximal
compact subalgebras.

To obtain explicit formulas for the dimensions of weight spaces of irreducible unitary highest weight representations of $\even$, we will later use a result of Enright and Willenbring \cite{Enright_Willenbring}. We therefore introduce the root-theoretic data entering the formula.

Let $\nu\in\hh^{\ast}$ be dominant integral with respect to $\Delta_{c}^{+}$, and let $L_{\kk^{\CC}}(\nu)$ denote the simple $\kk^{\CC}$-representation of highest weight $\nu$. Following \cite{Enright_Willenbring}, define
\begin{equation}
\Psi_{\nu}
=
\left\{
\beta\in\Delta_{\bar0}:
(\nu+\rho_{\bar0},\beta)=0
\right\}.
\end{equation}
Let $\Phi_{\nu}\subseteq\Delta_{\mathrm{nc}}^{+}$ be the set of roots $\alpha$ satisfying
\begin{enumerate}
\item $\langle\nu+\rho_{\bar0},\alpha^{\vee}\rangle\in\ZZ_{+}$;
\item $(\alpha,\beta)=0$ for every $\beta\in\Psi_{\nu}$;
\item if $\Psi_{\nu}$ contains a long root in the simple factor containing $\alpha$, then $\alpha$ is short.
\end{enumerate}
The associated reflection subgroup and root system are
$
W_{\nu}
=
\langle s_{\alpha}:\alpha\in\Phi_{\nu}\rangle$ and $\Delta_{\nu}
=
\left\{
\alpha\in\Delta_{\bar0}:
s_{\alpha}\in W_{\nu}
\right\},
$
where $s_{\alpha}$ denotes the usual even reflection.
Set
$
\Delta_{\nu}^{+}
=
\Delta_{\nu}\cap\Delta_{\bar0}^{+},$ and $
\Delta_{\nu,c}^{+}
=
\Delta_{\nu}^{+}\cap\Delta_{c}^{+},
$
and define
\begin{equation}
W_{\nu}^{\kk}
=
\left\{
w\in W_{\nu}:
w^{-1}\Delta_{\nu,c}^{+}
\subseteq
\Delta_{\nu}^{+}
\right\}.
\end{equation}
Finally, the relevant length function on $W_{\nu}$ is
\begin{equation}
\ell_{\nu}(w)
=
\#
\left\{
\alpha\in\Delta_{\nu}^{+}:
w\alpha\in-\Delta_{\nu}^{+}
\right\}.
\end{equation}
Thus $\ell_{\nu}$ is the Coxeter length associated with the root system $\Delta_{\nu}$ and, in general, need not coincide with the restriction to $W_{\nu}$ of the length function on $W$. 

\subsection{Highest and lowest weight representations and atypicality}
\label{subsec::highest_weight_representations_and_atypicality}

Let $\gg$ be either $\sl(m\vert n)$ with $m\neq n, \gl(m\vert n), \osp(m\vert 2n)$ or $F(4)$. Let $\Delta^+$ be a positive system with triangular decomposition $\gg=\nn^-\oplus\hh\oplus\nn^+$. A representation $M$ of $\gg$ is a \emph{highest weight representation} of highest weight $\Lambda\in\hh^{\ast}$ if it contains a nonzero homogeneous vector $v_\Lambda$ such that
\begin{equation}
\nn^+v_\Lambda=0,\qquad
Hv_\Lambda=\Lambda(H)v_\Lambda\quad\text{for all }H\in\hh,\qquad
\UE(\gg)v_\Lambda=M.
\end{equation}
A \emph{lowest weight representation} of lowest weight $\Lambda$ is defined by replacing $\nn^+$ with $\nn^-$. The vector $v_\Lambda$ is called a highest, respectively lowest, weight vector. Throughout this article, we assume that the highest weight vector is even. This entails no restriction, since a highest weight representation with an odd highest weight vector may be replaced by the same representation with the $\ZZ_2$-grading reversed.

We recall some elementary properties of highest weight representations that will be used throughout. Let $M$ be a highest weight $\gg$-representation of highest weight $\Lambda$ relative to a positive system $\Delta^{+}$, with triangular decomposition
$
\gg=\nn^{-}\oplus\hh\oplus\nn^{+}.
$
Then $M=\UE(\nn^{-})v_{\Lambda}$, where $v_{\Lambda}$ is a highest weight vector. For $\mu\in\hh^{\ast}$, the corresponding weight space is
$
M^{\mu}\coloneqq\{v\in M\mid hv=\mu(h)v\text{ for all }h\in\hh\}.
$
The PBW theorem implies that $M$ is a weight representation, that is,
$
M=\bigoplus_{\mu\in\hh^{\ast}}M^{\mu},
$
with $\dim M^{\mu}<\infty$ for all $\mu$ and $\dim M^{\Lambda}=1$. We denote the set of weights of $M$ by
$
\mathcal{P}_{M}\coloneqq\{\mu\in\hh^{\ast}\mid M^{\mu}\neq0\}.
$
Every $\mu\in\mathcal{P}_{M}$ satisfies
$
\Lambda-\mu\in\ZZ_{\geq0}\Delta^{+}.
$ We use the partial order on $\hh^{\ast}$ induced by $\Delta^{+}$, defined by
$
\mu\leq\nu
$
if
$
\nu-\mu\in\ZZ_{\geq0}\Delta^{+}.
$
In particular, $\mu\leq\Lambda$ for every $\mu\in\mathcal{P}_{M}$.

Every nonzero quotient of $M$ is again a highest weight representation of highest weight $\Lambda$. Moreover, $M$ has a unique maximal proper subrepresentation and hence a unique simple quotient, denoted by $L(\Lambda;\Delta^{+})$. The simple highest weight representation of highest weight $\Lambda$ is unique up to isomorphism. These properties follow immediately from the structure of two fundamental classes of highest weight representations, namely Verma modules and Kac modules.

\begin{example}[Verma modules] \label{ex::Verma_modules} Fix a positive system $\Delta^{+}$ and $\Lambda \in \hh^{\ast}$. Let $\CC _\Lambda$ be the one-dimensional representation of $\bb=\hh\oplus\nn^+$ on which $\nn^+$ acts trivially and each $H\in\hh$ acts by multiplication by $\Lambda(H)$. The \emph{Verma module} of highest weight $\Lambda$ is
\begin{equation}
M(\Lambda;\Delta^+)
\coloneqq
\UE(\gg)\otimes_{\UE(\bb)}\CC _\Lambda.
\end{equation}
It is defined for every positive system and is the universal highest weight representation of highest weight $\Lambda$: every highest weight representation of highest weight $\Lambda$ with respect to $\Delta^{+}$ is a quotient of $M(\Lambda;\Delta^+)$. Moreover, $M(\Lambda;\Delta^+)$ has a unique maximal proper subrepresentation and hence a unique simple quotient, denoted by $L(\Lambda;\Delta^+)$. Consequently, every simple highest weight representation of highest weight $\Lambda$ is isomorphic, up to parity reversal, to $L(\Lambda;\Delta^+)$.
\end{example}

\begin{example}[Kac modules] If $\gg=\sl(m\vert n)$, equipped with the standard positive system, or $\gg=\osp(2\vert 2n)$, equipped with the Case~B positive system, then $\gg$ admits a consistent $\ZZ$-grading $\gg=\gg_{-1}\oplus\gg_0\oplus\gg_{+1}$ with $\gg_0=\even$, $\gg_{+1}=\nn_{\bar1}^{+}$, and $\gg_{-1}=\nn_{\bar1}^{-}$. In particular, $\gg_{\pm1}$ are simple $\even$-representations under the adjoint action and are abelian. Consequently, one may induce directly from $\even\oplus\gg_{+1}$, with $\gg_{+1}$ acting trivially, yielding Kac modules. The grading defines the parabolic subsuperalgebra
$
\mathfrak p^+\coloneqq\even\oplus\gg_{+1}.
$
Let $V$ be a representation of $\even$, extended to $\mathfrak p^+$ by letting $\gg_{+1}$ act trivially. The associated \emph{Kac module} is
\begin{equation}
K(V)
\coloneqq
\UE(\gg)\otimes_{\UE(\mathfrak p^+)}V.
\end{equation}
As $\even$-representations 
$
K(V)\big\vert_{\even}
\cong
\bigwedge(\gg_{-1})\otimes V.
$ 
Suppose that $V=L_{0}(\Lambda)$ is the simple highest weight representation of $\even$ of highest weight $\Lambda$. Then $K(L_0(\Lambda))$ is a highest weight representation with respect to $\Delta_{\mathrm{st}}^{+}$ for $\sl(m\vert n)$ and with respect to $\Delta_{B}^{+}$ for $\osp(2\vert 2n)$, and is a quotient of $M(\Lambda;\Delta_{\mathrm{st}}^{+})$ or $M(\Lambda;\Delta_{B}^{+})$. Both representations have the same unique simple quotient $L(\Lambda;\Delta_{\mathrm{st}}^{+})$ or $L(\Lambda;\Delta_{B}^{+})$. Consequently, every simple highest weight representation with respect to the standard positive system arises as the unique simple quotient of a Kac module.
\end{example}

The significance of these two classes of representations is that their simple quotients include the irreducible unitary representations considered below.

\begin{theorem}[{\cite[Section~7]{neeb2011lie}}]
\label{theorem::unitary-highest-lowest-weight} \begin{enumerate}
 \item[(a)]  Let $p,q\in\ZZ_{\geq0}$ satisfy $p+q=m$ and $pq\neq 0$. Every irreducible unitary representation of $\sl(m\vert n)$ associated with $\mathfrak{su}(p,q\vert0,n)$ is of highest weight type. Every nontrivial irreducible unitary representation associated with $\mathfrak{su}(p,q\vert n,0)$ is of lowest weight type.
 \item[(b)] 
 Every irreducible unitary representation of
 $\osp(m\vert2n)$ associated with $\osp(m,0 \vert 2n; \RR)$ is of highest
 weight type, whereas every nontrivial irreducible unitary representation
 associated with $\osp(0,m\vert2n;\RR)$ is of lowest weight type. If
 $m=2r$, every nontrivial irreducible unitary representation associated
 with $\osp^{\ast}(2r\vert0,2n)$ is of highest weight type, whereas every
 nontrivial irreducible unitary representation associated with
 $\osp^{\ast}(2r\vert2n,0)$ is of lowest weight type. 
 \item[(c)]
Every irreducible unitary representation of $F(4)$ associated with $F(4,0)$ is
of highest weight type, whereas
every nontrivial irreducible unitary representation associated with $F(4,0)'$ is of
lowest weight type. Likewise, every irreducible unitary
representation of $F(4,2)$ is of highest weight type, whereas every nontrivial irreducible
 unitary representation of $F(4,2)'$ is of lowest weight type.
\end{enumerate}
\end{theorem}

\begin{remark}\label{rmk::HW_LW}
There is a one-to-one correspondence between unitary highest weight representations of highest weight $\Lambda$ with respect to a conjugate-linear anti-involution $\omega$ and unitary lowest weight representations of lowest weight $-\Lambda$ with respect to the sign-twisted anti-involution $\omega^{\vee}=(-1)^{p(\cdot)}\omega$. More precisely, let $(L(\Lambda),\langle\cdot,\cdot\rangle)$ be a unitary highest weight representation
with respect to a conjugate-linear anti-involution $\omega$. Since the
weight spaces of $L(\Lambda)$ are finite-dimensional, its contragredient
dual is defined by
\[
 L(\Lambda)^{\vee}
 \coloneqq
 \bigoplus_{\mu\in\mathcal{P}_{L(\Lambda)}}
 \bigl(L(\Lambda)^{\mu}\bigr)^{*}.
\]
For homogeneous $x\in\gg$ and $f\in L(\Lambda)^{\vee}$, the
contragredient action is given by
\[
 (xf)(v)
 =
 -(-1)^{p(x)p(f)}f(xv).
\]
With this action, $L(\Lambda)^{\vee}$ is a lowest weight representation
of lowest weight $-\Lambda$.

We use the convention that the Hermitian form on $L(\Lambda)$ is
conjugate-linear in the first argument and linear in the second. The
Riesz map $R\colon L(\Lambda)\to L(\Lambda)^{\vee}$, defined by
$R(v)(w)\coloneqq\langle v,w\rangle$, is conjugate-linear. Since
distinct weight spaces are orthogonal and each weight space is
finite-dimensional, $R$ takes values in the contragredient dual and is
a bijection. We therefore define a positive definite Hermitian form on
$L(\Lambda)^{\vee}$ by
$\langle R(v),R(w)\rangle_{\vee}\coloneqq\langle w,v\rangle$.

For homogeneous $x,v$, unitarity of $L(\Lambda)$ gives
$xR(v)=-(-1)^{p(x)p(v)}R(\omega(x)v)$. Define the conjugate-linear
anti-involution $\omega^{\vee}$ by
$\omega^{\vee}(x)\coloneqq(-1)^{p(x)}\omega(x)$ for homogeneous
$x\in\gg$. Then, for homogeneous $x,v,w$,
\[
\begin{aligned}
 \langle xR(v),R(w)\rangle_{\vee}
 &=
 -(-1)^{p(x)p(v)}
 \langle w,\omega(x)v\rangle,\\
 \langle R(v),\omega^{\vee}(x)R(w)\rangle_{\vee}
 &=
 -(-1)^{p(x)+p(x)p(w)}
 \langle w,\omega(x)v\rangle.
\end{aligned}
\]
If $\langle w,\omega(x)v\rangle\neq0$, evenness of the Hermitian form
implies $p(w)=p(x)+p(v)\pmod 2$, and hence
$p(x)+p(x)p(w)=p(x)p(v)\pmod 2$. Thus
$\langle xR(v),R(w)\rangle_{\vee}
=\langle R(v),\omega^{\vee}(x)R(w)\rangle_{\vee}$.

Consequently, $L(\Lambda)^{\vee}$ is a unitary lowest weight
representation of lowest weight $-\Lambda$ with respect to
$\omega^{\vee}$. Applying the same construction once more recovers
$L(\Lambda)$, since $L(\Lambda)^{\vee\vee}\cong L(\Lambda)$ and
$(\omega^{\vee})^{\vee}=\omega$. Hence the correspondence is
bijective.
\end{remark}

\begin{remark} \label{rmk::infinite_sl}
Consider $\gg=\sl(m\vert n)$. If $pq=0$, the real form is compact and every irreducible unitary representation is finite-dimensional, and hence of highest and lowest weight type. If $pq\neq0$, the real forms $\su(p,q\vert0,n)$ and $\su(p,q\vert n,0)$ contain the noncompact real form $\su(p,q)$ in their even part. Since $\su(p,q)$ admits no nontrivial finite-dimensional unitary representations, neither real form admits nontrivial finite-dimensional irreducible unitary representations. Indeed, every finite-dimensional unitary representation is trivial on the noncompact simple Lie algebra $\su(p,q)$ and hence also on its complexification $\sl(m)$. Since $[\sl(m),\odd]=\odd$, it follows that $\odd$ acts trivially. Finally, as $[\odd,\odd]=\even$, the entire even part acts trivially as well. Hence the representation is trivial. Thus, apart from the trivial representation, the cases $pq=0$ and $pq\neq0$ correspond to the finite- and infinite-dimensional unitary representation theories, respectively.
\end{remark}

\begin{remark}
For the real form $\mathfrak{osp}(m,0\vert 2n;\mathbb R)$ with $n>0$, there are no nontrivial finite-dimensional unitary representations. Indeed, its even part is $\mathfrak{so}(m)\oplus\mathfrak{sp}(2n,\mathbb R)$, and $\mathfrak{sp}(2n,\mathbb R)$ admits no nontrivial finite-dimensional unitary representations. The same line of argumentation as for $\sl(m\vert n)$ applies here. Consequently, among the real forms of $\mathfrak{osp}(m\vert 2n)$ relevant to unitarity, finite-dimensional unitary representations can occur only for the real forms of type $\mathfrak{osp}^{\ast}(2r\vert 2n)$.
\end{remark}

In this article, we work primarily with highest weight representations. For $\sl(m\vert n)$, if $pq\neq0$, these are associated with the real forms $\mathfrak{su}(p,q\vert0,n)$, while the corresponding results for lowest weight representations associated with $\mathfrak{su}(p,q\vert n,0)$ are obtained analogously and will be stated without separate proofs. If $pq=0$, the distinction between the highest and lowest weight cases is immaterial. For $\osp(m\vert2n)$, we consider highest weight representations associated with $\osp(m,0\vert2n;\RR)$ and, when $m=2r$, with $\osp^{\ast}(2r\vert0,2n)$. For $F(4)$, we consider unitary highest weight representations associated with the real forms $F(4,2)$ and $F(4,0)$. The corresponding statements for lowest weight representations are again obtained analogously. For a fixed positive system $\Delta^+$, we denote the simple highest weight representation of highest weight $\Lambda$ by $L(\Lambda;\Delta^+)$.

We next explain why the standard and nonstandard positive systems for $\sl(m\vert n)$, the Case~A and Case~B positive systems for $\osp(m\vert 2n)$, and the positive systems associated to $\Uppi_{1}$ and $\Uppi_{2}$ of $F(4)$ are distinguished in the study of unitary highest weight representations. We discuss this first for $\sl(m\vert n)$. The same line of argument applies to $\osp(m\vert 2n)$ and $F(4)$.

Assume that $p,q>0$, and let $\mathfrak{k}^{\mathbb C}\coloneqq\mathfrak{s}\bigl(\mathfrak{gl}(p)\oplus\mathfrak{gl}(q)\oplus\mathfrak{gl}(n)\bigr)\subseteq\even$ be the block-diagonally embedded complexification of the maximal compact even subalgebra. A positive system extending $\Delta_{\bar0}^{+}$ is called adapted to the real form if $\mathfrak{n}_{\bar1}^{+}$ is $\mathfrak{k}^{\mathbb C}$-stable. For a highest weight vector $v_\Lambda$, the inclusion $[\mathfrak{k}^{\mathbb C},\mathfrak{n}_{\bar1}^{+}]\subseteq\mathfrak{n}_{\bar1}^{+}$ then implies
$\mathfrak{n}_{\bar1}^{+}U(\mathfrak{k}^{\mathbb C})v_\Lambda=0$.
Thus the odd highest weight conditions hold on the entire $\mathfrak{k}^{\mathbb C}$-representation generated by $v_\Lambda$. In a unitary highest weight representation, this representation is the finite-dimensional irreducible highest $\mathfrak{k}^{\mathbb C}$-type. Adaptedness therefore makes the odd annihilation conditions compatible with this compact type. This is a natural condition. Indeed, it is the odd analogue of the admissibility condition used in the highest weight Harish--Chandra theory of Hermitian Lie supergroups \cite{Chuah_Fioresi,CVF,arXiv:2405.16251}. Already for an ordinary Hermitian Lie algebra $\gg_{\mathbb R}=\kk\oplus\mathfrak p$, an admissible positive system is characterized by a $\kk^{\mathbb C}$-stable decomposition $\mathfrak p^{\mathbb C}=\mathfrak p^{+}\oplus\mathfrak p^{-},$ where $\mathfrak p^{+}$ is the sum of the positive noncompact root spaces. This is precisely the structure underlying highest weight Harish--Chandra representations: their distinguished highest $\kk^{\mathbb C}$-type is annihilated by $\mathfrak p^{+}$, and the representation is generated from this compact type by $\mathfrak p^{-}$. The admissibility condition extends this classical compatibility between the triangular decomposition and the $K$-type decomposition to Lie superalgebras. In the present setting, the corresponding additional requirement on the odd part is exactly the $\kk^{\mathbb C}$-stability of $\nn_{\bar1}^{+}$.

For the fixed even positive system $\Delta_{\bar0}^{+}$, the adapted positive systems are precisely the standard, minus-standard, and nonstandard systems. The minus-standard system reverses the standard odd polarization and corresponds to the analogous lowest weight picture, so it yields no additional case for the highest weight analysis pursued below. We therefore restrict attention to the standard and nonstandard systems and study their relation, with the latter defined by the fixed decomposition $m=p+q$. If $pq=0$, the nonstandard system coincides with one of the two standard orientations, and it suffices to consider the standard positive system.

The highest weight of a simple representation depends on the choice of positive system. Its transformation under an odd reflection is given by the following standard lemma \cite[Lemma~1.40]{cheng2012dualities}.

\begin{lemma}
\label{lemma::relating_highest_weight_representations}
Let $\Uppi$ be a simple system, let $\theta\in\Uppi\cap\Delta_{\bar1}$ be isotropic, and let $\Uppi_\theta=r_\theta(\Uppi)$. Suppose that $L(\Lambda;\Delta^+)$ is simple. Then it is a highest weight representation relative to $\Uppi_\theta$, with highest weight
\begin{equation*}
\Lambda_\theta=
\begin{cases}
\Lambda,&(\Lambda,\theta)=0,\\
\Lambda-\theta,&(\Lambda,\theta)\neq0.
\end{cases}
\label{eq::odd-reflection-highest-weight}
\end{equation*}
In the second case, the parity of the highest weight vector is reversed.
\end{lemma}

\begin{remark}\label{rmk::shifted_odd_reflections}
Let $\beta$ be an isotropic odd simple root in $\Uppi$ and let $\lambda'$ be the highest weight obtained from $\lambda$ by the odd reflection $r_\beta$. Then
\[
\lambda'+\rho'
=
\begin{cases}
\lambda+\rho, & (\lambda+\rho,\beta)\neq 0,\\[1mm]
\lambda+\rho+\beta, & (\lambda+\rho,\beta)=0,
\end{cases}
\]
where $\rho'$ is the Weyl vector associated with the reflected positive system. 
\end{remark}

As a consequence of the preceding lemma, for every
$\Lambda\in\hh^{\ast}$ there exists a unique $\Lambda'\in\hh^{\ast}$
such that, up to parity reversal,
\begin{equation}
\label{eq::equivalence_simples_odd_reflections}
\begin{aligned}
L(\Lambda;\Delta_{\st}^{+})
&\cong
L(\Lambda';\Delta_{\nst}^{+}),
\\
L(\Lambda;\Delta_{A}^{+})
&\cong
L(\Lambda';\Delta_{B}^{+}),
\\
L(\Lambda(\Uppi_{i});\Delta^{+}(\Uppi_{i}))
&\cong
L(\Lambda(\Uppi_{j});\Delta^{+}(\Uppi_{j})),
\end{aligned}
\end{equation}
according as $\gg=\sl(m\vert n)$ or $\gl(m\vert n)$,
$\gg=\osp(m\vert2n)$, or $\gg=F(4)$, respectively. In either case, the resulting correspondence $\Lambda\mapsto\Lambda'$ or $\Lambda(\Uppi_{i})\mapsto \Lambda(\Uppi_{j})$ is bijective. Although the highest weight changes under the passage between the two positive systems, unitarity is preserved up to parity reversal.

\begin{proposition}\label{prop::transport_of_unitarity}
For $\gg=\sl(m\vert n)$ or $\gl(m\vert n)$, the representation $L(\Lambda;\Delta_{\st}^{+})$ is unitary if and only if $L(\Lambda';\Delta_{\nst}^{+})$ is unitary. For $\gg=\osp(m\vert2n)$, the representation $L(\Lambda;\Delta_{A}^{+})$ is unitary if and only if $L(\Lambda';\Delta_{B}^{+})$ is unitary. For $F(4)$, the representation $L(\Lambda(\Uppi_{i});\Delta^{+}(\Uppi_{i}))$ is unitary if and only if $L(\Lambda(\Uppi_{j});\Delta^{+}(\Uppi_{j}))$ is unitary for all $i,j=1,\ldots, 5.$ In either case, the two representations are unitarily equivalent up to parity reversal.
\end{proposition}

\begin{proof}
By \eqref{eq::equivalence_simples_odd_reflections}, there exist
$\varepsilon\in\ZZ_2$ and an isomorphism of $\gg$-representations
\begin{equation*}
\begin{aligned}
\varphi:
L(\Lambda;\Delta_{\st}^{+})
&\longrightarrow
\Pi^\varepsilon L(\Lambda';\Delta_{\nst}^{+}),
\\
\varphi:
L(\Lambda;\Delta_{A}^{+})
&\longrightarrow
\Pi^\varepsilon L(\Lambda';\Delta_{B}^{+}),
\\
\varphi:
L(\Lambda(\Uppi_{i});\Delta^{+}(\Uppi_i))
&\longrightarrow
\Pi^\varepsilon L(\Lambda(\Uppi_{j});\Delta^{+}(\Uppi_j)),
\end{aligned}
\end{equation*}
according as $\gg=\sl(m\vert n)$ or $\gl(m\vert n)$,
$\gg=\osp(m\vert2n)$, or $\gg=F(4)$, respectively. Since parity reversal preserves unitarity, a unitary structure on either representation can be transported along $\varphi$ to the other. The converse follows in the same way.
\end{proof}

Outside the typical case, there is generally no uniform closed
expression relating the highest weights associated with two positive
systems. Although each individual odd reflection has a simple
transformation rule, the resulting highest weight depends on which
vanishing conditions occur at the intermediate stages. These
conditions remain case-dependent even under the assumption of
unitarity. For typical representations, however, the shifted highest
weight is invariant.

\begin{lemma}\label{lemm::typical_shifted_highest_weight}
Let $\Delta^{+}$ and $\widetilde{\Delta}^{+}$ be positive systems
connected by a sequence of odd reflections, and let $\rho$ and
$\widetilde{\rho}$ be their respective Weyl vectors. Suppose that
$
L(\Lambda;\Delta^{+})
\cong
\Pi^\varepsilon
L(\widetilde{\Lambda};\widetilde{\Delta}^{+})
$
for some $\varepsilon\in\ZZ_2$. If $\Lambda$ is typical, then
\[
\Lambda+\rho
=
\widetilde{\Lambda}+\widetilde{\rho}.
\]
\end{lemma}

\begin{proof}
It suffices to consider a single odd reflection with respect to an
isotropic simple root $\alpha$. Since $\Lambda$ is typical and
$
(\rho,\alpha)=\frac12(\alpha,\alpha)=0,
$
we have $(\Lambda,\alpha)\neq0$. The odd-reflection rule therefore
gives
\[
\widetilde{\Lambda}=\Lambda-\alpha,
\qquad
\widetilde{\rho}=\rho+\alpha.
\]
Consequently,
$
\widetilde{\Lambda}+\widetilde{\rho}
=
\Lambda+\rho.
$
The assertion follows by applying this argument successively along
the sequence of odd reflections.
\end{proof}

In particular, in the typical case,
\begin{equation}
\Lambda+\rho_{\st}
=
\Lambda'+\rho_{\nst},
\qquad
\Lambda+\rho_{A}
=
\Lambda'+\rho_{B}, \qquad \Lambda(\Uppi_i)+\rho(\Uppi_i)
=
\Lambda(\Uppi_j)+\rho(\Uppi_j).
\end{equation}

For $\slmn$, Remark~\ref{rmk::determination_Lambda'} provides a direct way to recover $\Lambda'$ from a given unitary
representation $L(\Lambda;\Delta_{\st}^{+})$.

Every highest weight representation $M$ of $\gg$ restricts to a representation of $\even$. In what follows, we forget the induced $\ZZ_2$-grading and regard this restriction as an ordinary $\even$-representation, denoted by $M\vert_{\even}$. Although this restriction need not be semisimple, it admits a finite filtration.

\begin{proposition}
 Let $M$ be a highest weight representation of $\gg$ relative to $\Delta^{+}$. Then $M\vert_{\even}$ admits a finite filtration whose successive quotients are highest weight representations of $\even$.
\end{proposition}

\begin{proof}
 It is well known that every highest weight representation $M$ belongs to the super BGG category $\mathcal{O}$ (see, for example, \cite{musson2012lie}). Moreover, restriction to $\even$ maps $\mathcal{O}$ into the BGG category $\mathcal{O}(\even)$ by \cite[Proposition~7.1]{Mazorchuk_Andersen}. Since every object of $\mathcal{O}(\even)$ has finite length and its simple objects are highest weight representations, $M\vert_{\even}$ admits a finite filtration whose successive quotients are highest weight representations of $\even$.
\end{proof}

If the highest weights of the successive quotients are
$\kk^{\CC}$-dominant integral, then their possible form is highly
constrained.

\begin{theorem}[{\cite[Theorem~2.5 and Corollary~2.7]{jakobsen1994full}}]
\label{thm::even_filtration}
Let $M$ be a highest weight representation of $\gg$ of highest weight
$\Lambda$ relative to $\Delta^{+}$. Assume that
$M\vert_{\even}$ admits a finite filtration whose nonzero successive
quotients have $\kk^{\CC}$-dominant integral highest weights. Then
every such highest weight is of the form $\Lambda-\gamma,$
where \[\gamma=\sum_{\alpha\in S}\alpha+\sum_{\beta\in\Delta_{\bar0,\mathrm{nc}}^{+}}
n_{\beta}\beta\]
for some subset $S\subseteq\Delta_{\bar1}^{+}$ and
$n_{\beta}\in\ZZ_{\geq0}$. In particular, the odd roots appearing in
the first sum are pairwise distinct. If $\gg=\sl(m\vert n)$ or $\gg=\gl(m\vert n)$, then one may take
$n_{\beta}=0$ for every
$\beta\in\Delta_{\bar0,\mathrm{nc}}^{+}$. Hence every such highest
weight is of the form
\[
\Lambda-\sum_{\alpha\in S}\alpha,
\qquad
S\subseteq\Delta_{\bar1}^{+}.
\]
\end{theorem}

Let $\Gamma$ denote the set of all weights $\gamma$ of the form $\gamma=\sum_{\alpha\in S}\alpha+\sum_{\beta\in\Delta_{\bar0,nc}^{+}}n_{\beta}\beta$, where $S\subseteq\Delta_{\bar1}^{+}$ and $n_{\beta}\in\ZZ_{\geq0}$. For $\gamma\in\Gamma$, let $p(\gamma)$ denote the number of such expressions. If $\gg=\sl(m\vert n)$ or $\gg=\gl(m\vert n)$, we take $\Gamma=\{\sum_{\alpha\in S}\alpha:S\subseteq\Delta_{\bar1}^{+}\}$.

\begin{corollary}
\label{corollary::even-filtration-simple}
If $L(\Lambda;\Delta^+)$ is unitary, then the filtration of Theorem~\ref{thm::even_filtration} splits, and $$L(\Lambda;\Delta^+)\vert_{\even}\cong\bigoplus_{\gamma\in\Gamma_{L}(\Lambda)}L_0(\Lambda-\gamma)^{\oplus m_{\Lambda-\gamma}}$$ for some finite subset $\Gamma_{L}(\Lambda)\subseteq\Gamma$, where $m_{\Lambda-\gamma}\in\ZZ_{\geq0}$ denotes the multiplicity of $L_0(\Lambda-\gamma)$. 
\end{corollary}

\begin{proof}
 It remains to prove that the sum is finite. Let $v_{\Lambda}$ be a highest weight vector of $L(\Lambda;\Delta^{+})$. Since $L(\Lambda;\Delta^{+})$ is simple, $L(\Lambda;\Delta^{+})=\UE(\gg)v_{\Lambda}$. By the PBW theorem, $\UE(\gg)$ is a finitely generated left
$\UE(\even)$-representation. 
Consequently, $L(\Lambda;\Delta^{+})$ is finitely generated as a $\even$-representation. Hence, if $L(\Lambda;\Delta^{+})$ is unitary, the sum is finite. 
\end{proof}

The preceding filtration determines the possible highest weights of the $\even$-constituents, but it does not determine which of them occur in a simple quotient or with what multiplicities. This problem is complicated by the presence of isotropic odd roots and the resulting phenomenon of atypicality.

\begin{definition}
Let $\Lambda\in\hh^{\ast}$ be a highest weight relative to $\Delta^+$, and let $\rho$ be the corresponding Weyl vector. The \emph{degree of atypicality} $\operatorname{atyp}(\Lambda)$ is the largest integer $r\in\ZZ_{\geq0}$ for which there exist pairwise orthogonal, linearly independent odd isotropic roots $\alpha_1,\ldots,\alpha_r\in\Delta_{\bar1}^{+}$ satisfying $(\Lambda+\rho,\alpha_i)=0$ for every $1\leq i\leq r$. The weight $\Lambda$ and the corresponding simple highest weight representation are called \emph{typical} if $\operatorname{atyp}(\Lambda)=0$ and \emph{atypical} otherwise.
\end{definition}

\begin{remark}
Although the definition depends on the chosen positive system, the corresponding highest weights $\Lambda$ and $\Lambda'$ introduced above have the same degree of atypicality. More precisely, $\operatorname{atyp}(\Lambda)=\operatorname{atyp}(\Lambda')$, where, for $\slmn$ and $\glmn$, the two sides are computed with respect to $\Delta_{\mathrm{st}}^{+}$ and $\Delta_{\mathrm{nst}}^{+}$, respectively, and, for $\osp(m\vert 2n)$ with respect to $\Delta_{A}^{+}$ and $\Delta_{B}^{+}$. Thus, the degree of atypicality is preserved under the correspondence induced by odd reflections.
\end{remark}

The degree of atypicality is closely related to the number of $\even$-constituents of a highest weight representation. In physics, this phenomenon is known as shortening: typical representations are called long, whereas atypical representations are called short. For instance, consider $\slmn$ with the standard positive system. In the finite-dimensional case, the basic relation between atypicality and shortening is given by Kac's irreducibility criterion \cite{Kac_representations}. Assume that $\Delta^+=\Delta_{\mathrm{st}}^+$ and that $\Lambda$ is dominant integral for $\even$, so that $L_0(\Lambda)$ and $K(L_0(\Lambda))$ are finite-dimensional. Then
\begin{equation}
K(L_0(\Lambda))
\text{ is simple}
\quad\Longleftrightarrow\quad
\operatorname{atyp}(\Lambda)=0.
\end{equation}
If $\Lambda$ is atypical, then $K(L_0(\Lambda))$ has a nonzero maximal proper subrepresentation, and its simple quotient $L(\Lambda;\Delta_{\mathrm{st}}^+)$ is therefore strictly smaller than the induced representation.

Taken together, the preceding filtration theorem provides a finite set of possible highest weights for the $\even$-constituents, but determines neither which of them actually occur nor their multiplicities in the simple quotient. In particular, these data are not determined by the degree of atypicality alone. Determining the occurring constituents and their multiplicities is precisely the branching problem studied in the following sections.

\section{Kostant's Dirac operator \texorpdfstring{$\Dirac_{\gg,\even}$}{} and unitarity}
\label{sec::Kostants_cubic_Dirac_operator}
\noindent
Let $\gg$ be either $\sl(m\vert n)$ with $m,n\geq1$ and $m\neq n$, $\gl(n\vert n)$ with $n\geq1$, $\osp(m\vert 2n)$ with $m,n\geq1$ or $F(4)$. In this section, we introduce the relative Dirac operator $\Dirac_{\gg,\even}$ associated with the pair $(\gg,\even)$. This operator was first introduced for quadratic Lie superalgebras by Huang and Pandžić \cite{huang2005dirac} and arises as the specialization of Kostant's relative cubic Dirac operator \cite{Schmidt_perturbations,Schmidt_Wernli} to the quadratic pair $(\gg,\even)$ and the decomposition $\gg=\even\oplus\odd$. 

We begin in Section~\ref{subsec::Dirac_general_theory} by recalling the general construction and basic properties of $\Dirac_{\gg,\even}$. In Section~\ref{subsec::Dirac_and_unitarity}, we study its relation to unitary highest weight representations. This relation depends essentially on the choice of positive system. We therefore introduce positive systems adapted to the relevant real form and show that, for such systems, $\Dirac_{\gg,\even}$ satisfies adjointness and positivity properties. For positive systems which are not adapted, these properties need not hold. In Section~\ref{subsec::Dirac_unitarity_slmn}, we turn to the standard positive system for $\sl(m\vert n)$, which is in general not adapted to the real forms under consideration. We develop a modification of the preceding theory which nevertheless permits the Dirac operator to be used effectively in this setting. This yields a new characterization of unitarity for the standard positive system and thereby completes the treatment initiated in \cite{Schmidt}. Finally, in Section~\ref{subsec::corrected_Dirac} we formulate Dirac inequalities for $\osp(m\vert2n)$ and $F(4)$ for non-$\omega$-adapted positive systems.

\subsection{General theory} \label{subsec::Dirac_general_theory} We briefly recall the construction and principal properties of the cubic Dirac operator, following \cite{huang2005dirac, Schmidt_perturbations}.

We first assume that $\gg\neq\sl(n\vert n)$. Let $(\cdot,\cdot)$
denote the form~\eqref{eq::supertrace_form}. In each
case, we equip $\gg$ with the normalized bilinear form $B(X,Y)\coloneqq\frac12(X,Y)$. Note that $B$ is nondegenerate, even, supersymmetric, and invariant. Its restriction to $\odd$ is symplectic. Choose complementary Lagrangian subspaces of $\odd$. For instance, any of the positive systems introduced in Example~\ref{ex::systems} or Example~\ref{ex::osp-systems} determines a polarization $\odd=\nn_{\bar1}^{+}\oplus\nn_{\bar1}^{-}$. Fix a basis $\{\partial_i\}_{i=1}^{N}$ of the first Lagrangian subspace, for example $\nn_{\bar1}^{+}$, and the corresponding basis $\{x_i\}_{i=1}^{N}$ of the second, normalized by
\begin{equation}
B(\partial_i,x_j)=-B(x_j,\partial_i)=\frac12\delta_{ij}.
\end{equation}

The associated Weyl algebra $\mathscr W(\gg_{\bar1})$ is then the algebra of differential operators with polynomial coefficients in $x_1,\ldots,x_{N}$, where $\partial_i$ acts as $\frac{\partial}{\partial x_i}$. Its defining commutation relations are $[x_i,x_j]_{\mathscr{W}}=[\partial_i,\partial_j]_{\mathscr{W}}=0$ and $[\partial_i,x_j]_{\mathscr{W}}=\delta_{ij}$ for $1\leq i,j\leq N$. There is a natural Lie algebra homomorphism $\alpha\colon\even\to\mathscr W(\odd)$ defined by $\alpha\coloneqq q\circ\lambda\circ\ad_{\odd}$ \cite{huang2005dirac}. Here, $\ad_{\odd}\colon\even\to\mathfrak{sp}(\odd,B)$ is induced by the adjoint action, which preserves $B$, while $\lambda\colon\mathfrak{sp}(\odd,B)\to\bigwedge^{2}(\odd)$\footnote{Here $\bigwedge(\odd)$ denotes the super exterior algebra of the purely odd vector space $\odd$. In particular, it coincides with the usual symmetric algebra $S(\odd)$ of the underlying vector space $\odd$.} is the canonical identification and $q\colon\bigwedge^{2}(\odd)\to\mathscr W(\odd)$ is Weyl quantization. Concretely (\emph{cf.}~\cite[Equation 10]{huang2005dirac}):
\begin{multline} \label{eq::alpha_explicit}
\alpha(X)= \sum_{i,j=1}^{N}(B(X,[\partial_{i},\partial_{j}])x_{i}x_{j}+B(X,[x_{i},x_{j}])\partial_{i}\partial_{j}) 
\\ -\sum_{i,j=1}^{N}2B(X,[x_{i},\partial_{j}])x_{j}\partial_{i}-\sum_{l=1}^{N}B(X,[\partial_{l},x_{l}])
\end{multline}
for any $X\in \even$. 

Consider $\UE(\gg)\otimes\mathscr W(\odd)$ and define the diagonal Lie algebra homomorphism $\Updelta\colon\even\to\UE(\gg)\otimes\mathscr W(\odd)$ by $\Updelta(x)\coloneqq x\otimes1+1\otimes\alpha(x)$. It induces the natural $\even$-action on $\UE(\gg)\otimes\mathscr W(\odd)$ given by $x\cdot(u\otimes w)\coloneqq [x,u]\otimes w+u\otimes[\alpha(x),w]_{\mathscr W}$. With respect to this action, the subalgebra of $\even$-invariants in $\UE(\gg)\otimes\mathscr W(\odd)$ is denoted by $\mathcal W(\gg,\even)$.

 \begin{definition}\label{def::Dirac}
Kostant's Dirac element for $(\gg,\even)$ is 
\[
\Dirac_{\gg,\even} \coloneqq 2 \sum_{i=1}^{N}(\partial_{i}\otimes x_{i}-x_{i}\otimes \partial_{i}) \in \UE(\gg)\otimes \mathscr{W}(\odd).
\]
 \end{definition}

 The following theorem summarizes its main properties.

\begin{theorem}[\cite{huang2005dirac}]\label{thm::properties_D}
 The relative Dirac element $\Dirac_{\gg,\even}$ has the following properties.
 \begin{enumerate}
 \item[(a)] It is independent of the chosen dual Lagrangian bases.
 \item[(b)] It is invariant under the diagonal action of $\even$, and hence
 $\Dirac_{\gg,\even}\in\mathcal W(\gg,\even)$.
 \item[(c)] It satisfies the Parthasarathy square formula
 \begin{equation*}
 \Dirac_{\gg,\even}^{2}
 =
 -\Omega_{\gg}\otimes1
 +\Omega_{\even,\Delta}
 -\frac18\operatorname{tr}_{\odd}
 \bigl(\Omega_{\even}\bigr),
 \end{equation*}
 where $\Omega_{\gg}$ and $\Omega_{\even}$ denote the quadratic Casimir
 elements with respect to $B$, and $\Omega_{\even,\Delta}$ denotes the
 diagonal image of $\Omega_{\even}$.
 \end{enumerate}
\end{theorem}

 \begin{remark} \label{rmk::Explicit_form_D_square}
For explicit calculations, we use the alternative square formula
\begin{equation*}
\Dirac^{2}_{\gg,\even} = 2 \sum_{i,j = 1}^{N} ([\partial_{i},\partial_{j}] \otimes x_{i}x_{j} + [x_{i},x_{j}] \otimes \partial_{i}\partial_{j} - 2[\partial_{i},x_{j}]\otimes x_{i}\partial_{j}) - 4 \sum_{i=1}^{N}x_{i}\partial_{i}\otimes 1.
\end{equation*}
\end{remark}

To extract representation-theoretic information from the intrinsically defined Dirac element $\Dirac_{\gg,\even}$, we pass to the associated Dirac operator, denoted by the same symbol,
\begin{equation}
\Dirac_{\gg,\even}\in\End_{\CC}\bigl(M\otimes M(\odd)\bigr),
\end{equation}
where $M$ is a $\gg$-representation \footnote{All properties of the Dirac operator remain preserved under parity reversal.} and $M(\odd)$ is the oscillator representation of the Weyl algebra $\mathscr{W}(\odd)$, which we now briefly recall. Fixing the polarization $\odd=\nn_{\bar1}^{-}\oplus\nn_{\bar1}^{+}$, the oscillator representation is realized on the polynomial space
\begin{equation}
M(\odd)\coloneqq\CC[x_1,\ldots,x_{N}],
\end{equation}
where $x_i$ acts by multiplication and the corresponding element of $\nn_{\bar1}^{+}$ by the partial derivative $\partial_i=\partial/\partial x_i$. This is the simple Fock representation associated with the chosen polarization: it is generated by the vacuum vector $1$, which is annihilated by all $\partial_i$. This realization is the familiar Bargmann–Fock representation from quantum mechanics, in which multiplication and differentiation play the role of creation and annihilation operators, respectively \cite{BargmannHermitianForm}. Through the map $\alpha:\even\to\mathscr{W}(\odd)$, $M(\odd)$ becomes a $\even$-representation. The vacuum vector $1$ is a $\even$-weight vector and satisfies (\emph{cf}.~\cite{SchmidtDirac})
\begin{equation}
\alpha(H)1=-\rho_{\bar1}(H)1,\qquad H\in\hh.
\end{equation}
We equip $M(\odd)$ with the Bargmann–Fock Hermitian form
\begin{equation}\label{eq::Bargmann_Fock_form}
\langle\prod_{k=1}^{N}x_k^{p_k},\prod_{k=1}^{N}x_k^{q_k}\rangle_{M(\odd)}
=
\delta_{p,q}\prod_{k=1}^{N}p_k!,
\end{equation}
for which multiplication by $x_k$ is adjoint to $\partial_k$. 

Although the Bargmann--Fock form is positive definite, the $\even$-representation on $M(\odd)$ induced by $\alpha$ is unitary with respect to this form only when the chosen polarization is compatible with the conjugate-linear anti-involution $\omega$. Before establishing the precise relation between unitarity and $\omega$-adaptedness, we first record the compatibility of $B$ with the conjugate-linear anti-involutions introduced in Section~\ref{subsec:real_forms_slmb}.

\begin{lemma}\label{lemm::B_omega_consistent}
 Let $\omega$ be one of the conjugate-linear anti-involutions introduced in Section~\ref{subsec:real_forms_slmb} for $\gg=\sl(m\vert n), \osp(m\vert 2n)$ or $F(4)$. Then, for all homogeneous $X,Y\in\gg$,
 \[
 B(\omega(X),\omega(Y))
 =
 (-1)^{p(X)p(Y)}\overline{B(X,Y)}.
 \]
\end{lemma}

\begin{proof} Assume first $\gg \neq \sl(n\vert n)$. Then $\gg$ is simple. 
For homogeneous $X,Y\in\mathfrak{g}$, define
\[
B_{\omega}(X,Y)
\coloneqq
(-1)^{p(X)p(Y)}
\overline{B(\omega(X),\omega(Y))}.
\]
Since $\omega$ is conjugate-linear, $B_{\omega}$ is complex bilinear.
It is clearly even and nondegenerate. Moreover, a straightforward computation shows that $B_{\omega}$ is supersymmetric and invariant. Hence, since $\gg$ is simple, 
there exists $c\in\mathbb{C}^{\times}$ such that
$
B_{\omega}=cB.
$ It remains to determine $c$. For any $X \in \gg_{\bar{0},\RR}$ one has $\omega(X)=-X$, and hence for all $X,X'\in \gg_{\bar{0},\RR}$
\[
B(\omega(X),\omega(X'))
=
\overline{B(X,X')}
\]
as $B(X,X')\in \RR$ for all $X,X' \in \gg_{\bar{0},\RR}$. Since $\left.B\right|_{\gg_{\bar0,\RR}}\not\equiv0$, the equality
$B_{\omega}=cB$ implies $c=1$.

It remains to consider $\sl(n\vert n)$. Here, the anti-involution is of the form
 \[
 \omega(X)=J^{-1}X^{\dagger}J,
 \qquad
 X^{\dagger}=\overline{X}^{\,T},
 \]
 for some even invertible matrix $J$. Since the supertrace is invariant under
 conjugation by even invertible matrices, we obtain
 \begin{equation*}
  \str(\omega(X)\omega(Y))
  =
  \str\bigl(J^{-1}X^{\dagger}Y^{\dagger}J\bigr)
  =
  \str(X^{\dagger}Y^{\dagger})
  =
  \str((YX)^{\dagger})
  =
  \overline{\str(YX)}
  =
  (-1)^{p(X)p(Y)}\overline{\str(XY)}.
 \end{equation*}
 The assertion now follows from the definition of $B$.
\end{proof}

We now characterize the unitarity of the oscillator representation in terms of the $\omega$-adaptedness of the positive system defining the polarization.

\begin{proposition}\label{prop::unitarity_oscillator_representation}
$(M(\odd),\langle\cdot,\cdot\rangle_{M(\odd)})$ is a unitary
 $\even$-representation under the action induced by $\alpha$ if and only if
 the positive system used to define the polarization $\odd=\nn_{\bar1}^{-}\oplus\nn_{\bar1}^{+}$
 is $\omega$-adapted.
\end{proposition}

\begin{proof}
We first note that $M(\odd)$ is a unitary $\even$-representation if and only if
\begin{equation}\label{eq::oscillator_unitarity}
 \alpha(X)^\dagger=\alpha(\omega(X))
 \qquad
 \text{for all }X\in\even.
\end{equation}
Using \eqref{eq::alpha_explicit}, we have
\begin{equation*}
\begin{aligned}
\alpha(X)^\dagger
={}&
\sum_{i,j=1}^{N}
\overline{B(X,[\partial_i,\partial_j])}\,
\partial_i\partial_j
+
\sum_{i,j=1}^{N}
\overline{B(X,[x_i,x_j])}\,
x_ix_j
\\
&-
2\sum_{i,j=1}^{N}
\overline{B(X,[x_j,\partial_i])}\,
x_j\partial_i
-
\sum_{l=1}^{N}
\overline{B(X,[\partial_l,x_l])},
\end{aligned}
\end{equation*}
where the indices $i$ and $j$ have been interchanged in the mixed term.
On the other hand,

\begin{multline*}
\alpha(\omega(X))
=
\sum_{i,j=1}^{N}
B(\omega(X),[\partial_i,\partial_j])\,x_ix_j
+
\sum_{i,j=1}^{N}
B(\omega(X),[x_i,x_j])\,\partial_i\partial_j
\\
-
2\sum_{i,j=1}^{N}
B(\omega(X),[x_i,\partial_j])\,x_j\partial_i
-
\sum_{l=1}^{N}
B(\omega(X),[\partial_l,x_l]).
\end{multline*}
Suppose first that the positive system is $\omega$-adapted. Thus the dual
bases may be chosen such that
\[
 \omega(x_i)=\varepsilon\partial_i,
 \qquad
 \omega(\partial_i)=\varepsilon x_i,
 \qquad
 \varepsilon\in\{\pm1\}.
\]
Using the compatibility of $B$ with $\omega$ and $\varepsilon^2=1$, we obtain
\[
\begin{gathered}
\overline{B(X,[\partial_i,\partial_j])}
 = B(\omega(X),[x_i,x_j]),
\qquad
\overline{B(X,[x_i,x_j])}
 = B(\omega(X),[\partial_i,\partial_j]),\\[0.5em]
\overline{B(X,[x_j,\partial_i])}
 = B(\omega(X),[x_i,\partial_j]).
\end{gathered}
\]
and similarly
\[
 \overline{B(X,[\partial_l,x_l])}
 =
 B(\omega(X),[\partial_l,x_l]).
\]
Consequently,
$\alpha(X)^\dagger=\alpha(\omega(X))
$
for every $X\in\even$, and hence $M(\odd)$ is unitary.

Conversely, suppose that $M(\odd)$ is unitary, so that
\eqref{eq::oscillator_unitarity} holds. Let $\tau$ be the conjugate-linear
involution of $\odd$ determined by $\tau(x_i)=\partial_i$ and
$\tau(\partial_i)=x_i$. It satisfies by definition $B(\tau(X),\tau(Y))=-\overline{B(X,Y)}$ for all $X,Y \in \odd$. Since
$[\alpha(X),u]_{\mathscr W}=[X,u]$ for $X\in\even$ and $u\in\odd$, taking
adjoints and using \eqref{eq::oscillator_unitarity} gives
$\tau([X,u])=-[\omega(X),\tau(u)]$. On the other hand, since $\omega$ is an
anti-involution, $\omega([X,u])=-[\omega(X),\omega(u)]$. Hence
$T\coloneqq\tau^{-1}\omega|_{\odd}$ is a complex-linear
$\even$-endomorphism of $\odd$. Moreover, $T$ preserves the form $B$, since
\[
B(T(X),T(Y))
=
B\bigl(\tau^{-1}(\omega(X)),\tau^{-1}(\omega(Y))\bigr)
=
-\overline{B(\omega(X),\omega(Y))}
=
B(X,Y)
\]
for all $X,Y\in\odd$, where the last equality follows from Lemma~\ref{lemm::B_omega_consistent}.

If $\gg$ is $\osp(m\vert 2n)$ with $m\neq 2$ or $F(4)$, then $\odd$ is a simple $\even$-representation, and Schur's
lemma yields $T=c\,\id_{\odd}$ for some $c\in\CC^\times$. Since $T$
preserves the nondegenerate invariant form on $\odd$, one has $c^2=1$.
Thus $c=\varepsilon\in\{\pm1\}$, and therefore
$\omega(x_i)=\varepsilon\partial_i$ and
$\omega(\partial_i)=\varepsilon x_i$. Hence the polarization is
$\omega$-adapted.

If $\gg$ is $\osp(2\vert 2n)$, $\sl(m\vert n)$ or $\glmn$, then
$\odd=\gg_{-1}\oplus\gg_{1}$ is the direct sum of two simple
$\even$-representations. Thus
$T|_{\gg_{-1}}=c_{-}\id$ and $T|_{\gg_{1}}=c_{+}\id$ for some
$c_{\pm}\in\CC^\times$. Since the invariant form pairs $\gg_{-1}$ and
$\gg_{1}$ nondegenerately, $c_{-}c_{+}=1$, while $\omega^2=\id$ implies
$c_{\pm}\in\RR^\times$. Hence $c_{-}$ and $c_{+}$ have the same sign
$\varepsilon\in\{\pm1\}$. After rescaling each dual pair by
$x_i'=a_i x_i$ and $\partial_i'=a_i^{-1}\partial_i$, with
$|a_i|^2=|c_i|^{-1}$, one obtains
$\omega(x_i')=\varepsilon\partial_i'$ and
$\omega(\partial_i')=\varepsilon x_i'$. Thus the polarization is again
$\omega$-adapted.
\end{proof}

\begin{corollary} \begin{enumerate} \item[(a)] For $\gg=\sl(m\vert n)$, the oscillator representation $M(\odd)$ associated with the nonstandard positive system is unitary with respect to both $\omega_{p,q\vert 0,n}$ and $\omega_{p,q\vert n,0}$. If $p,q>0$, the oscillator representation associated with the standard positive system is not unitary with respect to either anti-involution. If $pq=0$, it is unitary with respect to both. \item[(b)] For $\gg=\osp(m\vert 2n)$, the oscillator representation associated with the Case~A positive system is unitary with respect to both $\upomega_{m,0\vert 2n}$ and $\upomega_{0,m\vert 2n}$. The oscillator representation associated with the Case~B positive system is unitary with respect to either anti-involution if and only if $m=1$, in which case $\Delta_A^{+}=\Delta_B^{+}$. If $m=2r$, the oscillator representation associated with the Case~B positive system is unitary with respect to both $\upomega^{*}_{2r\vert2n,0}$ and $\upomega^{*}_{2r\vert0,2n}$. The oscillator representation associated with the Case~A positive system is unitary with respect to neither anti-involution. \end{enumerate} \end{corollary}

Since our subsequent application is to irreducible unitary representations, Theorem~\ref{theorem::unitary-highest-lowest-weight} allows us to restrict there to highest or lowest weight representations. We formulate the following square formula for a simple highest weight representation with respect to a fixed positive system $\Delta^{+}$. The square formula of Theorem~\ref{thm::properties_D} then determines the induced action of $\Dirac_{\gg,\even}^{2}$ on the composition factors of $M\otimes M(\odd)$ as a $\even$-representation.\footnote{As a $\even$-representation, $M\otimes M(\odd)$ is a weight representation which is locally finite for $\nn_{\bar{0}}^{+}$. Hence its simple subquotients are highest weight $\even$-representations, which we denote by $L_0(\mu)$.} 

\begin{theorem}\label{thm::dirac-square-action}
Let $M$ be a simple highest weight $\gg$-representation of highest weight $\Lambda$, and let $L_{0}(\mu)$ be a $\even$-composition factor of $M\otimes M(\odd)$ of highest weight $\mu$. Then $\Dirac_{\gg,\even}^{2}$ acts on the composition factor $L_{0}(\mu)$ by the scalar
\begin{equation*}
2((\mu+\rho_{\bar{1}}+2\rho,\mu+\rho_{\bar{1}})-(\Lambda+2\rho,\Lambda)).
\end{equation*}
\end{theorem}
\begin{proof}
By Theorem~\ref{thm::properties_D}, $\Dirac_{\gg,\even}^{2}$ acts on the composition factor $L_0(\mu)$ by the scalar
\begin{equation*}
-B(\Lambda+2\rho,\Lambda)+B(\mu+\rho_{\bar{0}},\mu)-\frac{1}{8}\tr_{\odd}(\Omega_{\even})=2(-(\Lambda+2\rho,\Lambda)+(\mu+2\rho_{\bar0},\mu))-\frac18\tr_{\odd}(\Omega_{\even}).
\end{equation*}
Let $v_{\Lambda}$ be a highest weight vector of $M$. Then $v_{\Lambda}\otimes 1\in M\otimes M(\odd)$ has weight $\Lambda-\rho_{\bar1}$ and generates a $\even$-highest weight subrepresentation of $M\otimes M(\odd)$. Moreover, $v_{\Lambda}\otimes1\in\ker\Dirac_{\gg,\even}^{2}$ as a direct calculation shows (\emph{cf.}~\cite{SchmidtDirac}). Hence
\begin{equation*}
0
=
2(-(\Lambda+2\rho,\Lambda)
+
(\Lambda-\rho_{\bar1}+2\rho_{\bar0}),\Lambda-\rho_{\bar1})
-
\frac18\tr_{\odd}(\Omega_{\even}).
\end{equation*}
Using $\rho=\rho_{\bar0}-\rho_{\bar1}$, the terms involving $\Lambda$ cancel, and we obtain
\begin{equation*}
\frac18\tr_{\odd}(\Omega_{\even})
=
2(\rho_{\bar1},\rho_{\bar1})
-
4(\rho_{\bar0},\rho_{\bar1}).
\end{equation*}
Finally,
\begin{equation*}
(\mu+2\rho_{\bar0},\mu)
-
(\rho_{\bar1},\rho_{\bar1})
+
2(\rho_{\bar0},\rho_{\bar1})
=
(\mu+\rho_{\bar1}+2\rho,\mu+\rho_{\bar1}),
\end{equation*}
which proves the claim.
\end{proof}

It remains to consider the case $\sl(n\vert n)$ for $n\geq 2$. Although the restriction of the supertrace form to $\sl(n\vert n)$ is degenerate, its restriction to $\sl(n\vert n)_{\bar{1}}$ is nondegenerate. Since $\gl(n\vert n)_{\bar{1}}=\sl(n\vert n)_{\bar{1}}$, Definition~\ref{def::Dirac} still defines $\Dirac_{\gg,\even}$, and hence a Dirac operator on $M\otimes M(\odd)$ for every $\sl(n\vert n)$-representation $M$. To obtain the corresponding square formula, however, we pass to $\gl(n\vert n)$, for which the supertrace form is nondegenerate. Every highest weight $\sl(n\vert n)$-representation extends to a highest weight $\gl(n\vert n)$-representation upon choosing an extension $\widetilde{\Lambda}\in\mathfrak{d}^{\ast}$ of its highest weight $\Lambda\in\hh^{\ast}$. Although this extension is not unique, the action of the square in Theorem~\ref{thm::dirac-square-action} is independent of its choice. Indeed, any two extensions differ by a multiple of
$
 \chi\coloneqq\sum_{i=1}^{n}\varepsilon_i-\sum_{a=1}^{n}\delta_a,
$
which vanishes on $\hh$ and satisfies $(\chi,\chi)=0$ and $(\chi,\alpha)=0$ for every root $\alpha$ of $\mathfrak{gl}(n\vert n)$. Moreover, once $\widetilde\Lambda$ is fixed, the extension of a weight $\mu$ occurring in the representation is chosen by
$\widetilde{\mu}=\widetilde{\Lambda}-(\Lambda-\mu)$. Replacing $\widetilde{\Lambda}$ and $\widetilde{\mu}$ by $\widetilde{\Lambda}+c\chi$ and $\widetilde{\mu}+c\chi$, respectively, therefore changes
$
 (\widetilde{\mu}+\rho_{\bar{1}}+2\rho,\widetilde{\mu}+\rho_{\bar{1}})
 -
 (\widetilde{\Lambda}+2\rho,\widetilde{\Lambda})
$
by $2c(\widetilde{\mu}+\rho_{\bar{1}}-\widetilde{\Lambda},\chi)=0$, since $\widetilde{\mu}-\widetilde{\Lambda}$ lies in the root lattice and $(\rho_{\bar{1}},\chi)=0$. Hence the scalar is independent of the chosen extension, and Theorem~\ref{thm::dirac-square-action} holds unchanged for $\sl(n\vert n)$.

\subsection{\texorpdfstring{$\Dirac_{\gg,\even}$}{} and unitarity}
\label{subsec::Dirac_and_unitarity}

We now relate the adjointness properties of $\Dirac_{\gg,\even}$ to the
real form of $\gg$ and derive the corresponding Dirac inequalities for
unitary highest weight representations.

Let $(M,\langle\cdot,\cdot\rangle_M)$ be a unitary highest weight
representation of $\gg$ with respect to a real form defined by a
conjugate-linear anti-involution $\omega$. We equip $M\otimes M(\odd)$ with the tensor-product Hermitian form
\begin{equation}
 \langle v\otimes P,w\otimes Q\rangle_{M\otimes M(\odd)}
 =
 \langle v,w\rangle_M
 \langle P,Q\rangle_{M(\odd)}.
\end{equation}
The corresponding adjoint on $U(\gg)\otimes\mathscr W(\odd)$ is given by
\begin{equation}
 (u\otimes a)^{\ast}
 =
 \omega(u)\otimes a^{\dagger},
 \qquad
 u\in U(\gg),\quad a\in\mathscr W(\odd),
\end{equation}
where $\omega$ denotes the conjugate-linear anti-involution of $U(\gg)$
extending that of $\gg$, and $(\cdot)^{\dagger}$ is the Bargmann--Fock
adjoint on $\mathscr W(\odd)$, characterized by
$x_i^{\dagger}=\partial_i$ and $\partial_i^{\dagger}=x_i$.

The Bargmann--Fock form $\langle\cdot,\cdot\rangle_{M(\odd)}$ is always
positive definite. However, under the $\even$-action induced by $\alpha$,
the oscillator representation $M(\odd)$ is unitary precisely when the positive
system defining the polarization is $\omega$-adapted (Proposition~\ref{prop::unitarity_oscillator_representation}). The same
condition governs the adjointness properties of the Dirac operator.

\begin{theorem}\label{thm::Dirac_adjointness}
 Let $(M,\langle\cdot,\cdot\rangle_M)$ be a unitary highest weight
 representation and suppose that $\Delta^{+}$ is $\omega$-adapted with
 sign $\varepsilon\in\{\pm1\}$. Then
 \[
  \Dirac_{\gg,\even}^{\ast}
  =
  -\varepsilon\Dirac_{\gg,\even}.
 \]
 In particular, $\Dirac_{\gg,\even}$ is selfadjoint if
 $\varepsilon=-1$ and skewadjoint if $\varepsilon=+1$. Consequently,
 $\Dirac_{\gg,\even}^{2}$ is positive semidefinite if
 $\varepsilon=-1$ and negative semidefinite if $\varepsilon=+1$.
\end{theorem}

\begin{proof}
 Since $\Delta^{+}$ is $\omega$-adapted with sign $\varepsilon$, the
 dual bases may be chosen such that
 $\omega(x_i)=\varepsilon\partial_i$ and
 $\omega(\partial_i)=\varepsilon x_i$. Unitarity of $M$ therefore gives
 $x_i^{\dagger}=\varepsilon\partial_i$ and
 $\partial_i^{\dagger}=\varepsilon x_i$, whereas on $M(\odd)$ the
 Bargmann--Fock form satisfies $x_i^{\dagger}=\partial_i$ and
 $\partial_i^{\dagger}=x_i$. Hence
 \begin{equation*}
  \Dirac_{\gg,\even}^{\ast}
  =
  2\sum_{i=1}^{N}
  \bigl(
  (\partial_i\otimes x_i)^{\ast}
  -
  (x_i\otimes\partial_i)^{\ast}
  \bigr)=
  2\varepsilon\sum_{i=1}^{N}
  \bigl(
  x_i\otimes\partial_i
  -
  \partial_i\otimes x_i
  \bigr)
  =
  -\varepsilon\Dirac_{\gg,\even}.
 \end{equation*}
 If $\varepsilon=-1$, then
 $\langle\Dirac_{\gg,\even}^{2}v,v\rangle
 =\|\Dirac_{\gg,\even}v\|^{2}\geq0$, while for $\varepsilon=+1$, $\langle\Dirac_{\gg,\even}^{2}v,v\rangle
  =
  -\|\Dirac_{\gg,\even}v\|^{2}
  \leq0$.
\end{proof}

For the remainder of this discussion, fix an $\omega$-adapted positive
system of sign $\varepsilon$.

\begin{proposition}\label{prop::tensor_product_unitary_even}
 The tensor product $M\otimes M(\odd)$ is a unitary $\even$-representation with
 respect to $\langle\cdot,\cdot\rangle_{M\otimes M(\odd)}$. In
 particular, its $\even$-decomposition is completely reducible.
\end{proposition}

\begin{proof}
 Both $M$ and $M(\odd)$ are unitary $\even$-representations, and therefore so is
 their tensor product under the diagonal $\even$-action. Complete
 reducibility follows from the positive definiteness of the Hermitian
 form and the invariance of orthogonal complements together with finite-dimensionality of the weight spaces.
\end{proof}

Combining Theorem~\ref{thm::Dirac_adjointness} with
Theorem~\ref{thm::dirac-square-action} yields the corresponding Dirac
inequality.

\begin{proposition}\label{prop::Dirac_inequality_adapted}
 Let $M$ have highest weight $\Lambda$, and let $L_{0}(\mu)$ be an
 $\even$-constituent of $M\otimes M(\odd)$. Then
 \[
  -\varepsilon
  \left(
  (\mu+\rho_{\bar1}+2\rho,\mu+\rho_{\bar1})
  -
  (\Lambda+2\rho,\Lambda)
  \right)
  \geq0.
 \]
\end{proposition}

\begin{proof}
 By Theorem~\ref{thm::dirac-square-action},
 $\Dirac_{\gg,\even}^{2}$ acts on $L_{0}(\mu)$ by the scalar
 \[
  (\mu+\rho_{\bar1}+2\rho,\mu+\rho_{\bar1})
  -
  (\Lambda+2\rho,\Lambda).
 \]
 The assertion therefore follows from Theorem~\ref{thm::Dirac_adjointness}.
\end{proof}

The preceding Dirac inequality can be sharpened to an inequality directly on $M$.

\begin{corollary}\label{cor::Dirac_inequality_on_M}
 Let $M$ be a unitary highest weight representation of highest weight
 $\Lambda$, and write
 $
 M\vert_{\even}
 =
 \bigoplus_{\mu}L_{0}(\mu)^{m_{\Lambda}(\mu)}.
 $
 Then, for every $\mu$ with $m_{\Lambda}(\mu)\neq0$, one has
 \[
 -\varepsilon
 \bigl(
 -(\Lambda+2\rho,\Lambda)
 +(\mu+2\rho,\mu)
 \bigr)
 \geq0.
 \]
 Moreover, the inequality is strict whenever $\mu\neq\Lambda$.
\end{corollary}

\begin{proof}
 Let $v_{\mu}$ be a $\even$-highest weight vector of
 $L_{0}(\mu)$. Then $v_{\mu}\otimes 1$ generates an
 $\even$-subrepresentation of $M\otimes M(\odd)$ with highest weight
 $\mu-\rho_{\bar1}$. Applying
 Proposition~\ref{prop::Dirac_inequality_adapted} to this constituent
 gives
 \[
 -\varepsilon
 \bigl(
 (\mu+2\rho,\mu)-(\Lambda+2\rho,\Lambda)
 \bigr)
 \geq0,
 \]
 which is the claimed inequality.

 It remains to prove strictness. Suppose that equality holds. Since
 $\partial_i1=0$ for all $i$, we have
 $
 \Dirac_{\gg,\even}(v_{\mu}\otimes1)
 =
 2\sum_{i=1}^{mn}\partial_i v_{\mu}\otimes x_i.
 $
 Hence, using
 $\Dirac_{\gg,\even}^{*}=-\varepsilon\Dirac_{\gg,\even}$,
 \[
 0
 =
 \bigl\langle
 \Dirac_{\gg,\even}^{2}(v_{\mu}\otimes1),
 v_{\mu}\otimes1
 \bigr\rangle_{M\otimes M(\odd)}
 =
 -4\varepsilon
 \sum_{i=1}^{mn}
 \langle
 \partial_i v_{\mu},
 \partial_i v_{\mu}
 \rangle_M.
 \]
 Since $M$ is unitary, it follows that
 $\partial_i v_{\mu}=0$ for every $i$. As $v_{\mu}$ is already
 $\even$-highest and the $\partial_i$ span
 $\nn_{\bar1}^{+}$, the vector $v_{\mu}$ is a
 $\gg$-highest weight vector. Since $M$ is simple, the nonzero
 subrepresentation generated by $v_{\mu}$ is all of $M$. Hence its highest
 weight must equal that of $M$, and therefore $\mu=\Lambda$.
 Thus the inequality is strict whenever $\mu\neq\Lambda$.
\end{proof}

For the special linear Lie superalgebras $\slmn$, the Dirac inequality is strong
enough to characterize unitarity completely. Recall that every irreducible highest weight representation $M$ of highest weight $\Lambda$ carries, up to normalization, a nondegenerate $\omega$-contravariant Hermitian form $\langle\cdot,\cdot\rangle_M$ whenever $\overline{\Lambda(H)}=\Lambda(\omega(H))$ for all $H\in\hh$. This form is called the Shapovalov form. In what follows, whenever this condition is satisfied, we equip $M$ with its Shapovalov form $\langle\cdot,\cdot\rangle_M$. In particular, the condition holds for every unitary highest weight representation. Here, it is characterized by
\begin{equation}
 \langle xv,w\rangle_M
 =
 \langle v,\omega(x)w\rangle_M,
 \qquad
 x\in\gg,\quad v,w\in M.
\end{equation}
We refer to \cite{SchmidtDirac,Schmidt} for the construction of this
Shapovalov form. The representation $M$ is unitary precisely when this
form is positive definite. In \cite{SchmidtDirac,Schmidt}, the positivity of the Shapovalov form is
studied through its relation with the Dirac operator. For this purpose,
$M\otimes M(\odd)$ is equipped with the Hermitian form
\begin{equation}\label{eq::Hermitian_form_M_oscillator}
 \langle v\otimes P,w\otimes Q\rangle
 =
 \langle v,w\rangle_M
 \langle P,Q\rangle_{M(\odd)}.
\end{equation}
Moreover, the restriction of $M$ to $\even$ admits a finite filtration
whose composition factors are irreducible highest weight
$\even$-representations. The resulting characterization of unitarity is
the following.

\begin{theorem}[{\cite{SchmidtDirac}}]
\label{thm::unitarity_via_Dirac_inequality}
Let $M$ be a nontrivial irreducible highest weight $\sl(m\vert n)$-representation
of highest weight $\Lambda$ with respect to a real form $\su(p,q\vert 0,n)$ with $pq\neq 0$. Then $M$ is unitary if and only if
$L_0(\Lambda)$ is a unitary $\even$-representation and
\begin{equation*}
 (\mu+2\rho,\mu)
 >
 (\Lambda+2\rho,\Lambda)
\end{equation*}
for every simple $\even$-composition factor $L_0(\mu)\not\cong
L_0(\Lambda)$ occurring in the $\even$-filtration of $M$.
\end{theorem}

If the positive system is not $\omega$-adapted, the Dirac operator need
not be either selfadjoint or skewadjoint, and the preceding argument does
not yield a Dirac inequality of a fixed sign. Nevertheless, its different
summands may have definite adjointness properties separately. We illustrate
this phenomenon for $\sl(m\vert n)$ equipped with the standard positive
system. In this case, $\Dirac_{\gg,\even}$ decomposes naturally into
selfadjoint and skewadjoint parts, which leads to a corresponding
modification of the preceding inequality.

\subsection{Dirac operators and unitarity beyond \texorpdfstring{$\omega$}{omega}-adapted positive systems for \texorpdfstring{$\slmn$}{sl(m|n)}}\label{subsec::Dirac_unitarity_slmn}
The relation between the Dirac operator $\Dirac_{\gg,\even}$ and unitarity was established in \cite{SchmidtDirac} for the standard positive system when $pq=0$ and for the nonstandard positive system when $pq\neq0$. By Lemma~\ref{lemm::adapted_systems_slmn}, these are precisely the $\omega_{p,q\vert 0,n}$-adapted positive systems for the corresponding real forms. We now turn to the complementary case. Assume $pq\neq0$, so that the unitary highest weight representations under consideration are infinite-dimensional (see Remark~\ref{rmk::infinite_sl}), and fix the standard positive system $\Delta^{+}\coloneqq\Delta_{\st}^{+}$, which is not $\omega_{p,q\vert 0,n}$-adapted. Throughout this subsection, we write $\omega\coloneqq\omega_{p,q\vert 0,n}$. We develop the Dirac theory for this choice and derive a new characterization of unitarity, thereby extending the results of \cite{SchmidtDirac} to the standard positive system in the noncompact case. 

The odd
root spaces decompose into $p$- and $q$-blocks, on which $\omega$ acts
with opposite signs:
\begin{equation}
 \omega(E_{m+j,i})
 =
 \begin{cases}
 -E_{i,m+j}, & 1\leq i\leq p,\\
 E_{i,m+j}, & p+1\leq i\leq m.
 \end{cases}
\end{equation}
Here, we refer to the odd root spaces spanned by $E_{i,m+j}$ and $E_{m+j,i}$, $1\leq j\leq n$, with $1\leq i\leq p$ and $p+1\leq i\leq m$ as the $p$- and $q$-blocks, respectively. We set $\partial_{(i-1)n+j}\coloneqq E_{i,m+j}$ and
$x_{(i-1)n+j}\coloneqq E_{m+j,i}$ for $1\leq i\leq m$ and
$1\leq j\leq n$, so that
$\nn_{\bar1}^{+}=\operatorname{span}_{\CC}\{\partial_k\}$ and
$\nn_{\bar1}^{-}=\operatorname{span}_{\CC}\{x_k\}$. Thus the $p$- and
$q$-blocks satisfy $\omega(x_k)=-\partial_k$ and
$\omega(x_k)=\partial_k$, respectively. Since $p,q>0$, dual bases
cannot simultaneously satisfy
$B(\partial_k,x_l)=\tfrac12\delta_{kl}$ and
$\omega(x_k)=-\partial_k$.

Accordingly, the Dirac operator decomposes into its $p$- and $q$-blocks as
\begin{equation}
 \Dirac_{\gg,\even}
 =
 2\sum_{i=1}^{mn}
 (\partial_i\otimes x_i-x_i\otimes\partial_i)
 =
 \Dirac_{\mathrm{self}}+\Dirac_{\mathrm{skew}}
 \in\mathcal{W}(\gg,\even),
\end{equation}
where
\begin{equation}\label{eq:Dself_Dskew}
 \Dirac_{\mathrm{self}}
 \coloneqq
 2\sum_{i=1}^{pn}
 (\partial_i\otimes x_i-x_i\otimes\partial_i),
 \qquad
 \Dirac_{\mathrm{skew}}
 \coloneqq
 2\sum_{i=pn+1}^{mn}
 (\partial_i\otimes x_i-x_i\otimes\partial_i).
\end{equation}
The notation will become clear below. Although this decomposition is
defined by the $p$- and $q$-blocks of the standard polarization, its
summands are themselves relative Dirac operators for naturally embedded
Lie subsuperalgebras of $\gg$. To make this precise, we first record the
following lemma.

\begin{lemma}
The natural block embeddings
\begin{equation*}
 \sl(p\vert n),\,\sl(q\vert n)\hookrightarrow\sl(m\vert n)
\end{equation*}
are preserved by $\omega$, whose restrictions define the compact real
forms $\su(p,0\vert0,n)$ and $\su(0,q\vert0,n)$, respectively.
\end{lemma}

 \begin{proof}
The embeddings are the natural inclusions of the corresponding block
subsuperalgebras. From the explicit formula for $\omega$ on the matrix
units, the $p$- and $q$-blocks are preserved separately. Hence $\omega$
restricts to conjugate-linear anti-involutions on
$\sl(p\vert n)$ and $\sl(q\vert n)$. The resulting associated real
forms are precisely $\su(p,0\vert0,n)$ and
$\su(0,q\vert0,n)$, respectively.
\end{proof}

In what follows, we regard $\sl(p\vert n)$ and $\sl(q\vert n)$ as
subsuperalgebras of $\slmn$ via the embeddings above. Their odd parts
are precisely the $p$- and $q$-blocks of $\odd$, respectively. Hence the
corresponding summands in the decomposition of $\Dirac_{\gg,\even}$ are
the relative Dirac operators associated with the quadratic pairs
$(\sl(p\vert n),\sl(p\vert n)_{\bar0})$ and
$(\sl(q\vert n),\sl(q\vert n)_{\bar0})$ if $p\neq n$ and $q \neq n$, respectively. If $p=n$ or $q=n$, we use the same
convention as above for $\mathfrak{sl}(n|n)$.

\begin{lemma}\label{lemm::properties_D_self_and_D_skew}
If $p\neq n$, respectively $q\neq n$, the operators
$\Dirac_{\self}$, respectively $\Dirac_{\skew}$, are the relative Dirac
operators for the quadratic pairs $\bigl(\mathfrak{sl}(p|n),\mathfrak{sl}(p|n)_{\bar 0}\bigr)$ and $\bigl(\mathfrak{sl}(q|n),\mathfrak{sl}(q|n)_{\bar 0}\bigr).$ In particular,
\begin{equation*}
 \Dirac_{\self}
 \in
 \mathcal{W}\bigl(\sl(p\vert n),\sl(p\vert n)_{\bar0}\bigr), \qquad 
 \Dirac_{\skew}
 \in
 \mathcal{W}\bigl(\sl(q\vert n),\sl(q\vert n)_{\bar0}\bigr).
\end{equation*}
If $p=n$, respectively $q=n$, including the case $n=1$, the same statement is understood with the convention introduced above for $\mathfrak{sl}(n\vert n)$. In every case, the operators satisfy the properties of Theorem~\ref{thm::properties_D}.
\end{lemma}

\begin{proof}
Under the embeddings above, the odd parts of $\sl(p\vert n)$ and
$\sl(q\vert n)$ are spanned by
$\{\partial_i,x_i\}_{i=1}^{pn}$ and
$\{\partial_i,x_i\}_{i=pn+1}^{mn}$, respectively. The formulas for
$\Dirac_{\self}$ and $\Dirac_{\skew}$ therefore agree exactly with the
definition of the relative Dirac operator for the corresponding
quadratic pairs.
\end{proof}

\begin{remark}\label{rmk::D2_standard_system} In $\UE(\gg)\otimes \WW(\odd)$, one has
\begin{equation*}
 \Dirac_{\gg,\even}^{2}
 =
 \Dirac_{\self}^{2}
 +
 \Dirac_{\skew}^{2}
 +
 \Dirac_{\self}\Dirac_{\skew}
 +
 \Dirac_{\skew}\Dirac_{\self}.
\end{equation*}
The mixed term $\Dirac_{\self}\Dirac_{\skew}
 +
 \Dirac_{\skew}\Dirac_{\self}$ measures the interaction between the $p$- and
$q$-blocks and is precisely the term absent from the two compact
subsuperalgebra Dirac squares.
\end{remark}

We can now proceed as in \cite{SchmidtDirac} and relate unitarity to the
operators $\Dirac_{\self}$ and $\Dirac_{\skew}$. Both operators act
naturally on $M\otimes M(\odd)$ for every $\gg$-representation $M$, with
the first tensor factor acting through the $\gg$-representation and the
second through the oscillator representation. With respect to the
decomposition into the $p$- and $q$-blocks, the polynomial realization of
$M(\odd)$ decomposes accordingly, and $\Dirac_{\self}$ and
$\Dirac_{\skew}$ involve only the variables and differential operators
belonging to the respective blocks. Thus
\begin{equation}
 \Dirac_{\self},\Dirac_{\skew}
 \in\End_{\CC}(M\otimes M(\odd)).
\end{equation}
The following immediate consequence of the construction explains the
terminology.

\begin{lemma}
 Let $M$ be a unitary $\gg$-representation. Then $\Dirac_{\self}$ is
 selfadjoint and $\Dirac_{\skew}$ is skewadjoint on
 $M\otimes M(\odd)$. In particular,
 \[
  \Dirac_{\self}^{2}\geq0,
  \qquad
  \Dirac_{\skew}^{2}\leq0.
 \]
\end{lemma}

\begin{proof}
 For $1\leq i\leq pn$, one has $\omega(x_i)=-\partial_i$ and
 $\omega(\partial_i)=-x_i$. Since the Bargmann--Fock adjoint satisfies
 $x_i^{\dagger}=\partial_i$ and $\partial_i^{\dagger}=x_i$, it follows
 that
 \[
  (\partial_i\otimes x_i-x_i\otimes\partial_i)^{\ast}
  =
  -x_i\otimes\partial_i+\partial_i\otimes x_i.
 \]
 Hence $\Dirac_{\self}^{\ast}=\Dirac_{\self}$. Since the Hermitian form
 on $M\otimes M(\odd)$ is positive definite,
 \[
  \langle\Dirac_{\self}^{2}v,v\rangle_{M\otimes M(\odd)}
  =
  \langle\Dirac_{\self}v,\Dirac_{\self}v\rangle_{M\otimes M(\odd)}
  \geq0
 \]
 for all $v\in M\otimes M(\odd)$, and therefore
 $\Dirac_{\self}^{2}\geq0$.

 On the $q$-block, one has $\omega(x_i)=\partial_i$ and
 $\omega(\partial_i)=x_i$. The same computation gives
 $\Dirac_{\skew}^{\ast}=-\Dirac_{\skew}$ and hence
 \[
  \langle\Dirac_{\skew}^{2}v,v\rangle_{M\otimes M(\odd)}
  =
  -\langle\Dirac_{\skew}v,\Dirac_{\skew}v\rangle_{M\otimes M(\odd)}
  \leq0.
 \]
 Thus $\Dirac_{\skew}^{2}\leq0$.
\end{proof}

Next we define the change-of-system correction term by
\begin{equation}
\mathcal C^{\st}_{\nst}
\coloneqq 4
\sum_{i=pn+1}^{mn}[\partial_i^{\st},x_i^{\st}]\otimes1.
\end{equation}
To explain this terminology, we first clarify the notation relating the two positive systems. Fix an irreducible unitary highest weight representation $L(\Lambda;\Delta_{\st}^{+})$. By Proposition~\ref{prop::transport_of_unitarity}, there exist $\Lambda'\in\hh^{\ast}$ and $\upepsilon\in\ZZ_2$ such that
$\varphi:L(\Lambda;\Delta_{\st}^{+})\longrightarrow \Uppi^{\upepsilon}L(\Lambda';\Delta_{\nst}^{+})$
is a $\gg$-equivalence. The pullback under $\varphi$ of the Shapovalov form on $ \Uppi^{\upepsilon}L(\Lambda';\Delta_{\nst}^{+})$ is again an $\omega$-contravariant Hermitian form on $L(\Lambda;\Delta_{\st}^{+})$. Hence, by uniqueness of the Shapovalov form up to a nonzero real scalar, there exists $c\in\mathbb R^{\times}$ such that
\begin{equation*}
\langle\varphi(v),\varphi(w)\rangle_{ \Uppi^{\upepsilon}L(\Lambda';\Delta_{\nst}^{+})}
=
c\langle v,w\rangle_{L(\Lambda;\Delta_{\st}^{+})}.
\end{equation*}
After choosing compatible normalizations of the two Shapovalov forms, we may assume $c=1$.

Let $1_{\st}$ and $1_{\nst}$ denote the vacuum vectors in the Fock realizations associated with $\Delta_{\st}^{+}$ and $\Delta_{\nst}^{+}$, respectively. We distinguish the polarized bases associated with the standard and non-standard positive systems by writing $\{\partial_i^{\st},x_i^{\st}\}_{i=1}^{mn}$ and $\{\partial_i^{\nst},x_i^{\nst}\}_{i=1}^{mn}$, respectively. These bases may be chosen such that
\begin{equation}
\partial_i^{\nst}
=
\begin{cases}
\partial_i^{\st}, & 1\leq i\leq pn,\\
x_i^{\st}, & pn+1\leq i\leq mn,
\end{cases}
\qquad
x_i^{\nst}
=
\begin{cases}
x_i^{\st}, & 1\leq i\leq pn,\\
-\partial_i^{\st}, & pn+1\leq i\leq mn.
\end{cases}
\end{equation}
For $v\in L(\Lambda;\Delta_{\st}^{+})$, write $v'\coloneqq\varphi(v)$. Since $\varphi$ is $\gg$-equivariant, we then have
\begin{equation}
\partial_i^{\nst}v'
=
\begin{cases}
\varphi(\partial_i^{\st}v), & 1\leq i\leq pn,\\
\varphi(x_i^{\st}v), & pn+1\leq i\leq mn,
\end{cases}
\qquad
x_i^{\nst}v'
=
\begin{cases}
\varphi(x_i^{\st}v), & 1\leq i\leq pn,\\
-\varphi(\partial_i^{\st}v), & pn+1\leq i\leq mn.
\end{cases}
\end{equation}

Its name is explained by the following lemma. 

\begin{lemma} \label{lemm::corrected_Dirac_action}
 One has for any $v \in L(\Lambda;\Delta_{\st}^{+})$
 \[
 \langle \Dirac_{\gg,\even}^{2}(\varphi(v)\otimes 1_{\nst}),\varphi(v)\otimes 1_{\nst}\rangle_{ \Uppi^{\upepsilon}L(\Lambda';\Delta_{\nst}^{+})\otimes M(\odd)}=\langle (\Dirac_{\gg,\even}^{2}+\mathcal{C}_{\nst}^{\st})(v\otimes 1_{\st}),v\otimes 1_{\st}\rangle_{L(\Lambda;\Delta_{\st}^{+})\otimes M(\odd)}.\]
\end{lemma}

\begin{proof}
Using Remark~\ref{rmk::Explicit_form_D_square} and
$[\nn_{\bar1}^{+},\nn_{\bar1}^{+}]\subseteq\nn_{\bar0}^{+}$ for the
non-standard positive system, together with the fact that $\varphi$
preserves the Shapovalov forms, we obtain for any $v'=\varphi(v)$ with $v \in L(\Lambda;\Delta_{\st}^{+})$
\begin{equation*}
\begin{aligned}
&
\left\langle
\Dirac_{\gg,\even}^{2}(v'\otimes1_{\nst}),
v'\otimes1_{\nst}
\right\rangle_
{ \Uppi^{\upepsilon}L(\Lambda';\Delta_{\nst}^{+})\otimes M(\odd)}
\\
&\qquad=
4\sum_{i=1}^{mn}
\left\langle
\partial_i^{\nst}v',
\partial_i^{\nst}v'
\right\rangle_
{ \Uppi^{\upepsilon}L(\Lambda';\Delta_{\nst}^{+})}
\\
&\qquad=
4\sum_{i=1}^{pn}
\left\langle
\partial_i^{\st}v,\partial_i^{\st}v
\right\rangle_
{L(\Lambda;\Delta_{\st}^{+})}
+
4\sum_{i=pn+1}^{mn}
\left\langle
x_i^{\st}v,x_i^{\st}v
\right\rangle_
{L(\Lambda;\Delta_{\st}^{+})}.
\end{aligned}
\end{equation*}
For $pn+1\leq i\leq mn$, one has
$\omega(x_i^{\st})=\partial_i^{\st}$ and
$\omega(\partial_i^{\st})=x_i^{\st}$. Hence, by
$\omega$-contravariance,
\begin{equation*}
\begin{aligned}
\left\langle
[\partial_i^{\st},x_i^{\st}]v,v
\right\rangle_
{L(\Lambda;\Delta_{\st}^{+})}
&=
\left\langle
x_i^{\st}v,x_i^{\st}v
\right\rangle_
{L(\Lambda;\Delta_{\st}^{+})}
+
\left\langle
\partial_i^{\st}v,\partial_i^{\st}v
\right\rangle_
{L(\Lambda;\Delta_{\st}^{+})}.
\end{aligned}
\end{equation*}
Moreover, Remark~\ref{rmk::Explicit_form_D_square} and
Lemma~\ref{lemm::properties_D_self_and_D_skew} give
\begin{equation*}
\begin{aligned}
\left\langle
\Dirac_{\self}^{2}(v\otimes1_{\st}),
v\otimes1_{\st}
\right\rangle_
{L(\Lambda;\Delta_{\st}^{+})\otimes M(\odd)}
&=
4\sum_{i=1}^{pn}
\left\langle
\partial_i^{\st}v,\partial_i^{\st}v
\right\rangle_
{L(\Lambda;\Delta_{\st}^{+})}\\
\left\langle
\Dirac_{\skew}^{2}(v\otimes1_{\st}),
v\otimes1_{\st}
\right\rangle_
{L(\Lambda;\Delta_{\st}^{+})\otimes M(\odd)}
&=
-4\sum_{i=pn+1}^{mn}
\left\langle
\partial_i^{\st}v,\partial_i^{\st}v
\right\rangle_
{L(\Lambda;\Delta_{\st}^{+})}.
\end{aligned}
\end{equation*}
Consequently,
\begin{equation*}
\begin{aligned}
&
\left\langle
\Dirac_{\gg,\even}^{2}(v'\otimes1_{\nst}),
v'\otimes1_{\nst}
\right\rangle_
{ \Uppi^{\upepsilon}L(\Lambda';\Delta_{\nst}^{+})\otimes M(\odd)}
\\
&\qquad=
\left\langle
\Dirac_{\self}^{2}(v\otimes1_{\st}),
v\otimes1_{\st}
\right\rangle_
{L(\Lambda;\Delta_{\st}^{+})\otimes M(\odd)}
+
\left\langle
\Dirac_{\skew}^{2}(v\otimes1_{\st}),
v\otimes1_{\st}
\right\rangle_
{L(\Lambda;\Delta_{\st}^{+})\otimes M(\odd)}
\\
&\qquad\quad+
\left\langle
\mathcal C_{\nst}^{\st}(v\otimes1_{\st}),
v\otimes1_{\st}
\right\rangle_
{L(\Lambda;\Delta_{\st}^{+})\otimes M(\odd)}.
\end{aligned}
\end{equation*}
By Remark~\ref{rmk::D2_standard_system}, $
\Dirac_{\gg,\even}^{2}
=
\Dirac_{\self}^{2}
+
\Dirac_{\skew}^{2}
+
\Dirac_{\self}\Dirac_{\skew}
+
\Dirac_{\skew}\Dirac_{\self}$, and a direct computation gives
\begin{equation*}
\Dirac_{\self}\Dirac_{\skew}
+
\Dirac_{\skew}\Dirac_{\self}
=
-4
\sum_{\substack{1\leq i\leq pn\\ pn+1\leq j\leq mn}}
\left(
[\partial_i^{\st},x_j^{\st}]\otimes x_i^{\st}\partial_j^{\st}
+
[x_i^{\st},\partial_j^{\st}]\otimes x_j^{\st}\partial_i^{\st}
\right).
\end{equation*}
Since $\partial_i^{\st}1_{\st}=0$ for every $i$, we have $(\Dirac_{\self}\Dirac_{\skew}+\Dirac_{\skew}\Dirac_{\self})(v\otimes1_{\st})=0$, which proves the statement.
\end{proof}

\begin{lemma}\label{lemm:transport_Dirac_inequality} For any weight $\mu$ of $L(\Lambda;\Delta_{\st}^{+})$, one has 
\[
 -(\Lambda'+2\rho^{\nst},\Lambda')+(\mu+2\rho^{\nst},\mu)= -(\Lambda+2\rho^{\st},\Lambda)+(\mu+2\rho^{\st},\mu)+2\sum_{i=p+1}^{m}\sum_{a=1}^{n}(\mu,\epsilon_{i}-\delta_{a}).
 \]
\end{lemma}

\begin{proof}
  Let $v$ be a nontrivial element of the weight space $L(\Lambda;\Delta_{\st}^{+})^{\mu} \cong  \Uppi^{\upepsilon}L(\Lambda';\Delta_{\nst}^{+})^{\mu}$. We evaluate
 $\langle\Dirac_{\gg,\even}^{2}(v\otimes1_{\st}),v\otimes1_{\st}\rangle_{L(\Lambda;\Delta^{+}_{\st})\otimes M(\odd)}$ and $\langle\Dirac_{\gg,\even}^{2}(\varphi(v)\otimes1_{\nst}),\varphi(v)\otimes1_{\nst}\rangle_{L(\Lambda';\Delta^{+}_{\nst})\otimes M(\odd)}$. We start with $\langle\Dirac_{\gg,\even}^{2}(v\otimes1_{\st}),v\otimes1_{\st}\rangle_{L(\Lambda;\Delta^{+}_{\st})\otimes M(\odd)}$. By the Parthasarathy square formula, one has
 \[
 \Dirac_{\gg,\even}^{2}
 =
 -\Omega_{\gg}\otimes1+\Omega_{\even,\Delta}-(\rho_{\bar1},\rho_{\bar1})
+
2(\rho_{\bar0},\rho_{\bar1}).
 \]
Choose root vectors $e_\alpha\in\gg_\alpha$ and $f_\alpha\in\gg^{-\alpha}$ for all $\alpha\in\Delta^+$. Normalize them by $B(e_\alpha,f_\alpha)=1$, such that we can identify $e_{\alpha_i}=\sqrt{2}\partial_i$ and $f_{\alpha_i}=\sqrt{2}x_i$. If $\{h_r\}$ and $\{h^r\}$ are dual bases of $\hh$, then
 \[
 \Omega_{\even}
 =
 \sum_r h_rh^r
 +
 \sum_{\alpha\in\Delta_{\bar0}^+}
 (e_\alpha f_\alpha+f_\alpha e_\alpha).
 \]
 Writing $\rho^{\st}=\rho_{\bar0}-\rho_{\bar1}^{\st}$, we have
 $-2\sum_i[\partial_i,x_i]v=-2(\rho_{\bar1},\mu)v$. Hence
 \[
 \begin{aligned}
\langle\Dirac_{\gg,\even}^{2}(v\otimes1),v\otimes1\rangle_{L(\Lambda;\Delta_{\st}^{+})\otimes M(\odd)}
 =
 -(\Lambda+2\rho^{\st},\Lambda)+(\mu+2\rho^{\st},\mu) +2\sum_{\alpha\in\Delta_{\bar0}^+}
 \langle f_\alpha e_\alpha v,v\rangle_{L(\Lambda;\Delta_{\st}^{+})}.
 \end{aligned}
 \]
 Analogously, writing $\rho^{\nst}=\rho_{\bar{0}}-\rho_{\bar{1}}^{\nst}$, it follows that
 \[
 \begin{aligned}
\langle\Dirac_{\gg,\even}^{2}(\varphi(v)\otimes1),\varphi(v)\otimes1\rangle_{L(\Lambda';\Delta_{\nst}^{+})}
 =
 -(\Lambda'+2\rho^{\nst},\Lambda')+(\mu+2\rho^{\nst},\mu) +2\sum_{\alpha\in\Delta_{\bar0}^+}
 \langle f_\alpha e_\alpha \varphi(v),\varphi(v)\rangle_{L(\Lambda';\Delta_{\nst}^{+})}.
 \end{aligned}
 \]
Since the $\gg$-equivalence $\varphi$ preserves the norm, one has
\[
\sum_{\alpha\in \Delta_{\bar{0}}^{+}}\langle f_{\alpha}e_{\alpha}\varphi(v),\varphi(v)\rangle_{L(\Lambda';\Delta_{\nst}^{+})} = \sum_{\alpha\in \Delta_{\bar{0}}^{+}}\langle \varphi(f_{\alpha}e_{\alpha}v),\varphi(v)\rangle_{L(\Lambda';\Delta_{\nst}^{+})}= \sum_{\alpha\in \Delta_{\bar{0}}^{+}}\langle f_{\alpha}e_{\alpha}v,v\rangle_{L(\Lambda;\Delta_{\st}^{+})}
\]
Now, using Lemma~\ref{lemm::corrected_Dirac_action}, and the fact that $\mathcal{C}^{\st}_{\nst}(v)=4\sum_{i=pn+1}^{mn}\mu([\partial_{i}^{\st},x_{i}^{\st}])=2\sum_{i=p+1}^{m}\sum_{a=1}^{n}(\mu,\epsilon_{i}-\delta_{a})$, we conclude for any $\mu \in \mathcal{P}_{L(\Lambda;\Delta_{\st}^{+})}$
 \[
 -(\Lambda'+2\rho^{\nst},\Lambda')+(\mu+2\rho^{\nst},\mu)= -(\Lambda+2\rho^{\st},\Lambda)+(\mu+2\rho^{\st},\mu)+2\sum_{i=p+1}^{m}\sum_{a=1}^{n}(\mu,\epsilon_{i}-\delta_{a}).\qedhere
 \]
\end{proof}

The following theorem reformulates the Dirac unitarity criterion
in the standard highest weight realization by transporting the unitarity criterion of Theorem~\ref{thm::unitarity_via_Dirac_inequality} along $\varphi$. In what follows, after forgetting the $\ZZ_2$-grading, we use $\varphi$ to identify $L(\Lambda;\Delta_{\st}^{+})$ with $L(\Lambda';\Delta_{\nst}^{+})$. In particular, we regard $L_0(\Lambda')$ as the same $\even$-subrepresentation in either realization.

\begin{theorem}\label{thm::unitarity_standard_case}
The highest weight representation
$
L(\Lambda;\Delta_{\st}^{+})
\cong
 \Uppi^{\upepsilon}L(\Lambda';\Delta_{\nst}^{+})
$
is unitary if and only if the following conditions hold:
\begin{enumerate}
 \item[(a)]
 $L_0(\Lambda')$ is a unitary highest weight
 $\even$-representation.

 \item[(b)]
 For every $\even$-highest weight $\mu\neq\Lambda'$ occurring in the
 $\even$-filtration of
 $L(\Lambda;\Delta_{\st}^{+})$, one has
 \[
 -(\Lambda+2\rho^{\st},\Lambda)
 +(\mu+2\rho^{\st},\mu)
 +2\sum_{i=p+1}^{m}\sum_{a=1}^{n}
 (\mu,\epsilon_i-\delta_a)
 >0.
 \]
\end{enumerate}
\end{theorem}

\begin{proof}
By Lemma~\ref{lemm:transport_Dirac_inequality}, the Dirac inequality for
$L(\Lambda';\Delta_{\nst}^{+})$ is equivalent to the inequality in~(b) for the
corresponding $\even$-highest weights occurring in the
$\even$-filtration of
$L(\Lambda;\Delta_{\st}^{+})$.
Therefore, by Theorem~\ref{thm::unitarity_via_Dirac_inequality},
$L(\Lambda;\Delta_{\st}^{+})$ is unitary if and only if
$L_0(\Lambda')$ is unitary and all these inequalities are strict.
\end{proof}

\begin{remark} \begin{enumerate}
 \item[(a)] This is an explicit unitarity criterion for the standard positive system. Nevertheless, its application requires determining the corresponding highest weight $\Lambda'$ with respect to the adapted positive system $\Delta_{\nst}^{+}$. 
 \item[(b)] The inequality \[
 -(\Lambda+2\rho^{\st},\Lambda)
 +(\mu+2\rho^{\st},\mu)
 +2\sum_{i=p+1}^{m}\sum_{a=1}^{n}
 (\mu,\epsilon_i-\delta_a)
 >0\]
 is regarded as the appropriate Dirac inequality for the standard positive system $\Delta^{+}_{\st}$.
\end{enumerate}
\end{remark}

For a unitary representation $L(\Lambda;\Delta_{\st}^{+})$, this
provides a characterization of the highest weight
$\Lambda'\in\hh^{\ast}$ of its nonstandard realization in terms
of its $\even$-constituents.

\begin{corollary}\label{cor::finding_Lambda'}
Let $L(\Lambda;\Delta_{\st}^{+})$ be a unitary highest weight
representation. Then there exists a unique $\even$-constituent
$L_0(\Lambda')$ satisfying
\[
-(\Lambda+2\rho^{\st},\Lambda)
+(\Lambda'+2\rho^{\st},\Lambda')
+2\sum_{i=p+1}^{m}\sum_{a=1}^{n}
(\Lambda',\epsilon_i-\delta_a)
=0.
\]
For this weight $\Lambda'$, we have
$
L(\Lambda;\Delta_{\st}^{+})
\cong
 \Uppi^{\upepsilon}L(\Lambda';\Delta_{\nst}^{+})
$
for some $\varepsilon\in\{0,1\}$.
\end{corollary}

\begin{remark}\label{rmk::determination_Lambda'}
There is a direct way to determine $\Lambda'$. Every $\even$-constituent of
$L(\Lambda;\Delta_{\st}^{+})$ has highest weight $\Lambda-\gamma$, where
$\gamma$ is a sum of distinct positive odd roots in $\Delta_{\st}^{+}$.
Only those weights $\Lambda-\gamma$ for which $L_0(\Lambda-\gamma)$ is
unitary and the corrected Dirac inequality is nonnegative can occur.
Among the weights $\mu=\Lambda-\gamma$ satisfying
\[
-(\Lambda+2\rho^{\st},\Lambda)
+(\mu+2\rho^{\st},\mu)
+2\sum_{i=p+1}^{m}\sum_{a=1}^{n}
(\mu,\epsilon_i-\delta_a)
=0,
\]
the weight $\Lambda'$ is the maximal one with respect to the order induced by
$\Delta_{\nst}^{+}$, that is, $\mu\leq_{\nst}\nu$ if
$\nu-\mu\in Q_{\nst}^{+}\coloneqq \sum_{\alpha \in \Delta_{\nst}^{+}}\ZZ_{\geq 0}\alpha$. Hence $\Lambda'$ can be recovered directly from
the possible $\even$-constituents of $L(\Lambda;\Delta_{\st}^{+})$.
\end{remark}

\subsection{The Dirac inequality for non-\texorpdfstring{$\omega$}{}-adapted positive systems}\label{subsec::corrected_Dirac}

Let $\gg$ be one of the Lie superalgebras $\sl(m\vert n)$, with
$m,n\geq1$, $\osp(m\vert 2n)$, with $m,n\geq1$, or $F(4)$.
In Proposition~\ref{prop::Dirac_inequality_adapted}, we established the
Dirac inequality for $\omega$-adapted positive systems, which was sharpened
in Corollary~\ref{cor::Dirac_inequality_on_M} to an inequality directly on
$M$. More precisely, if $M$ has highest weight $\Lambda$ with respect to
an $\omega$-adapted positive system of sign $\varepsilon$ and
$L_{0}(\mu)$ occurs in the restriction of $M$ to $\even$, then
\begin{equation}
-\varepsilon
\left(
(\mu+2\rho,\mu)
-
(\Lambda+2\rho,\Lambda)
\right)
\geq0.
\end{equation}
In particular, this inequality is strict whenever $\mu \neq \Lambda$. 

We now state the corresponding Dirac inequalities for the
non-$\omega$-adapted positive systems. They are obtained by transporting
the preceding inequality along the odd reflections relating the different
positive systems. More precisely, we use
Proposition~\ref{prop::transport_of_unitarity} together with the explicit
sequences of odd reflections described in
Section~\ref{subsec::structure_theory_slmn}. The resulting correction
terms are determined entirely by these odd reflections. Since the
argument is identical to the one carried out for $\slmn$ in
Section~\ref{subsec::Dirac_unitarity_slmn}, we omit the details.

For $\osp(m\vert 2n)$, recall that the relevant positive systems are
$\Delta_A^{+}$ and $\Delta_B^{+}$, which are related as described in
Lemma~\ref{lemm::relating_A_and_B_simple_system}. For the real forms $\osp(m,0\vert 2n)$ admitting nontrivial unitary highest weight representations, the Case~A positive system is $\omega$-adapted, whereas the Case~B positive system is not. For the real forms $\osp^{\ast}(2r\vert 0,2n)$, the $\omega$-adaptedness of the Case~A and Case~B positive systems is reversed. Recall that $r=\lfloor m/2\rfloor$. Transporting the Dirac
inequality from Case~A to Case~B gives the following result.

\begin{proposition}
Let $L(\Lambda;\Delta_A^{+})
\cong
 \Uppi^{\upepsilon}L(\Lambda';\Delta_B^{+})
$
be a unitary highest weight representation with respect to a
conjugate-linear anti-involution $\omega$, and assume that
$\Delta_A^{+}$ is $\omega$-adapted with sign $\varepsilon$.
Then, for every $\even$-highest weight $\mu$ occurring in the
$\even$-decomposition of $L(\Lambda';\Delta_B^{+})$, one has
\[
-\varepsilon(
-(\Lambda'+2\rho^{B},\Lambda')
+(\mu+2\rho^{B},\mu)
+2\sum_{i=1}^{r}\sum_{a=1}^{n}
(\mu,\epsilon_i-\delta_a))\geq 0.
\]
Moreover, the inequality is strict for all $\even$-constituents of highest weight $\mu \neq \Lambda$. 

Conversely, assume that $\Delta_B^{+}$ is $\omega$-adapted with sign $\varepsilon$. Then, for every $\even$-highest weight $\mu$ occurring in the $\even$-decomposition of $L(\Lambda;\Delta_A^{+})$, one has
\[
-\varepsilon\left(
-(\Lambda+2\rho^{A},\Lambda)
+(\mu+2\rho^{A},\mu)
-2\sum_{i=1}^{r}\sum_{a=1}^{n}
(\mu,\epsilon_i-\delta_a)
\right)\geq0.
\]
Moreover, the inequality is strict for all $\even$-constituents of highest weight $\mu\neq\Lambda'$.
\end{proposition}

Finally, we give the Dirac inequalities for $F(4)$. Recall that the relevant real forms are $F(4,0)$ and $F(4,2)$, associated with the conjugate-linear anti-involutions $\omega_{0}$ and $\omega_{2}$, respectively. Their corresponding $\omega$-adapted positive systems are $\Delta^{+}(\Uppi_{1})$ and $\Delta^{+}(\Uppi_{2})$, with signs $-1$ and $+1$, respectively, by Lemma~\ref{lemm::F4_adapted_positive_systems}. To keep track of the unitary $\gg$-equivalence between the different choices of positive systems, we denote by $\Lambda(\Uppi_i)$ the highest weight of the given unitary representation with respect to $\Delta^{+}(\Uppi_i)$, and by $\rho(\Uppi_i)$ the corresponding Weyl vector.

\begin{proposition}
Let
$
L\bigl(\Lambda(\Uppi_i);\Delta^{+}(\Uppi_i)\bigr)
$
be a unitary highest weight representation.

\begin{enumerate}
 \item[(a)] If the representation is unitary with respect to the real form $F(4,0)$,
then, for every $\even$-highest weight $\mu$ occurring in the $\even$-decomposition of $L\bigl(\Lambda(\Uppi_i);\Delta^{+}(\Uppi_i)\bigr)$, one has
\[
\begin{aligned}
0\leq\,
-(\Lambda(\Uppi_i)+2\rho(\Uppi_i),\Lambda(\Uppi_i))
+(\mu+2\rho(\Uppi_i),\mu)
-
\begin{cases}
0,
& i=1,\\[2mm]
4(\mu,-\epsilon_1+\delta),
& i=2,\\[2mm]
(\mu,-3\epsilon_1-\epsilon_2-\epsilon_3+3\delta),
& i=3,\\[2mm]
(\mu,-\epsilon_1-\epsilon_2-\epsilon_3+\delta),
& i=4,\\[2mm]
2(\mu,-\epsilon_1-\epsilon_2+\delta),
& i=5.
\end{cases}
\end{aligned}
\]
Moreover, the inequality is strict for all $\even$-constituents of highest weight $\mu \neq \Lambda(\Uppi_{1})$.
\item[(b)] If the representation is unitary with respect to the real form $F(4,2)$,
then, for every $\even$-highest weight $\mu$ occurring in the $\even$-decomposition of $L\bigl(\Lambda(\Uppi_i);\Delta^{+}(\Uppi_i)\bigr)$, one has
\[
\begin{aligned}
0\geq\,
-(\Lambda(\Uppi_i)+2\rho(\Uppi_i),\Lambda(\Uppi_i))
+(\mu+2\rho(\Uppi_i),\mu)
+
\begin{cases}
4(\mu,-\epsilon_1+\delta),
& i=1,\\[2mm]
0,
& i=2,\\[2mm]
(\mu,-\epsilon_1+\epsilon_2+\epsilon_3+\delta),
& i=3,\\[2mm]
(\mu,-3\epsilon_1+\epsilon_2+\epsilon_{3}+3\delta),
& i=4,\\[2mm]
2(\mu,-\epsilon_1+\epsilon_2+\delta),
& i=5.
\end{cases}
\end{aligned}
\]
Moreover, the inequality is strict for all $\even$-constituents of highest weight $\mu \neq \Lambda(\Uppi_{2})$.
\end{enumerate}
\end{proposition}

\section{Unitary branching}\label{sec::unitary_branching}\noindent
Let $\gg$ be either $\sl(m\vert n)$ with $m,n\geq1$, $\osp(m\vert 2n)$ with $m,n\geq1$ or $F(4)$. We solve the branching problem $\gg\downarrow \even$ of unitary highest weight representations. Throughout this section, we fix a real form of $\gg$ determined by a conjugate-linear anti-involution $\omega$ and, unless otherwise stated, an $\omega$-adapted positive system $\Delta^{+}$.

Let $\HH=L(\Lambda;\Delta^{+})$ be a unitary highest weight representation of highest weight $\Lambda\in\hh^{\ast}$. Recall that its weight set satisfies $\mathcal P_{\HH}\subseteq\Lambda-Q^{+}$, where
\begin{equation}
Q^{+}\coloneqq\sum_{\alpha\in\Delta^{+}}\ZZ_{\geq0}\alpha,
\qquad
\nu>\mu
\quad\Longleftrightarrow\quad
\nu-\mu\in Q^{+}\setminus\{0\}.
\end{equation}
As above, in the case $m=n$ all weight-theoretic statements for
$\sl(n\vert n)$ are understood after passing to $\gl(n\vert n)$.
Analogously, we denote by $Q_{\bar{0}}^{+}$ the positive root lattice of the even root system. We use this partial order throughout the section. By Corollary~\ref{corollary::even-filtration-simple}, the restriction of $\HH$ to $\even$ is a finite direct sum of unitary highest weight $\even$-representations,
\begin{equation}
\HH\big\vert_{\even}
\cong
\bigoplus_{\gamma\in\Gamma_{L}(\Lambda)}
L_0(\Lambda-\gamma)^{\oplus m_{\Lambda-\gamma}},
\end{equation}
for some finite subset $\Gamma_{L}(\Lambda)\subseteq\Gamma$, where $m_{\Lambda-\gamma}\in\ZZ_{\geq0}$ denotes the multiplicity of $L_0(\Lambda-\gamma)$. Recall that $\Gamma$ consists of weights of the form $\gamma=\sum_{\alpha\in S}\alpha+\sum_{\beta\in\Delta_{\bar0,nc}^{+}}n_\beta\beta$, where $S\subseteq\Delta_{\bar1}^{+}$ and $n_\beta\in\ZZ_{\geq0}$, while for $\gg=\sl(m\vert n)$ one may take $\Gamma=\{\sum_{\alpha\in S}\alpha:S\subseteq\Delta_{\bar1}^{+}\}$.

Unitarity imposes further necessary conditions on the elements of $\Gamma_{L}(\Lambda)$. Indeed, if $\gamma\in\Gamma_{L}(\Lambda)$, then $L_0(\Lambda-\gamma)$ is a unitary highest weight $\even$-representation, and Corollary~\ref{cor::Dirac_inequality_on_M} implies
\begin{equation}
-\varepsilon
\left(
(\Lambda-\gamma+2\rho,\Lambda-\gamma)
-
(\Lambda+2\rho,\Lambda)
\right)
>0
\end{equation}
whenever $\gamma\neq0$. Consequently, $\Gamma_{L}(\Lambda)
\subseteq \Gamma(\Lambda)$ with
\begin{equation}
\Gamma(\Lambda)\coloneqq
\left\{
\gamma\in\Gamma:
L_0(\Lambda-\gamma)\text{ is unitary and }
-\varepsilon
\left(
(\Lambda-\gamma+2\rho,\Lambda-\gamma)
-
(\Lambda+2\rho,\Lambda)
\right)>0
\text{ if }\gamma\neq0
\right\}.
\end{equation}

For $\gg \neq \sl(m\vert n)$, the set $\Gamma(\Lambda)$ is infinite. However, the finiteness argument of Corollary~\ref{corollary::even-filtration-simple} yields the following explicit criterion, which bounds the possible even highest weights occurring in the restriction and, in particular, provides a finite stopping rule for the computation of the branching multiplicities.

For $\gamma=\sum_{\alpha\in\Pi}n_{\alpha}\alpha\in Q^{+}$, we define its height by $\operatorname{ht}(\gamma)\coloneqq\sum_{\alpha\in\Pi}n_{\alpha}$, where $\Pi$ is the simple system corresponding to $\Delta^{+}$.

\begin{lemma}\label{lemm::finite_stopping_branching}
Let $\HH=L(\Lambda;\Delta^{+})$ be a unitary highest weight representation.
\begin{enumerate}
\item[(a)] If $m_{\Lambda-\gamma}\neq0$ for some $\gamma\in\Gamma(\Lambda)$, then there exists a sum $\eta$ of pairwise distinct positive odd roots such that $\eta-\gamma\in Q_{\bar0}^{+}$.

\item[(b)] Let $N\coloneqq\max\{\operatorname{ht}(\eta):\eta\text{ is a sum of pairwise distinct positive odd roots}\}$. If $m_{\Lambda-\gamma}\neq0$, then $\operatorname{ht}(\gamma)\leq N$.
\end{enumerate}
\end{lemma}

\begin{proof}
Choose nonzero root vectors $x_{-\alpha}\in\gg^{-\alpha}$ for
$\alpha\in\Delta_{\bar1}^{+}$. By the PBW theorem, the odd root vectors
occur with exponents at most one. Since the positive root spaces
annihilate the highest weight vector $v_{\Lambda}$, it follows that
$\HH$ is generated as a $\UE(\even)$-representation by the finitely many vectors
\[
x_{-\alpha_1}\cdots x_{-\alpha_r}v_{\Lambda},
\]
where $\alpha_1,\ldots,\alpha_r\in\Delta_{\bar1}^{+}$ are pairwise
distinct.

Assume that $m_{\Lambda-\gamma}\neq0$. Since $\HH$ is unitary, its
restriction to $\even$ is completely reducible. Hence there exists a
$\even$-equivariant projection
\[
\pi:\HH\longrightarrow L_0(\Lambda-\gamma)
\]
onto one copy of $L_0(\Lambda-\gamma)$. As the above vectors generate
$\HH$ over $\UE(\even)$, there exist pairwise distinct
$\alpha_1,\ldots,\alpha_r\in\Delta_{\bar1}^{+}$ such that
\[
\pi\bigl(x_{-\alpha_1}\cdots x_{-\alpha_r}v_{\Lambda}\bigr)\neq0.
\]
Set $\eta\coloneqq\alpha_1+\cdots+\alpha_r$.
The vector
$x_{-\alpha_1}\cdots x_{-\alpha_r}v_{\Lambda}$ has weight
$\Lambda-\eta$. Since $\pi$ preserves weights, $\Lambda-\eta$ is a
weight of the highest weight $\even$-representation $L_0(\Lambda-\gamma)$.
Therefore
\[
\Lambda-\eta
=
\Lambda-\gamma-\beta
\]
for some $\beta\in Q_{\bar0}^{+}$. Hence
$
\eta-\gamma=\beta\in Q_{\bar0}^{+},
$
which proves~(a).

For~(b), write $\eta-\gamma\in Q_{\bar0}^{+}\subseteq Q^{+}$. It follows
that
\[
\operatorname{ht}(\gamma)
\leq
\operatorname{ht}(\eta) \leq N
\]
by the definition of $N$. This proves (b).
\end{proof}

These conditions considerably restrict the possible constituents, but they do not determine whether a given candidate actually occurs in the simple representation.

\begin{remark}\label{rmk::Dirac_inequality_not_enough}
Strictness of the Dirac inequality for an individual $\even$-constituent of a Kac module does not imply that this constituent survives in the unique simple quotient. Thus, the strict Dirac inequality is first of all a necessary condition for a $\even$-constituent to occur in a unitary simple representation, rather than a sufficient condition for its occurrence.

Consider $\gg=\mathfrak{sl}(2\vert 1)$ with the standard positive system $\Delta_{\bar0}^{+}=\{\epsilon_1-\epsilon_2\}$ and $\Delta_{\bar1}^{+}=\{\epsilon_1-\delta_1,\epsilon_2-\delta_1\}$. The corresponding Weyl vector is $\rho=-\epsilon_2+\delta_1$. Let $\Lambda=(1,0\vert0)$ and consider the Kac module $K(\Lambda)\coloneqq K(L_{0}(\Lambda))$. Its restriction to $\even$ decomposes as
\begin{equation*}
K(1,0\vert0)\big\vert_{\even}
\cong
L_0(1,0\vert0)
\oplus
L_0(0,0\vert1)
\oplus
L_0(1,-1\vert1)
\oplus
L_0(0,-1\vert2).
\end{equation*}
For a $\even$-constituent $L_0(\Lambda-\gamma)$, the strict Dirac inequality is $2(\Lambda+\rho,\gamma)-(\gamma,\gamma)>0$. Since $\Lambda+\rho=\epsilon_1-\epsilon_2+\delta_1$, the three nontrivial constituents correspond to
\begin{equation*}
\begin{aligned}
L_0(0,0\vert1):\quad
&\gamma=\epsilon_1-\delta_1,
&&2(\Lambda+\rho,\gamma)-(\gamma,\gamma)=4,\\
L_0(1,-1\vert1):\quad
&\gamma=\epsilon_2-\delta_1,
&&2(\Lambda+\rho,\gamma)-(\gamma,\gamma)=0,\\
L_0(0,-1\vert2):\quad
&\gamma=\epsilon_1+\epsilon_2-2\delta_1,
&&2(\Lambda+\rho,\gamma)-(\gamma,\gamma)=6.
\end{aligned}
\end{equation*}
In particular, the Dirac inequality holds strictly on $L_0(0,-1\vert2)$. However, the unique simple quotient of $K(1,0\vert0)$ is isomorphic to the unitary standard representation $L(1,0\vert0)\cong\CC^{2\vert1}$, and (see also Example~\ref{ex::easiest_example})
\begin{equation*}
L(1,0\vert0)\big\vert_{\even}
\cong
L_0(1,0\vert0)\oplus L_0(0,0\vert1).
\end{equation*}
Hence $L_0(0,-1\vert2)$ does not occur in the unique simple quotient, although the corresponding Dirac inequality is strict.
\end{remark}

The unitary branching problem is therefore to determine the subset $\Gamma_{L}(\Lambda)\subseteq\Gamma$ and, for each $\gamma\in\Gamma_{L}(\Lambda)$, the multiplicity $m_{\Lambda-\gamma}$ of $L_0(\Lambda-\gamma)$ in $\HH\vert_{\even}$. Neither the filtration of a Verma or Kac module nor the Dirac inequalities on its $\even$-constituents determine these data, since they do not decide which of the possible constituents survive in the simple quotient. Our approach below instead works directly with the weight multiplicities of the simple representation $\HH$. A Freudenthal-type identity yields a recursive determination of the multiplicities $m_{\Lambda-\gamma}$ and, in particular, of the set $\Gamma_{L}(\Lambda)$, without requiring a separate analysis according to the degree of atypicality.

We begin with a Dirac-theoretic derivation of the super analogue of Freudenthal's identity. To the best of our knowledge, this identity has received very little attention in the Lie superalgebra literature and has previously been treated explicitly only in \cite{Kang_Oh, Kang_Oh_Freudenthal}. The point of the argument below is different: the identity arises directly from the square of the relative Dirac operator $\Dirac_{\gg,\even}$. In particular, the quadratic expression appearing on the left-hand side is precisely the quantity governing the Dirac inequality in the unitary setting. Thus the Freudenthal recursion and the Dirac inequality have a common Dirac-theoretic origin.

 Although our main interest lies in unitary representations, the following argument applies more generally to any highest weight representation $M$ admitting an $\omega$-contravariant Hermitian form $\langle\cdot,\cdot\rangle_M$. For any $\mu\in\hh^{\ast}$, we denote by $d(\mu)$ the dimension of the weight space $M^\mu$, where $M^{\mu}=\{0\}$ if $\mu \notin \mathcal{P}_{M}$. We set $d(\mu)=0$ if $\mu \notin \mathcal{P}_{M}$. The square of $\Dirac_{\gg,\even}$ then yields the following identity for highest weight representations of $\gg$.

\begin{proposition}\label{prop::Freudenthal_identity}
Let $M$ be a highest weight representation of highest weight $\Lambda$. Then, for every $\mu\in\mathcal{P}_{M}$,
\[
\bigl(-(\Lambda+2\rho,\Lambda)+(\mu+2\rho,\mu)\bigr)d(\mu)
=
-2\sum_{\alpha\in\Delta^{+}}\sum_{k\geq1}
(-1)^{(k+1)p(\alpha)}
(\mu+k\alpha,\alpha)d(\mu+k\alpha).
\]
\end{proposition}

\begin{proof} Choose a pseudo-orthonormal basis $\{v_a\}$ of the weight space $M^\mu$, so that
\[
\langle v_a,v_b\rangle_M
=
\varepsilon_a\delta_{ab},
\qquad
\varepsilon_a\in\{\pm1\}.
\] We evaluate
$
\varepsilon_a
\bigl\langle
\Dirac_{\gg,\even}^{2}(v_a\otimes1),
v_a\otimes1
\bigr\rangle_{M\otimes M(\odd)}
$
in two ways and sum over $a$. Note that for every endomorphism
$T\in\operatorname{End}(M^\mu)$, one has
$
\operatorname{tr}_{M^\mu}(T)
=
\sum_a
\varepsilon_a
\langle Tv_a,v_a\rangle_M.
$

By the Parthasarathy square formula and
 Remark~\ref{rmk::Explicit_form_D_square},
 \[
 \Dirac_{\gg,\even}^{2}
 =
 -\Omega_{\gg}\otimes1+\Omega_{\even,\Delta}+C
 =
 2\sum_{i,j=1}^{mn}
 \bigl(
 [\partial_i,\partial_j]\otimes x_ix_j
 +[x_i,x_j]\otimes\partial_i\partial_j
 -2[\partial_i,x_j]\otimes x_i\partial_j
 \bigr)
 -4\sum_{i=1}^{mn}x_i\partial_i\otimes1.
 \]
 Since $\partial_i1=0$ and the vacuum $1$ is orthogonal to every polynomial of positive degree with respect to the Bargmann--Fock form,
 \[
 \langle\Dirac_{\gg,\even}^{2}(v_a\otimes1),v_a\otimes1\rangle_{M\otimes M(\odd)}
 =
 -4\sum_{i=1}^{mn}\langle x_i\partial_i v_a,v_a\rangle_M.
 \]

 We now evaluate the left-hand side of the Parthasarathy formula. Choose root vectors $e_\alpha\in\gg_\alpha$ and $f_\alpha\in\gg^{-\alpha}$ for all $\alpha\in\Delta^+$. Normalize them by $B(e_\alpha,f_\alpha)=1$, such that we can identify $e_{\alpha_i}=\sqrt{2}\partial_i$ and $f_{\alpha_i}=\sqrt{2}x_i$. If $\{h_r\}$ and $\{h^r\}$ are dual bases of $\hh$, then
 \[
 \Omega_{\even}
 =
 \sum_r h_rh^r
 +
 \sum_{\alpha\in\Delta_{\bar0}^+}
 (e_\alpha f_\alpha+f_\alpha e_\alpha).
 \]
 Writing $\rho=\rho_{\bar0}-\rho_{\bar1}$, we have
 $-2\sum_i[\partial_i,x_i]v_a=-4(\rho_{\bar1},\mu)v_a$. Hence
 \[
 \begin{aligned}
\langle\Dirac_{\gg,\even}^{2}(v_a\otimes1),v_a\otimes1\rangle
 =
 2(-(\Lambda+2\rho,\Lambda)+(\mu+2\rho,\mu)) +2\sum_{\alpha\in\Delta_{\bar0}^+}
 \langle f_\alpha e_\alpha v_a,v_a\rangle_M.
 \end{aligned}
 \]
 Therefore, after summing over $a$ and including $\varepsilon_{a}$,
 \[
 \begin{aligned}
 2(-(\Lambda+2\rho,\Lambda)+(\mu+2\rho,\mu))d(\mu) =
 -2\sum_{\alpha\in\Delta^+}
 \tr_{M^\mu}(f_\alpha e_\alpha).
 \end{aligned}
 \]

For any $\alpha\in\Delta^+$, cyclicity of the trace and
$e_\alpha f_\alpha=[e_\alpha,f_\alpha]+(-1)^{p(\alpha)}f_\alpha e_\alpha$ give
\[
\begin{aligned}
\tr_{M^\mu}(f_\alpha e_\alpha)
&=\tr_{M^{\mu+\alpha}}(e_\alpha f_\alpha)\\
&=2(\mu+\alpha,\alpha)d(\mu+\alpha)
+(-1)^{p(\alpha)}\tr_{M^{\mu+\alpha}}(f_\alpha e_\alpha)\\
&=2\sum_{k\geq1}(-1)^{(k+1)p(\alpha)}
(\mu+k\alpha,\alpha)d(\mu+k\alpha).
\end{aligned}
\]
Here the iteration terminates since $M^{\mu+k\alpha}=0$ for suffiently large $k$ as $M$ is of highest weight type. Summing over all $\alpha \in \Delta^{+}$ yields the statement.
\end{proof}

For a fixed conjugate-linear anti-involution $\omega$ admitting nontrivial unitary highest weight representations, one may choose an $\omega$-adapted positive system (see Section~\ref{subsec:real_forms_slmb}). For such a choice, Corollary~\ref{cor::Dirac_inequality_on_M} shows that the relevant eigenvalue of $\Dirac_{\gg,\even}^{2}$ is strictly positive or strictly negative on every non-top $\even$-constituent. In particular, the coefficient of $d(\mu)$ in the preceding Freudenthal identity does not vanish when $\mu$ is the highest weight of such a constituent. We therefore obtain the following recursive formula for the dimensions of these weight spaces.

\begin{corollary}
Let $M=L(\Lambda;\Delta^{+})$ be a unitary highest weight representation defined with respect to an $\omega$-adapted positive system. Then, for every $\gamma\in\Gamma_{L}(\Lambda)\setminus\{0\}$,
\begin{equation*}
d(\Lambda-\gamma)
=
\frac{2}{
(\Lambda+2\rho,\Lambda)
-
(\Lambda-\gamma+2\rho,\Lambda-\gamma)
}
\sum_{\alpha\in\Delta^{+}}\sum_{k\geq1}
(-1)^{(k+1)p(\alpha)}
(\Lambda-\gamma+k\alpha,\alpha)
d(\Lambda-\gamma+k\alpha).
\end{equation*}
\end{corollary}

\begin{remark}
Clearly, $d(\Lambda)=\dim L(\Lambda)^{\Lambda}=1$.
\end{remark}

\begin{example} \label{ex::easiest_example}
We begin with the simplest example, namely $L(1,0\vert0)$ from Remark~\ref{rmk::Dirac_inequality_not_enough}, and illustrate how the preceding recursion is used. Recall that
\begin{equation*}
K(1,0\vert0)\big\vert_{\even}
\cong
L_0(1,0\vert0)
\oplus
L_0(0,0\vert1)
\oplus
L_0(1,-1\vert1)
\oplus
L_0(0,-1\vert2).
\end{equation*}
Hence, a priori, $\Gamma(\Lambda)\subseteq\{0,\epsilon_1-\delta_1,\epsilon_2-\delta_1,\epsilon_1+\epsilon_2-2\delta_1\}$. Since $L_0(1,-1\vert1)$ does not satisfy the strict Dirac inequality, it cannot occur in the unitary simple quotient, and therefore
\begin{equation*}
\Gamma(\Lambda)
\subseteq
\{0,\epsilon_1-\delta_1,\epsilon_1+\epsilon_2-2\delta_1\}.
\end{equation*}
Thus it remains to compute the dimensions of the weight spaces of weights $(0,0\vert1)$ and $(0,-1\vert2)$, that is, $d((0,0\vert1))$ and $d((0,-1\vert2))$.

Let $\mu=(0,0\vert1)=\delta_1$. Since $L(\Lambda)$ is a highest weight representation, $d(\mu+k\alpha)\neq0$ implies $\Lambda-\mu-k\alpha\in Q^+$. For $\alpha\in\Delta^+$, this leaves only $k=1$. Moreover, $\mu+(\epsilon_1-\epsilon_2)=(1,-1\vert1)$. This weight cannot occur in $L(\Lambda)$: by the above list of possible $\even$-constituents, the only constituent which could contain it is $L_0(1,-1\vert1)$, which has already been excluded by the strict Dirac inequality. Hence $d(1,-1\vert1)=0$. On the other hand,
\begin{equation*}
\mu+(\epsilon_1-\delta_1)=\Lambda,
\qquad
\mu+(\epsilon_2-\delta_1)=(0,1\vert0).
\end{equation*}
Clearly, $d(\Lambda)=1$, so it remains to determine $d((0,1\vert0))$. Set $\nu=(0,1\vert0)$. Applying the proposition to $\nu$, the only nonzero term on the right-hand side is $\nu+(\epsilon_1-\epsilon_2)=\Lambda$. Since
\begin{equation*}
-(\Lambda+2\rho,\Lambda)+(\nu+2\rho,\nu)=-2,
\qquad
(\Lambda,\epsilon_1-\epsilon_2)=1,
\end{equation*}
we obtain $d((0,1\vert0))=1$. Therefore, using the Dirac eigenvalue computed in Remark~\ref{rmk::Dirac_inequality_not_enough},
\begin{equation*}
-4d(\mu)
=
-2(\Lambda,\epsilon_1-\delta_1)d(\Lambda)
-2((0,1\vert0),\epsilon_2-\delta_1)d((0,1\vert0))
=
-4,
\end{equation*}
and hence $d((0,0\vert1))=1$.

It remains to determine $d((0,-1\vert2))$. Set $\mu=(0,-1\vert2)$. Since $\Lambda-\mu=(\epsilon_1-\delta_1)+(\epsilon_2-\delta_1)$, the only possible nonzero terms are $\mu+(\epsilon_2-\delta_1)=(0,0\vert1)$ and $\mu+2(\epsilon_2-\delta_1)=(0,1\vert0)$. Indeed, $\mu+(\epsilon_1-\epsilon_2)=\Lambda-2(\epsilon_2-\delta_1)$ is not a weight, while $\mu+(\epsilon_1-\delta_1)=\Lambda-(\epsilon_2-\delta_1)$ is not a weight of $L(\Lambda)$. Since
\[
-(\Lambda+2\rho,\Lambda)+(\mu+2\rho,\mu)=-6,
\]
the proposition, together with $d(0,0\vert1)=d(0,1\vert0)=1$, gives
\[
-6d(\mu)
=-2\bigl((\delta_1,\epsilon_2-\delta_1)-(\epsilon_2,\epsilon_2-\delta_1)\bigr)=0.
\]
Hence $d((0,-1\vert2))=0$.
\end{example}

\begin{example}\label{ex::oscillator_representation_dimensions}
We fix $p=q=1$ and the real form $\su(1,1\vert2)$ of $\sl(2\vert2)$. We consider the oscillator representation of $\sl(2\vert2)$ \cite{furutsu1991classification}, an infinite-dimensional unitary highest weight representation for this real form, with respect to the non-standard simple system $\Pi_{\nst}={\epsilon_1-\delta_1,\delta_1-\delta_2,\delta_2-\epsilon_2}$. The associated Weyl vector is
$
\rho
=
\left(
-\tfrac12,\tfrac12
\,\middle\vert\,
\tfrac12,-\tfrac12
\right).$
The oscillator representation
$L(\Lambda;\Delta_{\mathrm{nst}}^{+})$
has highest weight
\[
\Lambda
=
\left(
-\tfrac12,\tfrac12
\,\middle\vert\,
\tfrac12,-\tfrac12
\right).
\]
Thus
$
\Lambda+\rho=(-1,1\vert1,-1),
$ and $\Lambda$ is maximal atypical. Enumerate the positive odd roots as
\[
\alpha_1=\epsilon_1-\delta_1,\qquad
\alpha_2=\epsilon_1-\delta_2,\qquad
\alpha_3=\delta_1-\epsilon_2,\qquad
\alpha_4=\delta_2-\epsilon_2.
\]
Then every even constituent has a
highest weight $\Lambda-\sum_{j\in S}\alpha_j$ for some
$S\subseteq\{1,2,3,4\}$. We exclude a non-top candidate $\mu$ if it fails the strict Dirac
inequality (also referred to as Dirac gap) or if it is not a unitary highest weight for the even subalgebra.
For $\mu=(a,b\vert c,d)$ in this list, the noncompact condition is
automatic, since $a-b=-1-|S|<0$. The remaining even unitarity condition
is $c-d\in\mathbb Z_{\geq0}$.
The top weight is retained. We obtain 

\begin{center}
\small
\renewcommand{\arraystretch}{1.1}
\setlength{\tabcolsep}{5pt}
\begin{tabular}{|c|c|c|c|}
\hline
$S$
&
$\mu=\Lambda-\displaystyle\sum_{j\in S}\alpha_j$
&
Dirac gap
\\
\hline
$\varnothing$
& $(-\tfrac12,\tfrac12\vert\tfrac12,-\tfrac12)$
& $0$
\\
\hline
$\{2\}$
& $(-\tfrac32,\tfrac12\vert\tfrac12,\tfrac12)$
& $-4$
\\
\hline
$\{3\}$
& $(-\tfrac12,\tfrac32\vert-\tfrac12,-\tfrac12)$
& $-4$
\\
\hline
$\{1,2\}$
& $(-\tfrac52,\tfrac12\vert\tfrac32,\tfrac12)$
& $-6$
\\
\hline
$\{1,3\}$
& $(-\tfrac32,\tfrac32\vert\tfrac12,-\tfrac12)$
& $-6$
\\
\hline
$\{2,4\}$
& $(-\tfrac32,\tfrac32\vert\tfrac12,-\tfrac12)$
& $-6$
\\
\hline
$\{3,4\}$
& $(-\tfrac12,\tfrac52\vert-\tfrac12,-\tfrac32)$
& $-6$
\\
\hline
$\{1,2,3\}$
& $(-\tfrac52,\tfrac32\vert\tfrac12,\tfrac12)$
& $-12$
\\
\hline
$\{1,2,4\}$
& $(-\tfrac52,\tfrac32\vert\tfrac32,-\tfrac12)$
& $-8$
\\
\hline
$\{1,3,4\}$
& $(-\tfrac32,\tfrac52\vert\tfrac12,-\tfrac32)$
& $-8$
\\
\hline
$\{2,3,4\}$
& $(-\tfrac32,\tfrac52\vert-\tfrac12,-\tfrac12)$
& $-12$
\\
\hline
$\{1,2,3,4\}$
& $(-\tfrac52,\tfrac52\vert\tfrac12,-\tfrac12)$
& $-16$
\\
\hline
\end{tabular}
\end{center}

On the Cartan subalgebra of $\mathfrak{sl}(2\vert2)$, one has
$\alpha_1=\alpha_4$ and $\alpha_2=\alpha_3$. Therefore,
$\{2\}$ and $\{3\}$ give the same weight
$(-\tfrac32,\tfrac12\vert\tfrac12,\tfrac12)$,
the subsets $\{1,2\}$, $\{1,3\}$, $\{2,4\}$, and $\{3,4\}$ give the same weight
$(-\tfrac32,\tfrac32\vert\tfrac12,-\tfrac12)$,
the subsets $\{1,2,4\}$ and $\{1,3,4\}$ give the same weight
$(-\tfrac52,\tfrac32\vert\tfrac32,-\tfrac12)$,
and the subsets $\{1,2,3\}$ and $\{2,3,4\}$ give the same weight
$(-\tfrac52,\tfrac32\vert\tfrac12,\tfrac12)$.
Together with $\varnothing$ and $\{1,2,3,4\}$, these yield the six rows of the corrected table.

\begin{center}
\small
\setlength{\tabcolsep}{6pt}
\renewcommand{\arraystretch}{1.1}
\begin{tabular}{|c|c|c|}
\hline
$\mu$ & Dirac gap & $d(\mu)$ \\
\hline
$(-\tfrac12,\tfrac12\vert\tfrac12,-\tfrac12)$
& $0$ & $1$ \\
\hline
$(-\tfrac32,\tfrac12\vert\tfrac12,\tfrac12)$
& $-4$ & $2$ \\
\hline
$(-\tfrac32,\tfrac32\vert\tfrac12,-\tfrac12)$
& $-6$ & $1$  \\
\hline
$(-\tfrac52,\tfrac32\vert\tfrac32,-\tfrac12)$
& $-8$ & $0$  \\
\hline
$(-\tfrac52,\tfrac32\vert\tfrac12,\tfrac12)$
& $-12$ & $2$  \\
\hline
$(-\tfrac52,\tfrac52\vert\tfrac12,-\tfrac12)$
& $-16$ & $1$ \\
\hline
\end{tabular}
\end{center}
\end{example}

Set $a_{\nu,\mu}\coloneqq\dim L_0(\nu)^\mu$. If $L_0(\nu)$ is finite-dimensional, then $a_{\nu,\mu}$ is given by Kostant's multiplicity formula. More generally, for a possibly infinite-dimensional unitary highest weight representation $L_0(\nu)$, the weight multiplicities can be determined from the Enright--Willenbring character formula, which may be viewed as the corresponding extension to the unitary highest weight setting. Recall that $\kk^{\CC}$ denotes the complexification of a maximal compact subalgebra of $\even$, and recall the notation introduced in Section~\ref{subsec::maximal_compact_subalgebra}. Moreover, we denote by $P_{\bar0}$ the Kostant partition function for the positive even roots, which counts the number of ways a weight can be expressed as a nonnegative integral combination of roots in $\Delta_{\bar0}^{+}$.

\begin{lemma} Let $W_{\kk}$ denote the Weyl group of $\kk^{\CC}$, and let $l_{\kk}$ denote its length function. Moreover, let $\nu_{w}$ denote the unique
$\Delta_{c}^{+}$-dominant weight in the $W_{\kk}$-orbit of
$w(\nu+\rho_{\bar0})$. Then the multiplicity $a_{\nu,\mu}\coloneqq\dim L_0(\nu)^\mu$ is
\[
 a_{\nu,\mu}
 =
 \sum_{w\in W_\nu^{\kk}}
 \sum_{u\in W_{\kk}}
 (-1)^{\ell_\nu(w)+\ell_{\kk}(u)}
 P_{\bar{0}}
 \left(
 u(\nu_w+\rho_{c})
 -\rho_{c}-\mu
 \right).
\]
In particular, $a_{\nu,\nu}=1$.
\end{lemma}

\begin{proof}
 By the Enright--Willenbring character formula \cite{Enright_Willenbring}, we have 
\[
 \operatorname{ch}L_0(\nu)
 =
 \frac{
 \displaystyle
 \sum_{w\in W_\nu^{\kk}}
 (-1)^{\ell_\nu(w)}
 \operatorname{ch}F_{\kk}(\nu_w)
 }{
 \displaystyle
 \prod_{\alpha\in\Delta_{\mathrm{nc}}^+}
 (1-e^{-\alpha})
 }.
\]
The Weyl character formula for $\kk^{\CC}$ gives
\[
 \operatorname{ch}F_{\kk^{\CC}}(\nu_w)
 =
 \frac{
 \displaystyle
 \sum_{u\in W_{\kk}}
 (-1)^{\ell_{\kk}(u)}
 e^{u(\nu_w+\rho_{c})-\rho_{c}}
 }{
 \displaystyle
 \prod_{\alpha\in\Delta_{c}^+}
 (1-e^{-\alpha})
 }.
\]
Substituting this expression into the Enright--Willenbring character formula yields
\[
 \operatorname{ch}L_0(\nu)
 =
 \sum_{w\in W_\nu^{\kk}}
 \sum_{u\in W_{\kk}}
 (-1)^{\ell_\nu(w)+\ell_{\kk}(u)}
 e^{u(\nu_w+\rho_{c})-\rho_{c}}
 \prod_{\alpha\in\Delta_{\bar0}^+}
 (1-e^{-\alpha})^{-1}.
\]
Expanding the product using $P_{\bar{0}}$ and comparing
the coefficients of $e^\mu$ proves the assertion.
\end{proof}

In the following, we sum over the full candidate set $\Gamma(\Lambda)$. For $\gamma\in\Gamma(\Lambda)$, let $m_{\Lambda-\gamma}\coloneqq[\HH\vert_{\even}:L_0(\Lambda-\gamma)]\in\mathbb Z_{\geq0}$, where we allow $m_{\Lambda-\gamma}=0$. Thus $\HH\vert_{\even}\cong\bigoplus_{\gamma\in\Gamma(\Lambda)}L_0(\Lambda-\gamma)^{\oplus m_{\Lambda-\gamma}}$, with summands of multiplicity zero understood to be omitted. The actual branching support is
\[
\Gamma_L(\Lambda)
\coloneqq
\{\gamma\in\Gamma(\Lambda):m_{\Lambda-\gamma}\neq0\}.
\]
The super Freudenthal identity is initially formulated for
\(\mu\in\mathcal P_{\HH}\) and determines the corresponding weight
multiplicities \(d(\mu)=\dim \HH^\mu\). Recall that we extend this notation to all
$\gamma\in\Gamma(\Lambda)$ by declaring
$
d(\Lambda-\gamma)\coloneqq0$ whenever $\Lambda-\gamma\notin\mathcal P_{\HH}$.
With this convention, the super Freudenthal recursion may be applied
uniformly to every candidate weight \(\Lambda-\gamma\), with
\(\gamma\in\Gamma(\Lambda)\).

We obtain the following recursive formula for the branching multiplicities, which works uniformly for arbitrary atypicality.

\begin{theorem}\label{thm::branching_recursion}
Let $\HH=L(\Lambda;\Delta^+)$ be unitary with respect to an $\omega$-adapted positive system, and write $\HH\vert_{\even}\cong\bigoplus_{\mu\in\Lambda-\Gamma(\Lambda)}L_0(\mu)^{\oplus m_\mu}$. Then $m_\Lambda=1$, and for every $\mu\in(\Lambda-\Gamma(\Lambda))\setminus\{\Lambda\}$,
\begin{equation*}
m_\mu
=
d(\mu)-\sum_{\substack{\nu\in\Lambda-\Gamma(\Lambda)\\ \nu>\mu}}m_\nu a_{\nu,\mu}
=
\sum_{\substack{\nu\in\Lambda-\Gamma(\Lambda)\\ \nu>\mu}}
B_{\mu,\nu}(\Lambda)m_\nu,
\end{equation*}
where
\begin{equation*}
B_{\mu,\nu}(\Lambda)
=
\frac{2\displaystyle\sum_{\alpha\in\Delta^+}\sum_{k\geq1}
(-1)^{(k+1)p(\alpha)}
(\mu+k\alpha,\alpha)a_{\nu,\mu+k\alpha}}
{(\Lambda+2\rho,\Lambda)-(\mu+2\rho,\mu)}
-a_{\nu,\mu}.
\end{equation*}
\end{theorem}

\begin{proof}
The highest weight space $\HH_{\Lambda}$ is one-dimensional. Since the restriction of $\HH$ to $\even$ contains the $\even$-highest weight representation generated by a highest weight vector of $\HH$, it follows that $m_{\Lambda}=1$.

For any weight $\eta\in \mathcal{P}_{\HH}$, the decomposition
$
\HH\big\vert_{\even}
\cong
\bigoplus_{\nu\in\Lambda-\Gamma(\Lambda)}
L_0(\nu)^{\oplus m_\nu}
$
gives
\begin{equation}\label{eq::weight_multiplicity_branching}
d(\eta)
=
\sum_{\nu\in\Lambda-\Gamma(\Lambda)}
m_\nu a_{\nu,\eta}.
\end{equation}
If $a_{\nu,\eta}\neq0$, then $\eta$ is a weight of $L_0(\nu)$ and hence $\nu-\eta$ is a nonnegative integral combination of positive even roots. In particular, $\nu\geq\eta$. Taking $\eta=\mu$ and using $a_{\mu,\mu}=1$, we therefore obtain
\[
d(\mu)
=
m_\mu+
\sum_{\substack{\nu\in\Lambda-\Gamma(\Lambda)\\ \nu>\mu}}
m_\nu a_{\nu,\mu},
\]
and consequently
\begin{equation}\label{eq::branching_triangular_recursion}
m_\mu
=
d(\mu)-
\sum_{\substack{\nu\in\Lambda-\Gamma(\Lambda)\\ \nu>\mu}}
m_\nu a_{\nu,\mu}.
\end{equation}

Now let $\mu\neq\Lambda$. By the strict Dirac inequality,
$
(\Lambda+2\rho,\Lambda)-(\mu+2\rho,\mu)\gtrless 0,
$
so the denominator below is nonzero. Applying the preceding proposition gives
\[
d(\mu)
=
\frac{2}{(\Lambda+2\rho,\Lambda)-(\mu+2\rho,\mu)}
\sum_{\alpha\in\Delta^+}\sum_{k\geq1}
(-1)^{(k+1)p(\alpha)}
(\mu+k\alpha,\alpha)d(\mu+k\alpha).
\]

For every $\alpha\in\Delta^+$ and $k\geq1$, \eqref{eq::weight_multiplicity_branching} yields
\[
d(\mu+k\alpha)
=
\sum_{\nu\in\Lambda-\Gamma(\Lambda)}
m_\nu a_{\nu,\mu+k\alpha}= \sum_{\substack{\nu\in\Lambda-\Gamma(\Lambda)\\ \nu>\mu}}
m_\nu a_{\nu,\mu+k\alpha}.
\]
Substituting this into the preceding formula and interchanging the finite sums gives
\[
d(\mu)
=
\sum_{\substack{\nu\in\Lambda-\Gamma(\Lambda)\\ \nu>\mu}}
\frac{2\displaystyle\sum_{\alpha\in\Delta^+}\sum_{k\geq1}
(-1)^{(k+1)p(\alpha)}
(\mu+k\alpha,\alpha)a_{\nu,\mu+k\alpha}}
{(\Lambda+2\rho,\Lambda)-(\mu+2\rho,\mu)}
\,m_\nu.
\]
Combining this with \eqref{eq::branching_triangular_recursion}, we obtain
\[
m_\mu
=
\sum_{\substack{\nu\in\Lambda-\Gamma(\Lambda)\\ \nu>\mu}}
\left(
\frac{2\displaystyle\sum_{\alpha\in\Delta^+}\sum_{k\geq1}
(-1)^{(k+1)p(\alpha)}
(\mu+k\alpha,\alpha)a_{\nu,\mu+k\alpha}}
{(\Lambda+2\rho,\Lambda)-(\mu+2\rho,\mu)}
-a_{\nu,\mu}
\right)m_\nu,
\]
which is precisely
$
m_\mu
=
\sum_{\substack{\nu\in\Lambda-\Gamma(\Lambda)\\ \nu>\mu}}
B_{\mu,\nu}(\Lambda)m_\nu$.
\end{proof}

\begin{example}
Using the weight multiplicities computed for the unitary highest weight
representations $L((1,0\vert 0);\Delta_{\st}^{+})$ of $\sl(2\vert 1)$ in
Example~\ref{ex::easiest_example} and
$L((-\tfrac12,\tfrac12\vert\tfrac12,-\tfrac12);\Delta_{\nst}^{+})$ of
$\sl(2\vert 2)$ in
Example~\ref{ex::oscillator_representation_dimensions}, the recursive
branching formula yields
\[
L((1,0\vert 0);\Delta_{\st}^{+})\vert_{\even}
\cong
L_{0}(1,0\vert 0)\oplus L_{0}(0,0\vert 1),
\]
and
\[
L((-\tfrac12,\tfrac12\vert\tfrac12,-\tfrac12);\Delta_{\nst}^{+})\vert_{\even}
\cong
L_{0}(-\tfrac12,\tfrac12\vert\tfrac12,-\tfrac12)
\oplus
L_{0}(-\tfrac32,\tfrac12\vert\tfrac12,\tfrac12)^{\oplus 2}.
\]
\end{example}

\begin{corollary}\label{cor::branching_support}
Let $\HH=L(\Lambda;\Delta^+)$ be unitary with respect to an $\omega$-adapted positive system. Then
\[
\Gamma_L(\Lambda)
=
\{0\}
\cup
\{
\gamma\in\Gamma(\Lambda)\setminus\{0\}:
d(\Lambda-\gamma)
\neq
\sum_{\substack{\nu\in\Lambda-\Gamma(\Lambda)\\
\nu>\Lambda-\gamma}}
m_\nu a_{\nu,\Lambda-\gamma}
\}.
\]
Equivalently, for $\gamma\in\Gamma(\Lambda)\setminus\{0\}$,
\[
\gamma\in\Gamma_L(\Lambda)
\quad\Longleftrightarrow\quad
\sum_{\substack{\nu\in\Lambda-\Gamma(\Lambda)\\
\nu>\Lambda-\gamma}}
B_{\Lambda-\gamma,\nu}(\Lambda)m_\nu\neq0.
\]
\end{corollary}

\begin{corollary}
 Let $\HH=L(\Lambda;\Delta^{+})$ be a unitary highest weight representation. Then for all $\gamma \in \Gamma_{L}(\Lambda)$ one has
 \[
 0\leq  m_{\Lambda-\gamma} \leq d(\Lambda-\gamma).
 \]
\end{corollary}
\begin{proof}
 By Proposition~\ref{prop::transport_of_unitarity}, we can assume we work with an $\omega$-adapted positive system. The triangular recursion and the nonnegativity of all branching multiplicities $m_\nu$ and even weight multiplicities $a_{\nu,\mu}$ imply
\[
m_\mu
=
d(\mu)
-
\sum_{\substack{\nu\in \Lambda-\Gamma(\Lambda)\\ \nu>\mu}}
m_\nu a_{\nu,\mu}
\leq d(\mu).
\]
Since $m_\mu\geq0$, we obtain $0\leq m_\mu\leq d(\mu)$. 
\end{proof}

We next derive a closed expression for the branching multiplicities by iterating the preceding recursion along chains in the partially ordered set $\Lambda-\Gamma(\Lambda)$.

\begin{theorem}\label{thm::closed_formula}
Let $\HH=L(\Lambda;\Delta^+)$ be unitary with respect to an $\omega$-adapted positive system. Then, for every
$\mu\in(\Lambda-\Gamma(\Lambda))\setminus\{\Lambda\}$,
\[
m_\mu = 
\sum_{r\geq0}(-1)^r
\sum_{\substack{
\mu=\mu_0<\mu_1<\cdots<\mu_r\\
\mu_i\in \Lambda-\Gamma(\Lambda)
}}
d(\mu_r)\prod_{j=1}^r a_{\mu_j,\mu_{j-1}}
=
\sum_{r\geq1}
\ \sum_{\substack{
\mu=\mu_0<\mu_1<\cdots<\mu_r=\Lambda\\
\mu_i\in\Lambda-\Gamma(\Lambda)
}}
\prod_{j=0}^{r-1}
B_{\mu_j,\mu_{j+1}}(\Lambda).
\]
\end{theorem}

\begin{proof} The first equality follows from the triangular recursion
\[
m_\mu
=
d(\mu)
-
\sum_{\substack{\nu\in \Lambda-\Gamma(\Lambda)\\ \nu>\mu}}
m_\nu a_{\nu,\mu}.
\]
Substituting the same recursion for each $m_{\mu_1}$ occurring on the right-hand side gives
\[
\begin{aligned}
m_\mu
&=
d(\mu)
-
\sum_{\substack{\mu_1\in \Lambda-\Gamma(\Lambda)\\ \mu_1>\mu}}
a_{\mu_1,\mu}
\left(
d(\mu_1)
-
\sum_{\substack{\mu_2\in \Lambda-\Gamma(\Lambda)\\ \mu_2>\mu_1}}
m_{\mu_2}a_{\mu_2,\mu_1}
\right)\\
&=
d(\mu)
-
\sum_{\substack{\mu<\mu_1\\ \mu_1\in \Lambda-\Gamma(\Lambda)}}
d(\mu_1)a_{\mu_1,\mu}
+
\sum_{\substack{\mu<\mu_1<\mu_2\\ \mu_1,\mu_2\in \Lambda-\Gamma(\Lambda)}}
m_{\mu_2}a_{\mu_2,\mu_1}a_{\mu_1,\mu}.
\end{aligned}
\]
Iterating this substitution, each step adds one further strict inequality in a chain $\mu=\mu_0<\mu_1<\cdots<\mu_r$, one additional factor $a_{\mu_j,\mu_{j-1}}$, and changes the sign. For fixed $\mu$, the set
$
\{\nu\in\Lambda-\Gamma(\Lambda):\mu\leq\nu\leq\Lambda\}
$
is finite. Indeed, writing $\nu=\Lambda-\gamma$ with $\gamma\in Q^{+}$, the condition
$\mu\leq\nu\leq\Lambda$ implies
$
0\leq\gamma\leq\Lambda-\mu.
$
Since $Q^{+}$ is generated by the finitely many simple roots, there are only finitely many such
$\gamma$. Hence there are no strictly increasing chains
\[
\mu=\mu_{0}<\mu_{1}<\cdots<\mu_{r}=\Lambda
\]
of arbitrarily large length, and the iteration terminates after finitely many steps. Hence there are no strictly increasing chains $\mu=\mu_0<\mu_1<\cdots<\mu_r=\Lambda$ of arbitrarily large length, and the iterated substitution terminates after finitely many steps. This yields
\[
m_\mu
=
\sum_{r\geq0}(-1)^r
\sum_{\substack{
\mu=\mu_0<\mu_1<\cdots<\mu_r\\
\mu_i\in \Lambda-\Gamma(\Lambda)
}}
d(\mu_r)\prod_{j=1}^r a_{\mu_j,\mu_{j-1}}.
\]
For the second equality, we start from 
\[
m_\mu
=
\sum_{\substack{\nu\in\Lambda-\Gamma(\Lambda)\\ \nu>\mu}}
B_{\mu,\nu}(\Lambda)m_\nu.
\]
We repeatedly substitute the same recursion for each $m_\nu$ with $\nu\neq\Lambda$.
Since $\{\nu\in\Lambda-\Gamma(\Lambda):\mu\leq\nu\leq\Lambda\}$ is finite and every substitution strictly increases the weight, the process terminates after finitely many steps. Every resulting term is indexed by a strictly increasing chain
\[
\mu=\mu_0<\mu_1<\cdots<\mu_r=\Lambda
\]
and contributes
\[
B_{\mu_0,\mu_1}(\Lambda)\cdots
B_{\mu_{r-1},\mu_r}(\Lambda)m_\Lambda.
\]
Since $m_\Lambda=1$, summing over all such chains gives the asserted formula.
\end{proof}

We formulated the unitary branching law $\gg\downarrow\even$ for
$\omega$-adapted positive systems. By means of odd reflections and
Proposition~\ref{prop::transport_of_unitarity}, this result can be
transported by a finite procedure to all positive systems relevant to
unitarity. Together, these results solve the unitary branching problem.
Alternatively, for non-adapted positive systems one may work directly
with the corrected Dirac operators of
Section~\ref{subsec::Dirac_unitarity_slmn} and
Section~\ref{subsec::corrected_Dirac}. The crucial ingredient is again
a recursive formula for the dimensions of weight spaces.

We illustrate this for $\slmn$ with the standard positive system.
Analogously to Proposition~\ref{prop::Freudenthal_identity}, one obtains
\begin{equation}
\begin{aligned}
&\Biggl(
-(\Lambda+2\rho^{\st},\Lambda)
+(\mu+2\rho^{\st},\mu)
+2\sum_{i=p+1}^{m}\sum_{a=1}^{n}
(\mu,\epsilon_i-\delta_a)
\Biggr)d(\mu)
\\
&\qquad=
2\sum_{\alpha\in\Delta_{\bar0}^{+}}\sum_{k\ge1}
(\mu+k\alpha,\alpha)d(\mu+k\alpha)
\\
&\qquad\quad
+2\sum_{i=1}^{p}\sum_{a=1}^{n}\sum_{k\ge1}
(-1)^{k-1}
\bigl(\mu+k(\epsilon_i-\delta_a),\epsilon_i-\delta_a\bigr)
d\bigl(\mu+k(\epsilon_i-\delta_a)\bigr)
\\
&\qquad\quad
-2\sum_{i=p+1}^{m}\sum_{a=1}^{n}\sum_{k\ge1}
(-1)^{k-1}
\bigl(\mu+k(\epsilon_i-\delta_a),\epsilon_i-\delta_a\bigr)
d\bigl(\mu+k(\epsilon_i-\delta_a)\bigr).
\end{aligned}
\end{equation}
The coefficient
\begin{equation}\label{eq::corrected_Dirac_inequality_slmn}
-(\Lambda+2\rho^{\st},\Lambda)
+(\mu+2\rho^{\st},\mu)
+2\sum_{i=p+1}^{m}\sum_{a=1}^{n}
(\mu,\epsilon_i-\delta_a)
\end{equation}
is the scalar arising from the corrected Dirac inequality. By
Theorem~\ref{thm::unitarity_standard_case} and Corollary~\ref{cor::finding_Lambda'}, it is strictly positive
for every $\even$-highest weight $\mu\neq\Lambda'$ occurring in
$L(\Lambda;\Delta_{\st}^{+})$, and vanishes at the unique weight $\Lambda'$
such that
$L(\Lambda;\Delta_{\st}^{+})\cong
 \Uppi^{\upepsilon}L(\Lambda';\Delta_{\nst}^{+})$.
For this exceptional weight, $d(\Lambda')=1$, and $\Lambda'$ can be
determined directly by Remark~\ref{rmk::determination_Lambda'}.
Consequently, the corrected Freudenthal identity determines the dimensions
of the weight spaces corresponding to all $\even$-highest weights occurring
in $L(\Lambda;\Delta_{\st}^{+})$, and hence yields, exactly as above,
recursive and closed formulas for the branching multiplicities.

\bibliography{literatur}
\bibliographystyle{alpha}
\end{document}